\documentclass{mhkr}
\usepackage{algpseudocode}
\usetikzlibrary{cd}
\usepackage{forest}

\usepackage{lineno}

\usepackage[style = numeric, sorting = nyt]{biblatex}
\title{Martin's measurable dilator}
\author{Hanul Jeon}
\email{ \href{mailto:hj344@cornell.edu}{hj344@cornell.edu}}
\urladdr{ \href{https://hanuljeon95.github.io}{https://hanuljeon95.github.io} }
\address{Department of Mathematics, Cornell University, Ithaca, NY 14853}
\thanks{The research presented in this paper is supported in part by NSF grant DMS–2153975.}

\newcommand{\lag}{\langle}
\newcommand{\rag}{\rangle}
\newcommand{\lr}{\leftrightarrow}

\newcommand{\crit}{\operatorname{crit}}
\newcommand{\Ult}{\operatorname{Ult}}

\newcommand{\rank}{\operatorname{rank}}
\newcommand{\field}{\operatorname{field}}

\newcommand{\arity}{\operatorname{arity}}
\newcommand{\pred}{\operatorname{pred}}
\newcommand{\supp}{\operatorname{supp}}

\newcommand{\Diag}{\operatorname{Diag}}

\newcommand{\Cell}{\operatorname{Cell}}
\newcommand{\Dec}{\operatorname{Dec}}

\newcommand{\lh}{\operatorname{lh}}
\newcommand{\term}{\operatorname{term}}
\newcommand{\Ord}{\mathrm{Ord}}

\newcommand{\WO}{\mathsf{WO}}
\newcommand{\LO}{\mathsf{LO}}

\newcommand{\Dil}{\mathsf{Dil}}
\newcommand{\SDil}{\mathsf{SDil}}

\newcommand{\KB}{\mathsf{KB}}

\newcommand{\LV}{\mathsf{LV}}

\newcommand{\Emb}{\mathsf{Emb}}

\newcommand{\RFN}[1][]{\ifthenelse{\equal{#1}{}}{}{#1\mhyphen}\mathsf{RFN}}

\newcommand{\ZF}{\mathsf{ZF}}

\newcommand{\ZFC}{\mathsf{ZFC}}

\newcommand{\en}{\operatorname{en}}

\newcommand{\Tr}{\operatorname{Tr}}

\newcommand{\hmu}{\hat{\mu}}
\newcommand{\hnu}{\hat{\nu}}

\def\cyrillic{\usefont{T2A}{antt}{m}{n}}

\newcommand{\cyrDe}{{\text{\cyrillic \char"C4}}}

\def\bfcyrillic{\usefont{T2A}{antt}{eb}{n}}

\newcommand{\bfcyrDe}{{\text{\bfcyrillic \char"C4}}}

\begin{document}
\maketitle
\begin{abstract}
    Martin's remarkable proof \cite{Martin1980InfiniteGames} of $\bfPi^1_2$-determinacy from an iterable rank-into-rank embedding highlighted the connection between large cardinals and determinacy.
    In this paper, we isolate a large cardinal object called a \emph{measurable dilator} from Martin's proof of $\bfPi^1_2$-determinacy, which captures the structural essence of Martin's proof of $\bfPi^1_2$-determinacy.
\end{abstract}

\section{Introduction}

Stanis\l aw Mazur raised a game-theoretic problem in the \emph{Scottish book} \cite[Problem 43]{TheScottishBook}, whose following generalization is formulated by Stanis\l aw Ulam:
\begin{quote}
    Given a set $E$ of reals, Player I and II give in turn the digits 0 or 1. If the resulting real is in $E$, then Player I wins, and Player II wins otherwise. For which $E$ does one of the players have a winning strategy?
\end{quote}
Under the standard set-theoretic tradition, we identify a real with an infinite sequence of natural numbers, so we may think of $E$ as a set of infinite binary sequences, which is called a \emph{payoff set}.
In 1953, David Gale and Frank M. Stewart \cite{GaleStewart1953InfiniteGames} studied a generalization of Ulam's problem by considering an infinite game over an arbitrary set $A$. Gale and Stewart proved that if a payoff set $E\subseteq A^\omega$ is \emph{open or closed}, then either one of the players has a winning strategy in the corresponding game.
They also showed from the axiom of choice that there is a payoff set $E\subseteq 2^\omega$ in which no players have a winning strategy.
It raises the following question: \emph{Does one of the players have a winning strategy for a `reasonably definable' payoff set?}

For a class $\Gamma$ of sets of reals, \emph{$\Gamma$-determinacy} is the assertion that for a payoff set $E\in \Gamma$, the game given by $E$ is \emph{determined} in the sense that either one of the players has a winning strategy.
After some partial results from other mathematicians, Martin \cite{Martin1975BorelDeterminacy} proved Borel determinacy. Then what can we say about determinacy for larger classes? It turns out that $\Gamma$-determinacy for a larger $\Gamma$ is closely related to large cardinal axioms.
In 1968, Martin \cite{Martin1969MeasurableAnalyticDet} proved the $\bfPi^1_1$-determinacy from a measurable cardinal. Later in 1978, Harrington proved that if $\bfPi^1_1$-determinacy holds, then every real has a sharp \cite{Harrington1978AnalyticDetandSharp}. Indeed, $\bfPi^1_1$-determinacy and the existence of sharps for reals are equivalent.

For a class larger than $\bfPi^1_1$, there was no significant progress for years. But in 1980, Martin \cite{Martin1980InfiniteGames} proved the $\bfPi^1_2$-determinacy from a rank-into-rank large cardinal called an \emph{iterable cardinal}.
A proof of $\bfPi^1_2$-determinacy and projective determinacy (determinacy for $\bfPi^1_n$-sets for every natural $n$) from a near-optimal hypothesis appeared in 1989 by Martin and Steel \cite{MartinSteel1989ProofPD}; Namely, we have $\bfPi^1_n$-determinacy from $(n-1)$ many Woodin cardinals and a measurable above.
The optimal strength of $\bfPi^1_n$-determinacy requires $M_n^\sharp(x)$, a sharp for a canonical inner model with $n$ many Woodin cardinals.

From a completely different side, Girard developed a notion of \emph{dilator} for his $\Pi^1_2$-logic.
To motivate Girard's $\Pi^1_2$-logic, let us briefly review ordinal analysis: Ordinal analysis gauges the strength of a theory $T$ by looking at its \emph{proof-theoretic ordinal}
\begin{equation*}
    |T|_{\Pi^1_1} = \sup \{\alpha\mid \text{$\alpha$ is recursive and } T\vdash\text{$\alpha$ is well-ordered}\}.
\end{equation*}
$|T|_{\Pi^1_1}$ gauges the $\Pi^1_1$-consequences of a theory in some sense; One of the main reasons comes from Kleene normal form theorem, stating that for every $\Pi^1_1$-statement $\phi(X)$, we can find an $X$-recursive linear order $\alpha(X)$ such that $\phi(X)$ holds iff $\alpha(X)$ is a well-order.%
\footnote{See \autoref{Lemma: A continuous family of finite linear orders} for its refined version. For a more discussion between the proof-theoretic ordinal and the $\Pi^1_1$-consequences of a theory, see \cite{Walsh2023characterizations} or \cite[\S 1]{Jeon2024HigherProofTheoryI}.} Girard wanted to analyze $\Pi^1_2$-consequences of a theory, requiring an object corresponding to $\Pi^1_2$-statements like well-orders correspond to $\Pi^1_1$-statements.

One way to explain a dilator is by viewing it as a representation of a class ordinal: There is no transitive class isomorphic to $\Ord+\Ord$ or $\Ord^2$, but we can still express their ordertype.
In the case of $\Ord+\Ord$, we can think of it as the collection of $(i,\xi)$ for $i=0,1$ and $\xi\in\Ord$, and compare them under the lexicographic order.
Interestingly, the same construction gives not only the ordertype $\Ord+\Ord$, but also that of $X+X$ for every linear order $X$:
That is, $X+X$ is isomorphic to the collection of $(i,\xi)$ for $i=0,1$ and $\xi\in X$ endowed with the lexicographic order.
The uniform construction $X\mapsto X+X$ is an example of a dilator.

It turns out that \emph{dilators} correspond to $\Pi^1_2$-statements: 
A \emph{semidilator} is an autofunctor over the category of linear orders preserving direct limits and pullbacks, and a \emph{dilator} is a semidilator preserving well-orderedness; That is, a semidilator $D$ is a dilator if $D(X)$ is a well-order for every well-order $X$.
Semidilators and dilators look gigantic, but it is known that we can recover the full (semi)dilator from its restriction over the category of natural numbers with strictly increasing maps so that we can code them as a set. We can also talk about how a given (semi)dilator is recursive by saying there is a recursive code for the restriction of a (semi)dilator to the category of natural numbers.%
\footnote{However, we will not use the definition of (semi)dilators as functors preserving direct limits and pullbacks. See \autoref{Section: Dilators} for a precise definition.}
Like well-orders are associated with $\Pi^1_1$-statements, dilators are associated with $\Pi^1_2$-statements: Girard proved that for a given $\Pi^1_2$-statement $\phi(X)$, we could find an $X$-recursive predilator $D(X)$ such that $\phi(X)$ holds iff $D(X)$ is a dilator. (See \autoref{Lemma: A continuous family of finite dilators} for its refined version.)
Girard pointed out the connection between dilators and descriptive set theory in \cite[\S 9]{Girard1985IntroPi12Logic}, and Kechris \cite{KechrisUnpublishedDilators} examined a connection between dilators, \emph{ptykes}\footnote{Ptykes (sing. \emph{ptyx}) is a generalization of a dilator corresponding to $\Pi^1_n$-formulas. We will not introduce its definition since this paper will not use general ptykes. See \cite{Girard1982Logical, GirardRessayre1985} for more details about ptykes.} and descriptive set theory. Kechris introduced a notion of \emph{measurable dilator} as a dilator version of a measurable cardinal and stated that the existence of a measurable dilator implies $\bfPi^1_2$-determinacy.

Going back to the determinacy side, a proof of $\bfPi^1_1$-determinacy from a measurable cardinal shows a curious aspect that most of its proofs use a well-order characterization of a $\bfPi^1_1$-statement in any form: See \autoref{Subsection: Measurable cardinal and Pi 1 1 Det} for the proof of $\bfPi^1_1$-determinacy from a measurable cardinal;
Many proofs of the $\bfPi^1_1$-determinacy from a large cardinal axiom use the fact that $\bfPi^1_1$-sets are $\kappa$-Suslin for an uncountable regular cardinal $\kappa$, and choose a large $\kappa$ so we get a homogeneously Suslin tree representation.
However, these proofs implicitly use a well-order characterization of $\bfPi^1_1$-sentences. Every proof of the Susliness of a $\bfPi^1_1$ set the author knows goes as follows: Start from a well-order representation $\alpha$ of a $\bfPi^1_1$-set (usually taking the form of a tree over $\omega\times\omega$, which is a linear order under the Kleene-Brouwer order), and construct a predilator $D$ trying to construct an embedding from $\alpha$ to $\kappa$.
$D(\kappa)$ corresponds to the $\kappa$-Suslin representation of the $\bfPi^1_1$-set, and we may think of $D$ as an `effective part' and $\kappa$ a `large cardinal part' of the Suslin representation.
This type of idea is implicit in the proof of \autoref{Lemma: A continuous family of finite dilators}.

We may ask if a proof of $\bfPi^1_2$-determinacy from a large cardinal assumption also uses a dilator characterization of a $\bfPi^1_2$-statement.
That is, we can ask if we can decompose a proof of $\bfPi^1_2$-determinacy into the following two steps:
\begin{enumerate}
    \item Starting from a large cardinal assumption, construct a measurable dilator.
    \item From a measurable dilator, prove $\bfPi^1_2$-determinacy.
\end{enumerate}
We will illustrate in \autoref{Subsection: Measurable dilator and Pi 1 2 Det} that the second step indeed holds. The main goal of this paper is to extract a construction of a measurable dilator from Martin's proof \cite{Martin1980InfiniteGames} of $\bfPi^1_2$-determinacy from an iterable cardinal, thus establishing the first step.

Suppose a rank-into-rank embedding $j\colon V_\lambda\to V_\lambda$ with $\kappa=\crit j$ iterable such that $\lambda = \sup_{n<\omega} j^n(\kappa)$.
Martin \cite[\S 4]{Martin1980InfiniteGames} used $\lambda$-Suslin tree structure for a $\bfPi^1_2$-set to prove $\bfPi^1_2$-determinacy. Martin also used a measure family given by an iteration of measures along a $\lambda$-Suslin tree.
To extract dilator-related information from Martin's proof, we need a tree-like structure of a dilator. Girard \cite[\S6]{Girard1981Dilators} presented a notion of \emph{dendroid}, expressing a dilator as a functorial family of trees. Dendroids themselves are not enough to translate Martin's proof into a language of dilator due to some terminological incoherence.%
\footnote{The main technical issue the author confronted is that there is no obvious dendroid-counterpart of $\varrho(\sigma,\tau)$-like function in \cite[Lemma 4.1]{Martin1980InfiniteGames}. In terms of a dendrogram, $\varrho$ corresponds to the parameter parts.}
Hence, we introduce a tree structure named \emph{dendrogram}, which codes a dendroid as a single tree. We will iterate measures along a dendrogram to get a measure family of a measurable dilator we construct.

\setcounter{tocdepth}{2}
\tableofcontents

\section{Elementary embeddings} \label{Section: Elementary embeddings}
In this section, we review facts about rank-into-rank embedding. We mostly focus on notions introduced by Martin \cite{Martin1980InfiniteGames} with additional details from \cite{Dimonte2018I0rank}.
\emph{We will avoid using the full axiom of choice in the rest of the paper unless specified,} although we may use its weaker variant, like the axiom of countable or dependent choice.

\subsection{Rank-into-rank embedding}
Let $j\colon V_\lambda\to V_\lambda$ be an $\rmI_3$-embedding such that $\lambda=\sup_{n<\omega} \kappa_n$. Let us define the following notions:
\begin{definition}
    \begin{enumerate}
        \item $M_0 = V_\lambda$, $j_0=j$.
        \item $M_{\alpha+1} = \bigcup_{\xi\in \Ord^{M_\alpha}}j_\alpha(V_\xi^{M_\alpha})$, $j_{\alpha+1}=j_\alpha\cdot j_\alpha$.
        \item $j_{\alpha,\alpha}$ is the identity, $j_{\alpha,\beta+1}=j_\beta\circ j_{\alpha,\beta}$ for $\alpha\le\beta$.
        \item If $\alpha>0$ is a limit ordinal, define $((M_\alpha, j_\alpha), j_{\beta,\alpha})_{\beta<\alpha}$ is the direct limit of $((M_\beta,j_\beta),j_{\beta,\gamma})_{\beta\le\gamma<\alpha}$.; More precisely,
        \begin{equation*} \textstyle
            M_\alpha = \bigcup_{\beta<\alpha} \{(\beta,x)\mid \beta<\alpha, x\in M_\beta\}/\sim,
        \end{equation*}
        where $(\beta,x)\sim(\gamma,y)$ iff there is $\delta<\alpha$ such that $\beta,\gamma\le\delta$ and $j_{\beta,\delta}(x)=j_{\gamma,\delta}(y)$. Then for $x\in M_\beta$,
        \begin{equation*}
            j_{\beta,\alpha}(x) := [\beta,x]_\sim, \quad 
            j_\alpha([\beta,x]_\sim) = [\beta,j_\beta(x)]_\sim = j_{\beta,\alpha}(j_\beta(x)).
        \end{equation*}
        where $[\beta,x]_\sim$ is a $\sim$-equivalence class given by $(\beta,x)$.
        We also take $M_{\alpha+n} = M_\alpha$.
    \end{enumerate}

    $M_\alpha$ may not be well-founded for a limit $\alpha$. If $M_\alpha$ is well-founded, then we say $j$ is \emph{$\alpha$-iterable}. We identify $M_\alpha$ with its transitive collapse if $M_\alpha$ is well-founded.
\end{definition}

\begin{lemma} \label{Lemma: Basic facts on iterating elementary embeddings}
    Let $\alpha, \beta$ be an ordinal and $n<\omega$.
    \begin{enumerate}
        \item $j_{\alpha,\alpha+n} = j_\alpha^n:= \underbrace{j_\alpha\circ\cdots\circ j_\alpha}_{\text{$n$ times}}$.
        \item $j_\alpha\cdot j_{\alpha+n} = j_{\alpha+n+1}$.
        \item $j_\alpha\circ j_{\alpha+n} = j_{\alpha+n+1}\circ j_\alpha$.
        \item $j_{\alpha,\beta} = j_{\alpha+1,\beta}\circ j_\alpha$ for $\alpha<\beta$.
        \item $j_{\alpha,\beta}\circ j_{\alpha+n} = j_{\beta+n}\circ j_{\alpha,\beta}$ for $\alpha\le\beta$.
        \item $j_\alpha$ is well-defined elementary embedding from $(M_\alpha,j_\alpha)$ to $(M_\alpha,j_{\alpha+1})$ and $M_\alpha=M_{\alpha+1}$.
    \end{enumerate}
\end{lemma}
\begin{proof} 
    We prove it by induction on $\max(\alpha,\beta)$.
    \begin{enumerate}[leftmargin=0pt]
        \item We can prove it by induction on $n$.
        \item The case $n=0$ is clear by definition. For the successor case,
        \begin{equation*}
            j_\alpha\cdot j_{\alpha+n+1} = (j_\alpha\cdot j_{\alpha+n})\cdot(j_\alpha\cdot j_{\alpha+n}) = j_{\alpha+n+1}\cdot j_{\alpha+n+1} = j_{\alpha+n+2}.
        \end{equation*}
        
        \item $(j_\alpha\circ j_{\alpha+n})(x) = j_\alpha(j_{\alpha+n}(x)) = (j_\alpha\cdot j_{\alpha+n})(j_\alpha(x)) = (j_{\alpha+n+1}\circ j_\alpha)(x)$.

        \item We can prove it by induction on $\beta$. 
        
        \item The case $\alpha=\beta$ is trivial. Also,
        \begin{equation*}
            j_{\alpha,\beta+1}\circ j_{\alpha+n} = j_\beta\circ j_{\alpha,\beta}\circ j_{\alpha+n} = j_\beta\circ j_{\beta+n}\circ j_{\alpha,\beta}
            = j_{\beta+n+1}\circ j_\beta\circ j_{\alpha+\beta} = j_{\beta+n+1}\circ j_{\alpha,\beta+1}.
        \end{equation*}
        For a limit $\beta$, let us prove it by induction on $n$: 
        \begin{equation*} 
            j_{\alpha,\beta}(j_{\alpha+n+1}(\xi)) = j_{\alpha,\beta}((j_\alpha\cdot j_{\alpha+n})(\xi)) = (j_{\alpha,\beta}\cdot (j_\alpha\cdot j_{\alpha+n}))(j_{\alpha,\beta}(\xi))
        \end{equation*}
        and
        \begin{multline*} \textstyle
            j_{\alpha,\beta}\cdot (j_\alpha\cdot j_{\alpha+n}) = j_{\alpha,\beta} \left(j_\alpha\left(\bigcup_{\eta\in\Ord^{M_\alpha}} j_{\alpha+n}\restriction V_\eta^{M_\alpha}\right)\right) = j_{\beta}\left(j_{\alpha,\beta}\left(\bigcup_{\eta\in\Ord^{M_\alpha}} j_{\alpha+n}\restriction V_\eta^{M_\alpha}\right)\right)
            \\ \textstyle = j_{\beta}\left(\bigcup_{\eta\in\Ord^{M_\alpha}} j_{\alpha,\beta}(j_{\alpha+n}\restriction V_\eta^{M_\alpha})\right)
            = j_{\beta}\left(\bigcup_{\eta\in\Ord^{M_\beta}} j_{\beta+n}\restriction V_{\eta}^{M_\beta}\right) = j_\beta\cdot j_{\beta+n} = j_{\beta+n+1}.
        \end{multline*}
        
        \item We prove it in the following order:
        \begin{enumerate}
            \item $j_\alpha$ is well-defined.
            \item $j_{\alpha}$ is ordinal-cofinal: For every $\xi\in \Ord^{M_\alpha}$ there is $\eta\in \Ord^{M_\alpha}$ such that $M_\alpha \vDash \xi < j_{\alpha}(\eta)$.
            \item If $j_\alpha\colon M_\alpha\to M_\alpha$ is elementary for formulas over the language $\{\in\}$, then $j_\alpha\colon (M_\alpha,j_\alpha)\to (M_\alpha,j_{\alpha+1})$.
            \item $j_{0,\alpha}\colon (V_\lambda,j_0)\to (M_\alpha, {j_\alpha})$ is elementary.
            \item $M_\alpha = M_{\alpha+1}$.
            \item $j_{\alpha+1} \colon M_\alpha\to M_\alpha$ is elementary.
        \end{enumerate}
        
        \begin{enumerate}[wide, labelwidth=!, labelindent=0pt]
            \item First, $j_\alpha$ is clearly well-defined if $\alpha=0$ or $\alpha=\gamma+1$ for some $\gamma<\alpha$.
            For a limit $\alpha$, the issue is if $(\gamma,x)\sim(\delta,y)$ for $\gamma,\delta<\alpha$, $x\in M_\gamma$, $y\in M_\delta$ ensures $j_\alpha([\gamma,x]_\sim) = j_\alpha([\delta,y]_\sim)$ as we defined $j_\alpha([\gamma,x]_\sim) = [\gamma,j_\gamma(x)]_\sim$.
            Fix $\zeta<\alpha$ such that $\gamma,\delta<\zeta$ and $j_{\gamma,\zeta}(x) = j_{\delta,\zeta}(y)$.
            \begin{multline*}
                 j_\alpha([\gamma,x]_\sim) = [\gamma,j_\gamma(x)]_\sim 
                 = [\zeta, j_{\gamma,\zeta}\circ j_\gamma(x)]_\sim
                 = [\zeta, j_\zeta \circ j_{\gamma,\zeta}(x)]_\sim \\ 
                 = [\zeta, j_\zeta \circ j_{\delta,\zeta}(y)]_\sim 
                 = [\zeta, j_{\delta,\zeta} \circ j_\delta (y)]_\sim 
                 = [\delta,j_\delta(x)]_\sim = j_\alpha([\delta,y]_\sim).
            \end{multline*}
            Thus $j_\alpha$ is always well-defined.

            \item The case $\alpha=0$ follows from the assumption $\lambda = \sup_{n<\omega} \kappa_n$. If $\alpha = \gamma+1$, the inductive hypothesis gives $M_\gamma=M_{\gamma+1}$. Fix $\xi\in \Ord^{M_\gamma} = \Ord^{M_\alpha}$, then we can find $\eta\in \Ord^{M_\alpha}$ such that $\xi<j_\gamma(\eta)$.
            Hence 
            \begin{equation*}
                \xi \le j_\gamma(\xi) < j_\gamma(j_\gamma(\eta)) = j_{\gamma+1}(j_\gamma(\eta)) = j_\alpha(j_\gamma(\eta)), 
            \end{equation*}
            as desired.
            If $\alpha$ is limit, then every ordinal in $M_\alpha$ has the form $[\gamma,\xi]_\sim$ for some $\gamma<\alpha$ and $\xi \in \Ord^{M_\gamma}$.
            We can find $\eta\in \Ord^{M_\gamma}$ such that $M_\gamma\vDash \xi<j_\gamma(\eta)$, so
            \begin{equation*}
                [\gamma,\xi]_\sim < [\gamma, j_\gamma(\eta)]_\sim = j_\alpha([\gamma, \eta]_\sim).
            \end{equation*}

            \item  Now suppose that $j_\alpha\colon M_\alpha\to M_\alpha$ is elementary for formulas over the language $\{\in\}$. 
            Then let us employ the following general fact: 
            \begin{lemma} 
                Suppose that $N$ is a model of $\mathsf{Z}$ + $\Sigma_1$-Collection + `$\xi\mapsto V_\xi$ is well-defined,' $j\colon N\to N$ is $\Delta_0$-elementary. If $A\subseteq N$ is \emph{amenable}, i.e., $x\cap A\in N$ for every $x\in N$, then $j\colon (N,A)\to (N,j[A])$ is $\Delta_0$-elementary over the language $(\in,A)$, where $j[A] = \bigcup_{\xi\in\Ord^N}j(A\cap V_\xi^N)$. 
            \end{lemma}
             Its proof follows from the proof of \cite[Lemma 4.14(1)]{JeonMatthews2022}. Since $j_\alpha\colon M_\alpha\to M_\alpha$ is ordinal cofinal, we can prove that $j_\alpha\colon (M_\alpha,j_\alpha)\to (M_\alpha,j_{\alpha+1})$ is fully elementary by induction on the quantifier complexity of a formula as presented in \cite[Lemma 4.14(2)]{JeonMatthews2022}.

            \item The case $\alpha=0$ is easy, and the successor case follows from the induction hypothesis and the previous item. The limit case follows from the definition of $(M_\alpha,j_\alpha,j_{\beta,\alpha})_{\beta<\alpha}$.

            \item $j_\alpha(V_\xi^{M_\alpha})\subseteq M_\alpha$ gives $M_{\alpha+1}\subseteq M_\alpha$. 
            For $M_\alpha\subseteq M_{\alpha+1}$, observe that
            \begin{equation*}
                V_\lambda \vDash \forall \xi\in \Ord [j_0(V_\xi) = V_{j_0(\xi)}].
            \end{equation*}
            Since $j_{0,\alpha}\colon (V_\lambda,j_0)\to (M_\alpha, {j_\alpha})$ is elementary, we have
            \begin{equation*}
                M_\alpha \vDash \forall \xi\in \Ord [j_\alpha(V_\xi) = V_{j_\alpha(\xi)}].
            \end{equation*}
            Furthermore, in $M_\alpha$, for each $x$ we can find $\xi\in\Ord$ such that $\rank x< j_\alpha(\xi)$. Hence $x\in V_{j_\alpha(\xi)} = j_\alpha(V_\xi)$. It proves $M_\alpha\subseteq M_{\alpha+1}$.

            \item  For the elementarity of $j_{\alpha+1}$, observe that for a given formula $\phi$ we have
            \begin{equation*}
                M_\alpha \vDash \forall\xi\in\Ord \forall \vec{x}\in V_\xi [\phi(\vec{x})\lr \phi((j\restriction V_\xi)(\vec{x}))].
            \end{equation*}
            Fix $\xi$ and apply $j_\alpha$. Then we get
            \begin{equation*}
                M_\alpha \vDash \forall \vec{x}\in j_\alpha(V_\xi) [\phi(\vec{x})\lr \phi(j_\alpha(j\restriction V_\xi)(\vec{x}))].
            \end{equation*}
            Since $\xi$ is arbitrary, we have
            \begin{equation*}
                M_\alpha \vDash \forall \vec{x}\in M_{\alpha+1} [\phi(\vec{x})\lr \phi(j_{\alpha+1}(\vec{x}))]. \qedhere 
            \end{equation*}
        \end{enumerate}
    \end{enumerate}
\end{proof}

The following is an easy corollary of the previous proposition:
\begin{corollary} \pushQED{\qed}
    $j_{\alpha,\beta}\circ j_{\alpha+n,\alpha+m} = j_{\beta+n,\beta+m}\circ j_{\alpha,\beta}$ for $\alpha\le\beta$ and $n\le m<\omega$. \qedhere 
\end{corollary}

We will use the convention $\kappa_\alpha = j_{0,\alpha}(\kappa_0)$ for a general $\alpha$. It is an ordinal if $j$ is $\alpha$-iterable, but it can be ill-founded otherwise. Now, let us state a lemma about critical points whose proof is straightforward:
\begin{lemma} \pushQED{\qed}
    Let $\alpha$, $\beta$ be ordinals and $n<\omega$.
    \begin{enumerate}
        \item $\crit j_\alpha = \kappa_\alpha$.
        \item $j_{\alpha,\beta}(\kappa_{\alpha+n}) = \kappa_{\beta+n}$ for $\alpha\le \beta$. \qedhere 
    \end{enumerate}
\end{lemma}

We will use the following large cardinal notion to construct a measurable dilator:
\begin{definition}
    Let $j\colon V_\lambda\to V_\lambda$ be an $\rmI_3$-embedding. We say $j$ is \emph{iterable} if $M_\alpha$ defined from $j$ is well-founded for every $\alpha$. We say $\lambda$ is \emph{iterable} if it has an iterable embedding $j\colon V_\lambda\to V_\lambda$.
\end{definition}
Although unnecessary in this paper, it is worthwhile to note that an $\rmI_3$-embedding is iterable if and only if $M_\alpha$ is well-founded for every $\alpha<\omega_1$. (See \cite{Dimonte2018I0rank} for the details.) We may compare it with the fact that a countable semidilator $D$ is a dilator if and only if $D(\alpha)$ is well-founded for every $\alpha<\omega_1$. It gives a clue that elementary embedding iteration may have the structure of a dilator.

\subsection{\texorpdfstring{$\beta$}{Beta}-embedding}

We will consider a measure over $\kappa_n$ for an iterable embedding $j\colon V_\lambda\to V_\lambda$. The measure will be a projection of a measure over the set of \emph{$\beta$-embeddings} defined as follows: 
\begin{definition}
    An embedding $k\colon V_{\alpha+\beta}\to V_{\alpha'+\beta}$ is a \emph{$\beta$-embedding} if $\crit k=\alpha>\beta$. 
\end{definition}
Then we can see that 
\begin{lemma}
    If $k\colon V_{\alpha+\beta}\to V_{\alpha'+\beta}$ is a $\beta$-embedding, then $k(\alpha)+\beta=\alpha'+\beta$ and $k(\alpha)\le \alpha'$
\end{lemma}
\begin{proof}
    Observe that
    \begin{equation*}
        V_{\alpha+\beta}\vDash \forall\xi\in\Ord (\xi<\alpha\lor \exists \eta<\beta (\xi=\alpha+\eta)).
    \end{equation*}
    Hence by elementarity,
    \begin{equation*}
        V_{\alpha'+\beta}\vDash \forall\xi\in\Ord (\xi<k(\alpha)\lor \exists \eta<\beta (\xi=k(\alpha)+\eta)).
    \end{equation*}
    This shows $\alpha'+\beta = k(\alpha)+\beta$.
    Since $\alpha$ is a critical point, it is inaccessible. This shows $k(\alpha)$ is also inaccessible, so if $\alpha'<k(\alpha)$ then $\alpha'+\beta < k(\alpha)$, a contradiction.
\end{proof}
It is not true that $k(\alpha)=\alpha'$ holds (It fails when, for example, $\alpha' = k(\alpha)+1$ and $\beta=\omega$.) However, by replacing $\alpha'$ if necessary, we may assume that $\alpha' = k(\alpha)$, and we will assume throughout this paper that every $\beta$-embedding $k\colon V_{\alpha+\beta}\to V_{\alpha'+\beta}$ maps $\alpha$ to $\alpha'$.

We can define a measure for $\beta$-embeddings:
\begin{definition}
    Let $k$ be a $\beta$-embedding and $\gamma+1<\beta$. Define
    \begin{equation*}
        \Emb^k_\gamma = \{k'\colon V_{\crit k'+\gamma}\to V_{k'(\crit k')+\gamma} \mid \crit k'< \crit k,\, k'(\crit k') = \crit k\}
    \end{equation*}
    and a measure $\mu^k_\gamma$ by
    \begin{equation*}
        X \in \mu^k_\gamma \iff k\restriction V_{\crit k+\gamma}\in k(X).
    \end{equation*}
\end{definition}
It can be easily shown that $\Emb^k_\gamma \in \mu^k_\gamma$, so we can think of $\mu^k_\gamma$ as a measure over $\Emb^k_\gamma$.
Also, the following facts are easy to verify:
\begin{lemma} \pushQED{\qed}
    Let $k$ be a $\beta_1$-embedding and $\gamma+1<\beta_0\le\beta_1$.
    Then $\Emb^k_\gamma = \Emb^{k\restriction V_{\crit k+\beta_0}}_\gamma$ and $\mu^k_\gamma = \mu^{k\restriction V_{\crit k+\beta_0}}_\gamma$. \qedhere 
\end{lemma}

\begin{lemma} \label{Lemma: Measure projection lemma} \pushQED{\qed}
    Let $k$ be a $\beta$-embedding and $\gamma_0\le \gamma_1<\gamma_1+1<\beta$.
    Consider the projection map $\pi^k_{\gamma_0,\gamma_1}\colon \Emb^k_{\gamma_1}\to\Emb^k_{\gamma_0}$, $\pi^k_{\gamma_0,\gamma_1}(k') = k'\restriction V_{\crit k'+\gamma_0}$. Then
    \begin{equation*}
        X\in \mu^k_{\gamma_0} \iff 
        (\pi^k_{\gamma_0,\gamma_1})^{-1}[X] = \{z\in \Emb^k_{\gamma_1}\mid z\restriction V_{\crit z+\gamma_0}\in X\}\in \mu^k_{\gamma_1}.
    \end{equation*}
    Also,
    \begin{equation*}
        Y\in \mu^k_{\gamma_1} \implies \pi^k_{\gamma_0,\gamma_1}[Y]\in \mu^k_{\gamma_0}.
        \qedhere 
    \end{equation*}
\end{lemma}

The following proof is a modification of Schlutzenberg's answer on MathOverflow:
\begin{lemma} \label{Lemma: Schlutzenberg's lemma for value comparison}
    Let $k\colon V_{\alpha}\to V_{\alpha'}$ and $\xi<\delta<\alpha$.
    If $\xi\notin \ran k$, then $k(k\restriction V_\delta)(\xi) < k(\xi)$.
\end{lemma}
\begin{proof}
    Let us choose the least $\gamma\le\delta$ such that $k(\gamma)>\xi$. (Such $\gamma$ exists since $k(\delta)\ge \delta>\xi$.) Note that $k(\gamma)>\xi+1$ also holds, otherwise $k(\gamma)=\xi+1$ implies $\gamma=\gamma'+1$ for some $\gamma'<\gamma$ and $k(\gamma') = \xi$.
    
    Clearly we have $\sup k[\gamma] \le \xi$, so $k(\sup k[\gamma])\le k(\xi)$.
    Also,
    \begin{equation*}
        k(\sup k[\gamma]) = k(\sup_{\zeta<\gamma} (k\restriction V_\delta)(\zeta))
        = \sup_{\zeta<k(\gamma)} (k(k\restriction V_\delta))(\zeta)
        > (k(k\restriction V_\delta))(\xi).
    \end{equation*}
    Putting everything together, we have $k(k\restriction V_\delta)(\xi) < k(\sup k[\gamma]) \le k(\xi)$.
\end{proof}

\begin{corollary}
    For an elementary embedding $k\colon V_\lambda\to V_\lambda$ we have $(k\cdot k)(\xi)\le k(\xi)$ for every $\xi<\lambda$. In particular, $j_{n+1}(\xi)\le j_n(\xi)$ holds for every $\xi<\lambda$ and $n\in\bbN$.
\end{corollary}
\begin{proof}
    If $\xi\in\ran k$, so if $\xi=k(\eta)$ for some $\eta<\lambda$, then
    \begin{equation*}
        (k\cdot k)(\xi) = (k\cdot k)(k(\eta)) = k(k(\eta)) = k(\xi).
    \end{equation*}
    If $\xi\notin \ran k$, then \autoref{Lemma: Schlutzenberg's lemma for value comparison} implies $(k\cdot k)(\xi)=k(k\restriction V_{\xi+\omega})(\xi) < k(\xi)$.
    $j_{n+1}\le j_n$ follows from $j_{n+1} = j_n\cdot j_n$.
\end{proof}

Most ultrafilters induced from an elementary embedding are normal. The ultrafilter $\mu^k_\gamma$ is also `normal' in the following sense:
\begin{lemma} \label{Lemma: Diagonal intersection of sets of embeddings}
    Let $\gamma_0<\gamma_1<\beta$ be limit ordinals and $k$ a $\beta$-embedding. 
    If $\{Y_{k'}\mid k'\in \Emb^k_{\gamma_0}\}\subseteq \mu^k_{\gamma_1}$ is a family of sets, then
    \begin{equation*}
        \triangle_{k'\in \Emb^k_{\gamma_0}} Y_{k'}:=\{k''\in \Emb^k_{\gamma_1}\mid \forall k'\in \ran k''\cap \Emb^k_{\gamma_0} (k''\in Y_{k'})\} \in \mu^k_{\gamma_1}.
    \end{equation*}
\end{lemma}
\begin{proof}
    We want to show that $k\restriction V_{\crit k+\gamma_1}\in k(\triangle_{k'\in \Emb^k_{\gamma_0}} Y_{k'})$, which is equivalent to
    \begin{equation*}
        \forall k'\in \ran (k\restriction V_{\crit k+\gamma_1}) \cap k(\Emb^k_{\gamma_0})\big[ k\restriction V_{\crit k+\gamma_1} \in k(Y)_{k'}\big].
    \end{equation*}
    For each $k'\in \ran (k\restriction V_{\crit k+\gamma_1})\cap k(\Emb^k_{\gamma_0})$, we can find $\hat{k}'\in V_{\crit k+\gamma_1}\cap \Emb^k_{\gamma_0}$ such that $k'=k(\hat{k}')$.
    Since $Y_{\hat{k}'}\in \mu^k_{\gamma_1}$, we have $k\restriction V_{\crit k+\gamma_1} \in k(Y_{\hat{k}'}) = k(Y)_{k'}$, as desired.
\end{proof}

We will frequently use `for $\mu$-almost all' throughout this paper, so it is convenient to introduce measure quantifier notation:
\begin{definition} \label{Definition: Measure quantifier}
    Let $\mu$ be an ultrafilter over $D$. Let us define
    \begin{equation*}
        \forall(\mu) x\in D \phi(x) \iff \{x\in D \mid \phi(x)\}\in \mu.
    \end{equation*}
\end{definition}

Measure quantifier can be iterated, so for example, if $\mu_i$ is a measure over $D_i$ for $i=0,1$,
\begin{equation*}
    \forall(\mu_0) x_0\in D_0 \forall(\mu_1) x_1\in D_1 \phi(x_0,x_1) \iff \{x_0\in D_0\mid \{x_1\in D_1\mid\phi(x_0,x_1)\}\in \mu_1\}\in \mu_0.
\end{equation*}

The reader should be careful that we cannot switch the order of two quantifiers in the above definition. Switching the order of measure quantifier is impossible even when the same measure quantifier repeats: For example, consider the following statement for a $\kappa$-complete $\mu$ over a measurable cardinal $\kappa$.
\begin{equation*}
    \forall(\mu) \alpha<\kappa \forall(\mu) \beta<\kappa [\alpha<\beta] :\iff \{\alpha<\kappa\mid \{\beta<\kappa\mid \alpha<\beta\}\in\mu \}\in\mu.
\end{equation*}
However, we can `delete' unused measure quantifiers:
\begin{lemma} \pushQED{\qed} \label{Lemma: Eliminating unused measure quantifiers}
    Let $\phi(x,y)$ be a formula with no $z$ as a free variable.
    If $\mu_i$ is a measure over $D_i$ for $i=0,1$ and $\nu$ a measure over $D$, then we have
    \begin{equation*}
        \forall(\mu_0) x\in D_0 \forall(\nu) z\in D \forall(\mu_1) y\in D_1 \phi(x,y)\iff \forall(\mu_0) x\in D_0 \forall(\mu_1) y\in D_1 \phi(x,y). \qedhere 
    \end{equation*}
\end{lemma}

The following lemma says an `upper diagonal' over the set $\Emb^k_{\gamma_0}\times \Emb^k_{\gamma_1}$ for $\gamma_0<\gamma_1$ is large, like the set $\{(\alpha,\beta)\mid\alpha<\beta<\kappa\}$ is large under a normal ultrafilter on a measurable cardinal $\kappa$:
\begin{lemma} \label{Lemma: Upper diagonal is large}
    Let $\gamma_0<\gamma_1<\beta$ be limit ordinals and $k$ be a $\beta$-embedding.
    Then 
    \begin{equation*}
        \forall (\mu^k_{\gamma_0}) k^0\in \Emb^k_{\gamma_0}\forall (\mu^k_{\gamma_1})k^1\in \Emb^k_{\gamma_1} \big[ k^0 \in \ran k^1\big].
    \end{equation*}
\end{lemma}
\begin{proof}
    The claim is equivalent to 
    \begin{equation*}
        \forall (\mu^k_{\gamma_0}) k^0\in \Emb^k_{\gamma_0}\big[ \{k^1\in \Emb^k_{\gamma_1}\mid k^0\in \ran k^1\}\in \mu^k_{\gamma_1}\big]
    \end{equation*}
    and by the definition of $\mu^k_{\gamma_1}$, it is equivalent to
    \begin{equation*}
        \forall (\mu^k_{\gamma_0}) k^0\in \Emb^k_{\gamma_0}\big[ k(k^0) \in \ran k\restriction V_{\crit k+\gamma_1}\big],
    \end{equation*}
    which holds since $\Emb^k_{\gamma_0}\subseteq V_{\crit k+\gamma_1}$.
\end{proof}

Now let us fix limit ordinals $\gamma_0<\cdots<\gamma_{m-1}<\beta$ and a $\beta$-embedding $k$.
We define the product of measures $\mu=\mu^k_{\gamma_0}\times\cdots\times \mu^k_{\gamma_{m-1}}$ over $\Emb^k_{\gamma_0}\times\cdots\times\Emb^k_{\gamma_{m-1}}$ by
\begin{equation*}
    X \in \mu\iff \forall (\mu^k_0) k^0\in \Emb^k_{\gamma_0}\cdots \forall (\mu^k_{m-1}) k^{m-1}\in \Emb^k_{\gamma_{m-1}} \big[ \lag k^0,\cdots, k^{m-1}\rag\in X\big].
\end{equation*}
As an application of \autoref{Lemma: Upper diagonal is large}, we have
\begin{equation*}
    \Delta^k_m := \big\{\lag k^0,\cdots,k^{m-1}\rag\in \Emb^k_{\gamma_0}\times\cdots\times\Emb^k_{\gamma_{m-1}} \mid  k^0\in \ran k^1\land\cdots\land k^{m-2}\in\ran k^{m-1}\} \in \mu.
\end{equation*}

The following proposition says the product measure $\mu$ is generated by an intersection of $\Delta^k_m$ and a cube whose each component is large:
\begin{proposition} \label{Proposition: Product measure is generated by cubes}
    For each $X\in \mu$ we can find $Y\in \mu^k_{\gamma_{m-1}}$ such that
    \begin{equation*}
        \Delta^k_m\cap \big((\pi^k_{\gamma_0,\gamma_{m-1}})[Y]\times \cdots \times (\pi^k_{\gamma_{m-2},\gamma_{m-1}})[Y]\times Y\big) \subseteq X.
    \end{equation*}
\end{proposition}
\begin{proof}
    We prove this proposition for $m=3$; The general case follows from a similar argument.
    Define
    \begin{equation*}
        Y_0 = \big\{k^0 \in \Emb^k_{\gamma_0} \mid \forall(\mu^k_{\gamma_1})k^1\in \Emb^k_{\gamma_1}\forall(\mu^k_{\gamma_2})k^2\in \Emb^k_{\gamma_2}\big[\lag k^0,k^1,k^2\rag\in X\big]\big\}.
    \end{equation*}
    Then $X\in \mu$ implies $Y_0\in \mu^k_{\gamma_0}$.
    Next, we define
    \begin{equation*}
        Y_1 = (\pi^k_{\gamma_0,\gamma_1})^{-1}[Y_0]\cap  \big\{k^1\in \Emb^k_{\gamma_1}\mid \forall k^0\in Y_0\cap \ran k^1 \forall(\mu^k_{\gamma_2}) k^2\in \Emb^k_{\gamma_2}\big[ \lag k^0,k^1,k^2\rag\in X\big]\big\}.
    \end{equation*}
    We claim $Y_1\in \mu^k_{\gamma_1}$: To see this, for each $k^0\in \Emb^k_{\gamma_0}$ let us take
    \begin{equation*}
        Y_{1,k^0}:= \{k^1\in\Emb^k_{\gamma_1}\mid k^0\in Y_0\to \forall (\mu^k_{\gamma_2})k^2\in\Emb^k_{\gamma_2} [\lag k^0,k^1,k^2\rag\in X]\}.
    \end{equation*}
    Then $Y_{1,k^0}\in\mu^k_{\gamma_1}$ for every $k^1\in\Emb^k_{\gamma_1}$. Then by \autoref{Lemma: Diagonal intersection of sets of embeddings},
    \begin{equation*}
        \triangle_{k^0\in\Emb^k_{\gamma_0}} Y_{1,k^0} = \big\{k^1\in \Emb^k_{\gamma_1}\mid \forall k^0\in Y_0\cap \ran k^1 \forall(\mu^k_{\gamma_2}) k^2\in \Emb^k_{\gamma_2}\big[ \lag k^0,k^1,k^2\rag\in X\big]\big\} \in \mu^k_{\gamma_1}.
    \end{equation*}
    Combining with \autoref{Lemma: Measure projection lemma}, we have $Y_1\in \mu^k_{\gamma_1}$.
    Lastly, let us define
    \begin{equation*}
        Y_2 = (\pi^k_{\gamma_1,\gamma_2})^{-1}[Y_1]\cap \{k^2\in\Emb^k_{\gamma_2}\mid \forall k^1\in Y_1\cap\ran k^2 \forall k^0\in Y_0\cap \ran k^1 [\lag k^0,k^1,k^2\rag\in X]\}.
    \end{equation*}
    Then we can show $Y_2\in\mu^k_{\gamma_2}$. 
    It is straightforward to see $\Delta^k_3\cap (Y_0\times Y_1\times Y_2)\subseteq X$ and $\pi^k_{\gamma_i,\gamma_2}[Y_2]\subseteq \pi^k_{\gamma_i,\gamma_2}[(\pi^k_{\gamma_i,\gamma_2})^{-1}[Y_i]]\subseteq Y_i$, so we have a desired result.
\end{proof}

Let us observe that for limit ordinals $\gamma_0<\gamma_1<\beta$, a $\beta$-embedding $k$, and $k_i\in\Emb^k_{\gamma_i}$ for $i=0,1$, the relation $k_0\in \ran k_1$ is similar to the Mitchell order. It can be easily seen that for $\gamma_0<\gamma_1<\gamma_2<\beta$ and $k_i\in\Emb^k_{\gamma_i}$ for $i<3$, $k_0\in \ran k_1$ and $k_1\in\ran k_2$ imply $k_0\in \ran k_2$.
The next proposition tells us there is an arbitrarily long countable sequence of elementary embeddings increasing under the Mitchell order:

\begin{proposition} \label{Proposition: Arbitrary long Mitchell chain}
    Let $\gamma<\beta$ be limit ordinals, $k$ a $\beta$-embedding, $\lag \gamma_\xi \mid \xi<\alpha\rag$ a countable increasing sequence of limit ordinals below $\gamma$, and $X\in \mu^k_\gamma$.
    Then we can find $\{k_\xi\mid \xi<\alpha\}\subseteq X$ such that for each $\eta<\xi<\alpha$, $k_\eta\restriction V_{\crit k_\eta+\gamma_\eta} \in \ran k_\xi\restriction V_{\crit k_\xi+\gamma_\xi}$.
\end{proposition}
\begin{proof}
    For $\delta\le\alpha$, define a \emph{$\delta$-chain} as a sequence $\vec{k}=\lag k_\xi\mid \xi<\delta\rag$ of members of $X$ such that for each $\eta<\xi<\delta$, $k_\eta\restriction V_{\crit k_\eta+\gamma_\eta} \in \ran k_\xi\restriction V_{\crit k_\xi+\gamma_\xi}$.
    We first claim that every $\delta$-chain extends to a $(\delta+1)$-chain:
    Suppose that $\vec{k}$ is a $\delta$-chain. Note that $\vec{k}\in V_{\crit k +\beta}$, so the following is witnessed by $k'=k\restriction V_{\crit k+\gamma}$:
    \begin{equation*}
        V_{k(\crit k)+\beta}\vDash \exists k'[k'\in k(X)\land k(\vec{k})\in \ran (k'\restriction V_{\crit k'+\gamma_\delta})].
    \end{equation*}
    Then by elementarity of $k$, we have
    \begin{equation*}
        V_{\crit k + \beta}\vDash  \exists k'[k'\in X\land \vec{k}\in \ran (k'\restriction V_{\crit k'+\gamma_\delta})].
    \end{equation*}
    (Note that $\gamma_\delta<\beta<\crit k$.) Take any $k'$ witnessing the previous claim, then $\vec{k}^\frown \lag k'\rag$ is a $(\delta+1)$-chain.

    Now let us prove by induction on $\delta\le\alpha$ that for every $\delta'<\delta$, a $\delta'$-chain extends to a $\delta$-chain.
    The case when $\delta$ is 0 or a successor is clear.
    For the limit case, fix a cofinal sequence $\lag \delta_n\mid n<\omega\rag$ of $\delta$ with $\delta_0=\delta'$.
    We know that every $\delta_n$-chain extends to a $\delta_{n+1}$-chain, so by the axiom of dependent choice, we can find a $\delta$-chain extending a given $\delta'$-chain. (This is why we require $\alpha<\omega_1$; Otherwise, we need a choice axiom stronger than the dependent choice.)
\end{proof}

Note that $\beta$-embedding in this subsection is irrelevant to Girard's $\beta$-logic \cite{Girard1982Logical}.

\section{Dilators} \label{Section: Dilators}
In this section, we review the details of the dilators we need in this paper.
This section constitutes a summary of an excerpt of the book draft \cite{Jeon??DilatorBook}, and we will state some results in the two subsections without proof. The proofs will appear in \cite{Jeon??DilatorBook}, but most of the proofs are also available in different sources (e.g., \cite{Girard1981Dilators, Girard1982Logical, Freund2024DilatorZoo, Jeon2024HigherProofTheoryI}.)

\subsection{Defining dilators}

There are several different but equivalent definitions of dilators. Girard defined dilators as autofunctors over the category of ordinals preserving direct limit and pullback, but this definition will not be used in this paper. We take two approaches to dilators: One is a denotation system, and the other is the Freund-styled definition.

We first define a preliminary notion named \emph{semidilators}.\footnote{\emph{Prae-dilator} in Freund's terminology}
Semidilators correspond to autofunctors over the category of linear orders preserving direct limit and pullback. Predilators additionally satisfy the monotonicity condition, and dilators additionally preserve well-orderedness. We will see that every dilator is a predilator.

Let us start with the definition of dilators as denotation systems:
\begin{definition}
    An \emph{arity diagram} $\cyrDe$ is a commutative diagram over the category of natural numbers with strictly increasing maps of the form
    \begin{equation} \label{Formula: Arity diagram definition diagram}
    \cyrDe = 
        \begin{tikzcd}[column sep=large]
                \cyrDe(\bot) & \cyrDe(1) \\
                \cyrDe(0) & \cyrDe(\top)
                \arrow[from=1-1, to=1-2]
                \arrow[from=1-1, to=2-1]
                \arrow[from=2-1, to=2-2, "{\cyrDe(0,\top)}", swap]
                \arrow[from=1-2, to=2-2, "{\cyrDe(1,\top)}"]
        \end{tikzcd}
    \end{equation}
    such that the above diagram is a pullback and $\ran \cyrDe(0,\top)\cup \ran \cyrDe(1,\top) = \field \cyrDe(\bot)$.
    We say an arity diagram is \emph{trivial} if $\cyrDe(\bot) = \cyrDe(0)=\cyrDe(1)=\cyrDe(\top)$ (so all arrows in an arity diagram are the identity map.)
    For an arity diagram $\cyrDe$ of the form \eqref{Formula: Arity diagram definition diagram}, the diagram $-\cyrDe$ is a diagram obtained by switching the order of $a_0$ and $a_1$:
    \begin{equation*}
    -\cyrDe = 
        \begin{tikzcd}[column sep=large]
                \cyrDe(\bot) & \cyrDe(0) \\
                \cyrDe(1) & \cyrDe(\top)
                \arrow[from=1-1, to=1-2]
                \arrow[from=1-1, to=2-1]
                \arrow[from=2-1, to=2-2, "{\cyrDe(1,\top)}", swap]
                \arrow[from=1-2, to=2-2, "{\cyrDe(0,\top)}"]
        \end{tikzcd}
    \end{equation*}
\end{definition}
Intuitively, an arity diagram is a diagrammatic way to express a pair of finite linear orders $(A,B)$ with their intersection and union.
Typical examples of arity diagrams are induced from the inclusion diagrams. For example, consider the following inclusion diagram:
\begin{equation*}
    \begin{tikzcd}
        \{1,3\} & \{0,1,3\} \\
        \{1,2,3,4\} & \{0,1,2,3,4\} 
        \arrow[from=1-1, to=1-2, "\subseteq"]
        \arrow[from=1-1, to=2-1, "\subseteq", swap]
        \arrow[from=2-1, to=2-2, "\subseteq", swap]
        \arrow[from=1-2, to=2-2, "\subseteq"]
    \end{tikzcd}
\end{equation*}
The above diagram is isomorphic to
\begin{equation*}
    \begin{tikzcd}
        \{\mathbf{0},\mathbf{1}\} & \{0,\mathbf{1},\mathbf{2}\} \\
        \{\mathbf{0},1,\mathbf{2},3\} & \{0,1,2,3,4\} 
        \arrow[from=1-1, to=1-2, "h"]
        \arrow[from=1-1, to=2-1, "k", swap]
        \arrow[from=2-1, to=2-2, "f", swap]
        \arrow[from=1-2, to=2-2, "g"]
    \end{tikzcd}
\end{equation*}
where $h$ and $k$ are maps sending boldface numbers to boldface numbers in an increasing manner, $f(n)=n+1$, and $g(0)=0$, $g(1)=1$, $g(2)=3$.

\begin{definition}
    Let $\ttL_m=(\ttL_m,\land,\lor)$ be the free distributive lattice generated by $\{0,1,\cdots,m-1\}$.%
    \footnote{We do not allow empty joins and meets in the free distributive lattice, so the least element of $\ttL_m$ is $0\land \cdots \land (m-1)$, and the largest element of $\ttL_m$ is $0\lor \cdots \lor (m-1)$.}
    An \emph{IU diagram $\bfcyrDe$ for $a_0,\cdots, a_{m-1} \in \LO$} (abbreviation of \emph{Intersection-Union diagram}) a functor from $\ttL_m$ to $\LO$ if we understand $\ttL_m$ as a category induced from its partial order structure satisfying the following, where $\bfcyrDe(i,j)$ for $i\le j\in \ttL_m$ denotes the map of the unique morphism $i\le j$ under $\bfcyrDe$:
    \begin{enumerate}
        \item $\bfcyrDe(i,i)$ is the identity map and for $i\le j\le k$, $\bfcyrDe(i,k)=\bfcyrDe(j,k)\circ \bfcyrDe(i,j)$.
        \item For each $i,j\in \ttL_m$, the following diagram is a pullback and $\ran \bfcyrDe(i,i\lor j)\cup\ran \bfcyrDe(j,i\lor j)=\bfcyrDe(i\lor j)$:
        \begin{equation*}
            \bfcyrDe[i,j] = 
            \begin{tikzcd}[column sep=large]
            \bfcyrDe(i\land j) & \bfcyrDe(j) \\
            \bfcyrDe(i) & \bfcyrDe(i\lor j)
            \arrow[from=1-1, to=1-2, "{\bfcyrDe(i\land j, j)}"]
            \arrow[from=1-1, to=2-1, "{\bfcyrDe(i\land j, i)}", swap]
            \arrow[from=2-1, to=2-2, "{\bfcyrDe(i, i\lor j)}", swap]
            \arrow[from=1-2, to=2-2, "{\bfcyrDe(j, i\lor j)}"]
            \end{tikzcd}
        \end{equation*}
        \item $\bfcyrDe(i) = a_i$ for each $i<m$.
    \end{enumerate}
    If every object in $\bfcyrDe$ is in a class $A\subseteq\LO$, then we say $\bfcyrDe$ is an \emph{IU diagram over $A$}.
\end{definition}

\begin{definition} \index{Diagram!for $n$ objects}
    Let $a, b\subseteq X$ be two finite suborders of a linear order $X$. \emph{The diagram $\Diag_X(a,b)$ of $a$, $b$ over $X$} is the unique arity diagram isomorphic to the inclusion diagram
    \begin{equation*}
        \begin{tikzcd}
            a\cap b & b \\
            a & a \cup b
            \arrow[from=1-1, to=1-2, "\subseteq"]
            \arrow[from=1-1, to=2-1, "\subseteq", swap]
            \arrow[from=2-1, to=2-2, "\subseteq", swap]
            \arrow[from=1-2, to=2-2, "\subseteq"]
        \end{tikzcd}
    \end{equation*}
    and more precisely, $\Diag_X(a,b)$ is the innermost diagram in the below commutative diagram, where $\en_a\colon |a|\to a$ is the unique order isomorphism for finite linear order $a$.
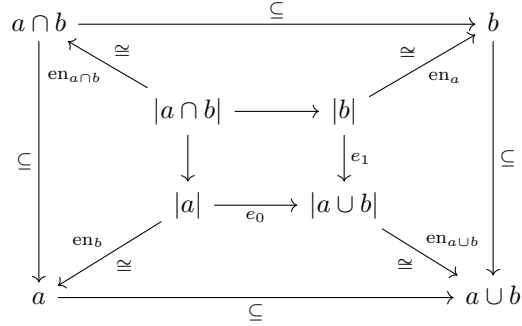
    \begin{equation*}
        \begin{tikzcd}
            a\cap b &&& b \\
            & {\lvert a\cap b\rvert} & {|b|} & \\
            & {|a|} & {\lvert a\cup b\rvert} & \\
            a &&& a\cup b
            \arrow[from=1-1, to=1-4, "\subseteq"]
            \arrow[from=1-1, to=4-1, "\subseteq", swap]
            \arrow[from=1-4, to=4-4, "\subseteq"]
            \arrow[from=4-1, to=4-4, "\subseteq", swap]
            \arrow[from=2-2, to=2-3]
            \arrow[from=2-2, to=3-2]
            \arrow[from=3-2, to=3-3, "e_0", swap]
            \arrow[from=2-3, to=3-3, "e_1"]
            \arrow[from=2-2, to=1-1, "\en_{a\cap b}" , "\cong"']
            \arrow[from=2-3, to=1-4, "\cong", "\en_{a}"']
            \arrow[from=3-2, to=4-1, "\cong", "\en_{b}"']
            \arrow[from=3-3, to=4-4, "\en_{a\cup b}", "\cong"']
        \end{tikzcd}
    \end{equation*}

    For finite suborders $a_0,\cdots,a_{n-1}\subseteq X$ of a linear order $X$, $\Diag_X(a_0,\cdots,a_{n-1})$ is the unique IU diagram over $\bbN$ isomorphic to the inclusion diagram $\calI\colon \ttL_n\to \calP(a_0\cup\cdots\cup a_{n-1})$, where $\calI$ is the unique lattice homomorphism from $(\ttL_n,\land,\lor)$ to $(\calP(a_0\cup\cdots\cup a_{n-1}),\cap,\cup)$ satisfying $\calI(i) = a_i$ for every $i<n$.
\end{definition}

We are ready to define semidilators:
\begin{definition}
    A \emph{semidilator} $D$ is a set of \emph{$D$-terms}, and each $D$-term $t$ comes with an \emph{arity} $\arity(t)\in \bbN$. 
    For each two $D$-terms $t_0$, $t_1$, an arity diagram $\cyrDe$ is an \emph{arity diagram for $t_0$ and $t_1$} if $\cyrDe(i)=\arity(t_i)$ for $i=0,1$.
    For such $t_0,t_1,\cyrDe$, we also have a binary relation $t_0<_\cyrDe t_1$.
    Then $D$ satisfies the following:
    \begin{enumerate}
        \item (Irreflexivity) If $t_0=t_1$ and $\cyrDe$ is trivial, then $t_0<_\cyrDe t_0$ does not hold.
        \item (Linearity) If $t_0\neq t_1$ or $\cyrDe$ is not trivial, then one of $t_0<_\cyrDe t_1$ or $t_1<_{-\cyrDe} t_0$ must hold.
        \item (Transitivity) For three $D$-terms $t_0,t_1,t_2$ such that $\arity(t_i) = a_i$, and an IU diagram $\bfcyrDe$ for $t_0,t_1,t_2$, if $t_0<_{\bfcyrDe[0,1]} t_1$ and $t_1 <_{\bfcyrDe[1,2]} t_2$, then $t_0 <_{\bfcyrDe[0,2]} t_2$.
    \end{enumerate}
    We also write
    \begin{equation*}
        t_0 \le_\cyrDe t_1 \iff t_0 <_\cyrDe t_1 \lor [\text{$\cyrDe$ is trivial and }t_0=t_1].
    \end{equation*}
\end{definition}
Note that semidilators is a model of a theory over a multi-sorted first-order logic:
Let us consider the sorts given by each arity, and take 
\begin{equation*}
    \calL^1 = \{<_\cyrDe\mid \cyrDe\text{ is an arity diagram}\}.
\end{equation*}
Then we can state the axioms of semidilators over the language $\calL^1$.

We can also talk about the morphism between two semidilators:
\begin{definition}
    Let $D$ and $E$ be two semidilators. A map $\iota\colon D\to E$ is an \emph{embedding} or a \emph{semidilator morphism} if it satisfies:
    \begin{enumerate}
        \item $\iota$ is a function $\field D$ to $\field E$.
        \item $\iota$ preserves the arity: i.e., $\arity(\iota(t)) = \arity(t)$ for every $t\in \field D$,
        \item For each two terms $t_0,t_1\in \field D$ and an arity diagram $\cyrDe$ between them, we have $t_0<_\cyrDe t_1$ iff $\iota(t_0) <_\cyrDe \iota(t_1)$.
    \end{enumerate}
    An embedding $\iota$ is an \emph{isomorphism} if $\iota\colon \field D\to  \field E$ is a bijection and $\iota^{-1}$ is also an embedding. We denote $D\le E$ or $D\cong E$ if there is an embedding or isomorphism from $D$ to $E$, respectively.
\end{definition}
We can see that every embedding is one-to-one, and the inverse function of a bijective embedding is also an embedding.

There are induced functors and natural transformations from semidilators and semidilator embeddings, respectively:
\begin{definition}
    From a given semidilator $D$ and a linear order $X$, let us define $D(X)$, the \emph{application of $D$ to $X$} by
    \begin{equation}
        D(X) = \{(t,a)\mid \text{$t$ is a $D$-term, $a\subseteq X$, and $|a|=\arity(t)$}\}.
    \end{equation}
    We write $t(a)$ instead of $(t, a)$, and we identify $a$ with a finite increasing sequence over $X$. The order of $D(X)$ is given by
    \begin{equation*}
        s(a) <_{D(X)} t(b) \iff s <_{\Diag_X(a,b)} t
    \end{equation*}
    For a strictly increasing function $f\colon X\to Y$, consider the map $D(f)\colon D(X)\to D(Y)$ given by
    \begin{equation*}
        D(f)(t,a) = (t,f[a]),
    \end{equation*}
    where $f[a] = \{f(x) \mid x\in a\}$.
\end{definition}

\begin{definition}
    For an embedding $\iota\colon D\to E$ and a linear order $X$, define $\iota_X\colon D(X)\to E(X)$ by
    \begin{equation*}
        \iota_X(t(a)) = (\iota(t),a).
    \end{equation*}
\end{definition}
We can see that $D(X)$ is a linear order if $D$ is a semidilator and $X$ is a linear order. It can also be shown that $\iota_X\colon D(X)\to E(X)$ and $\iota_Y\circ D(f) = E(f)\circ \iota_X$ for an increasing $f\colon X\to Y$ and an embedding $\iota\colon D\to E$. 

Freund (e.g., \cite{Freund2021Pi14}) defined semidilators as autofunctors over the category of linear orders with a support transformation.
\begin{definition}
    An \emph{F-semidilator} is a functor $F\colon \LO\to\LO$ with a support transformation $\supp^F\colon F\to [\cdot]^{<\omega}$ satisfying the \emph{support condition}: 
    For two linear orders $X$, $Y$ and an increasing $f\colon X\to Y$,
    \begin{equation*}
        \{\sigma\in D(Y) \mid \supp_Y(\sigma)\subseteq \ran(f)\} \subseteq \ran(D(f)).
    \end{equation*}
\end{definition}

The denotation system and the Freund-styled definition give different categories, but we can construct category equivalences between these two.

\begin{theorem}
    Let $\SDil$ be the category of semidilators with semidilator morphisms and $\SDil_\sfF$ be a category of F-semidilators with natural transformations.
    Then there are category equivalences $\fraka\colon \SDil\to \SDil_\sfF$ and $\frakf\colon \SDil_\sfF \to \SDil$ given by
    \begin{enumerate}
        \item For a semidilator $D$, $\fraka(D)(X) = D(X)$, $\fraka(\iota)_X(t(a))=\iota(t)(a)$ for $\iota\colon D\to E$, $X\in \LO$, $t\in \field(D)$, and $a\in [X]^{\arity t}$.
        \item For an F-semidilator $F$, $\frakf(F)$ has field 
        \begin{equation*}
            \Tr(F) = \{\sigma\in F(\omega)\mid \supp^F_n(\sigma)\in\omega\}.
        \end{equation*}
        We define $\arity^{\frakf(F)}(\sigma)=\supp^F_\omega(\sigma)$, and for two $\sigma,\tau\in \Tr(F)$ and an arity diagram $\cyrDe$ for these two terms, $\sigma <_\cyrDe \tau$ iff
        \begin{equation*}
            F(\omega)\vDash F(\cyrDe(0,\top))(\sigma) < F(\cyrDe(1,\top))(\tau).
        \end{equation*}
        Also for a natural transformation $\iota\colon F\to G$, we have $\frakf(\iota)\colon \frakf(F)\to\frakf(G)$ given by $\frakf(\iota)(\sigma) = \iota(\sigma)$.
    \end{enumerate}
\end{theorem}

In this paper, we conflate two different notions of dilators, so we pretend we get a (pre)dilator even when we actually get an F-(pre)dilator. We identify an F-(pre)dilator $F$ with a (pre)dilator $\frakf(F)$.
Note that the definition of the trace $\Tr(F)$ is slightly different from that of other materials (like \cite{Freund2021Pi14}), where $\Tr(F)$ is the set of $(\supp^F_\omega(\sigma),\sigma)$ for $\sigma \in F(\omega)$ such that $\supp^F_\omega(\sigma)\in\omega$.
Note that every F-semidilator is determined by its restriction to the category of natural numbers (i.e., the full subcategory of $\LO$ whose objects are natural numbers):
\begin{proposition}
    Let $F\colon \mathsf{Nat}\to \LO$ be a \emph{coded semidilator}; i.e., $F$ is a functor with the support function satisfying the support condition. Then $F$ extends to a semidilator $F\colon \LO\to\LO$ unique up to isomorphism. More precisely, we can define $\frakf(F)$, and it gives the desired extension.
    Moreover, suppose $\iota\colon F\to G$ is a natural transformation from a coded semidilator $F$ to another coded semidilator $G$. In that case, we can define $\frakf(\iota)$ and it gives a semidilator embedding from $\frakf(F)$ to $\frakf(G)$.
\end{proposition}

\subsection{Structure of predilators}
The following theorem characterizes prime dilators; i.e., a (pre)dilator with a unique term:
\begin{theorem}
    Let $D$ be a predilator. For $t\in \field(D)$, we have the \emph{priority permutation} $\Sigma^D_t$ over $\arity t$ such that for every linear order $X$ and $a,b\in [X]^{\arity t}$, $D(X)\vDash t(a)<t(b)$ if and only if
    \begin{equation*}
        \exists j<\arity t \big[a(\Sigma^D_t(j)) < b(\Sigma^D_t(j)) \land \forall i<j [a(\Sigma^D_t(i)) = b(\Sigma^D_t(i))]\big].
    \end{equation*}
\end{theorem}

The next question is how to compare two terms in a predilator.
Let us introduce subsidiary notions for the comparison:
\begin{definition}
    Let $D$ be a predilator and $s,t\in \field(D)$.
    Let $\hat{p}_{s,t} = \hat{p}^D_{s,t}\le \min(\arity s,\arity t)$ be the largest natural number such that 
    \begin{equation*}
        \forall i,j<\hat{p}_{s,t} [\Sigma^D_s(i) < \Sigma^D_s(j) \iff \Sigma^D_t(i) < \Sigma^D_t(j)].
    \end{equation*}
    For $p\le\hat{p}_{s,t}$ and an arity diagram $\cyrDe$ for $s$ and $t$ of the form 
    \begin{equation}
        \cyrDe =
        \begin{tikzcd}
            n_\cap & n_1 \\
            n_0 & n_\cup 
            \arrow[from=1-1, to=1-2] \arrow[from=2-1, to=2-2, "e_0"']
            \arrow[from=1-1, to=2-1] \arrow[from=1-2, to=2-2, "e_1"]
        \end{tikzcd}
    \end{equation}
    we say $\cyrDe$ is \emph{$p$-congruent (relative to $s$ and $t$)} if $e_0(\Sigma^D_s(i)) = e_1(\Sigma^D_t(i))$ holds for every $i<p$,

    Then we say $p\le \hat{p}_{s,t}$ is \emph{secure (relative to $s$ and $t$)} if $D\vDash s<_\cyrDe t$ holds for some $p$-congruent $\cyrDe$, then $D\vDash s<_\cyrDe t$ holds for every $p$-congruent $\cyrDe$. In other words, the validity of $D\vDash s<_\cyrDe t$ does not depend on the choice of a $p$-congruent $\cyrDe$.
\end{definition}

Note that a secure number between two terms always exists:
\begin{proposition}
    Let $D$ be a predilator and $s,t\in \field(D)$.
    Then $\hat{p}_{s,t}$ is secure relative to $s$ and $t$.
\end{proposition}

\begin{definition}
    Let $D$ be a predilator.
    Define $\bfp^D(s,t)$ be the least secure number between $s$ and $t$.
    We also define $\varepsilon^D_{s,t}\in \{+1,-1\}$ by $\varepsilon^D_{s,t}=+1$ if and only if $D\vDash s<_\cyrDe t$ for every $\bfp^D(s,t)$-congruent arity diagram $\cyrDe$ between $s$ and $t$.
\end{definition}

$\bfp^D$ and $\varepsilon^D$ determine the structure of a predilator in the following sense:
\begin{theorem} \label{Theorem: Fundamental theorem of predilator comparsion}
    For every linear order $X$ and $a\in [X]^{\arity s}$, $b\in [X]^{\arity t}$,
    $D(X)\vDash s(a)<t(b)$ if and only if either
    \begin{enumerate}
        \item $a(\Sigma^D_s(i))= b(\Sigma^D_t(i))$ for every $i<\bfp^D(s,t)$ and $\varepsilon^D_{s,t}=+1$, or
        \item There is $j<\bfp^D(s,t)$ such that $a(\Sigma^D_s(j)) < b(\Sigma^D_t(j))$ and $a(\Sigma^D_s(i))= b(\Sigma^D_t(i))$ for every $i<j$.
    \end{enumerate}
\end{theorem}

The next proposition says $\bfp^D$ behaves like an ultrametric and $\varepsilon^D$ gives a linear order:
\begin{proposition}
    For $s,t,u\in \field(D)$, $\bfp^D(s,u) \ge \min(\bfp^D(s,t),\bfp^D(t,u))$.
    If $\varepsilon^D_{s,t}=\varepsilon^D_{t,u}=+1$, then $\varepsilon^D_{s,u}=+1$ and $\bfp^D(s,u) = \min(\bfp^D(s,t),\bfp^D(t,u))$.
\end{proposition}

Hence if we define $<^D$ by $s<^D t$ iff $\varepsilon^D_{s,t}=+1$, then $<^D$ is a linear order over $\field(D)$ satisfying the following:
For $s,t,u\in \field(D)$, if $s \le^D t\le^D u$, then $\bfp^D(s,u)=\min(\bfp^D(s,t),\bfp^D(t,u))$.
It turns out that $\Sigma^D$, $\bfp^D$, and $<^D$ completely determine the structure of a predilator as Girard \cite{Girard1982Logical} proved:
\begin{theorem}[Abstract construction of predilators]
    Every predilator is characterized by the following data:
    \begin{enumerate}
        \item The domain set $X$ with a linear order $<^X$.
        \item A function $\bfp\colon X\times X\to \bbN$ such that 
        \begin{itemize}
            \item $\bfp(x,y)=\bfp(y,x)$ for $x,y\in X$.
            \item $\bfp(x,z)=\min(\bfp(x,y),\bfp(y,z))$ for $x,y,z\in X$ such that $x\le^X y\le^X z$.
        \end{itemize}
        \item For each $x\in X$, a permutation $\Sigma_x$ over $\bfp(x,x)$ such that 
        \begin{equation*}
            \forall i,j<\bfp(x,y) [\Sigma_x(i)<\Sigma_x(j) \iff \Sigma_y(i)<\Sigma_y(j)].
        \end{equation*}
    \end{enumerate}
\end{theorem}
Note that the above characterization does not apply to general semidilators. See \cite[8.G.2]{Girard1982Logical} for the corresponding theorem for semidilators.

\subsection{Flowers} \label{Subsection: Flowers}
We may understand dilators as functions `expanding' a given ordinal. However, a dilator may add new elements in the middle of the ordinal, which makes the dilator not `continuous': For example, $D(X)=X+X$ is a dilator, and $D(n)=n+n$ is finite. However, $D(\omega)=\omega+\omega$.
A `continuous' dilator has a crucial role in this paper, and Girard named it a \emph{flower}:
\begin{definition}
    A semidilator $D$ is a \emph{semiflower} if for every linear order $Y$ and its initial segment $X\subseteq Y$ (i.e., a downward closed suborder), $D(X)$ is also an initial segment of $D(Y)$.
    A semiflower $D$ is a \emph{(pre)flower} if $D$ is a (pre)dilator.
\end{definition}

(Semi)flowers add new elements at the end of a linear order and not in the middle, so we may ask if (semi)flowers take the form of the sum $D(\alpha)=\gamma+\sum_{\beta<\alpha} E(\beta)$. We will show that it really is, but let us first define the `sum' $\sum_{\beta<\alpha} E(\beta)$ as a semidilator as follows:

\begin{definition}
    For a semidilator $D$, let us define $\int D$ as follows: The set of $\int D$-terms is $\{t^{\int}\mid t\in \field(D)\}$ with $\arity^{\int D} t^{\int} = \arity^D t + 1$. For an arity diagram $\cyrDe$ of the form 
    \begin{equation} \label{Ch3: Formula: Arity diagram in Chapter 3-02}
        \cyrDe =
        \begin{tikzcd}
            n_\cap & n_1 \\
            n_0 & n_\cup 
            \arrow[from=1-1, to=1-2] \arrow[from=2-1, to=2-2, "e_0"']
            \arrow[from=1-1, to=2-1] \arrow[from=1-2, to=2-2, "e_1"]
        \end{tikzcd}
    \end{equation}
    let us define the comparison rule $t_0^{\int} <_\cyrDe t_1^{\int}$ if and only if
    \begin{enumerate}
        \item Either $\max e_0 < \max e_1$, or
        \item If $\max e_0 = \max e_1$, and if $\cyrDe^- = \Diag_{n_\cup}(\ran e_0\setminus \{\max e_0\}, \ran e_1\setminus \{\max e_1\})$, then $D\vDash t_0<_{\cyrDe^-} t_1$.
    \end{enumerate}
\end{definition}

The main idea of the definition is using the largest component of the new term $t^{\int}$ as an indicator of where the term comes from among copies of $D(\beta)$ for some $\beta<\alpha$.
We can see that for a linear order $X$ and a semidilator $D$, $(\int D)(X)$ is isomorphic to $\sum_{x\in X} D(X\restriction x)$, where $X\restriction x = \{y\in X\mid y<x\}$. We can also see that $\int D$ is a semiflower:
\begin{proposition}
    For a semidilator $D$, $\int D$ is a semiflower. For a semidilator embedding $f\colon D\to E$, if we define $\int f\colon \int D\to \int E$ by $(\int f)(t)=f(t)^{\int}$, then $\int f$ is also a semidilator morphism.
    Furthermore, if $D$ is a (pre)dilator, then $\int D$ is a (pre)flower.
\end{proposition}

Conversely, we can `differentiate' a semidilator as follows:
\begin{lemma} 
    Let $D$ be a semiflower. Let us define a new structure $\partial D$ as follows:
    The field of $\partial D$ is $\{t^\partial\mid t\in \field(D)\land \arity^D(t)\ge 1\}$, and $\arity^{\partial D}(t^\partial) = \arity^D(t)-1$.
    For each arity diagram $\cyrDe$ for $t_0^\partial$, $t_1^\partial$ of the form \eqref{Ch3: Formula: Arity diagram in Chapter 3-02}, let us define $\cyrDe^+ = \Diag(\ran e_0\cup\{n_\cup\},\ran e_1\cup\{n_\cup\})$.
    Then define $t^\partial_0 <_\cyrDe t^\partial_1$ iff $D\vDash t_0 <_{\cyrDe^+} t_1$.

    Then $\partial D$ is a semidilator. Furthermore, if $D$ and $E$ are semiflowers and $f\colon D\to E$, the map $\partial f\colon \partial D\to \partial E$ given by $(\partial f)(t^\partial) = (f(t))^\partial$ is a semidilator embedding.
\end{lemma}

The following theorem is the promised characterization for a semiflower:
\begin{theorem}
    Let $D$ be a semiflower. If $\mathsf{Init}(D)$ is the linear order given by the nullary $D$-terms, then $D\cong \mathsf{Init}(D) + \int(\partial D)$.
\end{theorem}
Then by the definition of $\int D$, we have the following different characterization of a semiflower in terms of the denotation system:
\begin{corollary} \label{Corollary: First order characterization of a semiflower}
    A semidilator $D$ is a semiflower iff for every $s,t\in \field(D)$ and an arity diagram $\cyrDe$ for $s,t$ of the form \eqref{Ch3: Formula: Arity diagram in Chapter 3-02}, we have
    \begin{enumerate}
        \item If $\arity s=0<\arity t$, then $s<_\cyrDe t$.
        \item If $\arity s,\arity t>0$ and $\max e_0<\max e_1$, then $s<_\cyrDe t$.
    \end{enumerate}
\end{corollary}

The following proposition shows that $\int$ and $\partial$ are inverses of each others:
\begin{proposition}
    Let $D$ be a semidilator. Then $D\cong \partial(\int D)$.
\end{proposition}

\subsection{Dendrograms}
A dendroid is a tree-like structure representing $D(\alpha)$ for a dilator $D$ and a well-order $\alpha$. The original notation of a dendroid given by Girard \cite{Girard1981Dilators} is a mixture of a dilator $D$ and a well-order $\alpha$, and its definition does not directly allow its pre- notion.
A dendrogram can be viewed as separating the dilator part from Girard's dendroid, so it solely captures the structure of a dilator. It appears during the construction of a measure family for the Martin flower, and the author believes a dendrogram is the best way to construct dilators by hand. Note that a relevant notion appeared in \cite{AguileraFreundPakhomov?Shallow} under the name \emph{cell decomposition}.
\begin{definition}
    A \emph{predendrogram} is a structure $C=(C,<,\multimap,\bfe)$ such that the following holds:
    \begin{enumerate}
        \item $(C,\multimap)$ is a forest with the immediate successor relation $\multimap$.
        Moreover, if $\multimap^*$ is the transitive closure of $\multimap$, then for each $x\in C$, $\{y\in C\mid y\multimap^* x\}$ is finite and well-ordered by $\multimap^*$. We call the size of $\{y\in C\mid y\multimap^* x\}$ the \emph{length of $x$} and denote it by $\lh(x)$.

        \item For every $x\in C$, either $x$ is a terminal point (i.e., $x$ has no immediate successor) or there is $y$ such that $x\multimap^* y$ and $y$ is a terminal point.

        \item $<$ is a partial order over $C$. Moreover, the following sets are linearly ordered by $<$:
        \begin{itemize}
            \item The set of \emph{roots} of $C$, i.e. elements with no immediate predecessor.
            \item The set immediate successors of $x$ for each $x\in C$.
        \end{itemize}
        
        \item $\bfe$ is a partial function from $C$ to $\bbN$ such that $\bfe(x)$ is defined if and only if $x$ is not a terminal point, and $\bfe(x)\le \lh(x)$.
    \end{enumerate}
    We denote the set of terminal points of $C$ by $\term(C)$. We also define the \emph{sequence $\pred(x)$ of predecessors of $x$} as the $\multimap$-increasing enumeration of $\{y\in C\mid y\multimap^* x\} \cup\{x\}$. $\pred(x)$ is a sequence of length $\lh(x)+1$.
\end{definition}

Each predendrogram induces a predilator in the following way:

\begin{definition} \label{Definition: Predendrogram to an alpha-predendroid}
    Let $C$ be a predendrogram, $\alpha$ a linear order.
    Let us define $C(\alpha)$ by the set of all $\lag x_0,\xi_0,\cdots,\xi_{m-1},x_m\rag$ such that there is $x=x_m\in \term(C)$ such that $\pred(x) = \lag x_0,\cdots, x_m\rag$ and for each $i<m$, $\xi_i$ is the $\bfe(i)$th least element over $\{\xi_0,\cdots,\xi_i\}$. We impose $C(\alpha)$ on the Kleene-Brouwer order, where we compare $x_i$ by the $C$-order, and $\xi_i$ by the $\alpha$-order. For $f\colon \alpha\to\beta$, we define $C(f)$ by
    \begin{equation*}
        C(f)(\lag x_0,\xi_0,\cdots,\xi_{m-1},x_m\rag) = \lag x_0,f(\xi_0),\cdots,f(\xi_{m-1}),x_m\rag.
    \end{equation*}
\end{definition}
We can see that $\alpha\mapsto (C(\alpha),<_\KB)$ is a predilator (more precisely, an F-predilator.) We will see later how to `decode' a predendrogram into a predilator as a denotational system.
Like predendroids, predendrograms also admit morphisms:
\begin{definition}
    For two predendrograms $C$ and $D$, a function $f\colon C\to D$ is a \emph{predendrogram morphism} if it preserves $<$, $\multimap$, and $\bfe$. That is, for $x,y\in C$
    \begin{enumerate}
        \item $C\vDash x\multimap y$ iff $D\vDash f(x)\multimap f(y)$.
        \item $C\vDash x<y$ iff $D\vDash f(x)<f(y)$, and
        \item $f(\bfe^C(x)) = \bfe^D(f(x))$.
    \end{enumerate}
\end{definition}

We can turn a predendrogram into a predilator as follows, which also gives a functor from the category of predendrograms to the category of predilators:
\begin{definition}
    Let $C$ be a predendrogram. Let us define $\Dec(C)$: Its field equals $\term(C)$. We define the comparison rule of $\Dec(C)$ in a way that the following holds:
    \begin{enumerate}
        \item $<^{\Dec(C)}$ is equal to the Kleene-Brouwer order over $C$.
        \item For $x,y\in C$, $\bfp^{\Dec(C)}(x,y)=m$, where $m$ is the least natural number such that $\pred(x)\restriction (m+1)\neq \pred(y)\restriction (m+1)$. 
        \item For $x\in C$, $\Sigma^{\Dec(C)}_\sigma$ is a permutation over $m=\lh(x)$ satisfying the following:
        For $\pred(x) = \lag x_0,x_1,\cdots,x_m\rag$ and $e_i = \bfe(x_i)$,
        $\Sigma^{\Dec(C)}_\sigma(i)$ is the $e_i$th least element of $\{\Sigma^{\Dec(C)}_\sigma(j)\mid j\le i\}$ for every $i<m$.
    \end{enumerate}
    For a predendrogram morphism $f\colon C\to D$, let us define $\Dec(f)=f$.
\end{definition}

Conversely, from a predilator $D$, we can get the corresponding predendrogram $\Cell(D)$, called the \emph{cell decomposition} of $D$. Its construction is similar to that of the Branching functor in \cite{Girard1981Dilators}, which is given as follows. First, let us find the field of $\Cell(D)$:

\begin{definition}
    For a predilator $D$ and $n\in\bbN$, let us define a equivalence relation $\equiv^D_n$ over $\field(D)$ as follows:
    \begin{equation*}
        s\equiv^D_n t \iff \bfp^D(s,t)>n \lor s=t.
    \end{equation*}
\end{definition}
Let $<^D$ be a linear order derived from the abstract construction of predilators. Then we can see that $\equiv^D_n$ is an interval over $(D,<^D)$:

\begin{lemma}
    An $\equiv^D_n$-equivalence class is an interval on $(D,<^D)$.
\end{lemma}
\begin{proof}
    Let $s,t,u\in \field(D)$, be such that $s<^D t<^D u$ and $\bfp^D(s,u)>n$.
    From $\min(\bfp^D(s,t),\bfp^D(t,u))=\bfp^D(s,u)>n$, we have that $s,t,u$ are all $\equiv^D_n$-equivalent.
\end{proof}
Hence we can define $(D,<^D)/\equiv^D_n$. Now, let us define the cell decomposition of $D$ as follows:
\begin{definition}
    For a predilator $D$, $\Cell(D)$ is the set of $[t]_{\equiv^D_m}$ for every $t\in \field(D)$ and $m\le \arity(t)$.
    We define relations over $\Cell(D)$ as follows: For $x,y\in \Cell(D)$,
    \begin{enumerate}
        \item $x\multimap y$ iff there is $t\in \field(D)$ and $m<\arity(t)$ such that $x=[t]_{\equiv^D_m}$ and $y=[t]_{\equiv^D_{m+1}}$.
        \item $x<y$ iff there are $s,t\in\field(D)$ and $m\le \arity s,\arity t$ such that $x=[s]_{\equiv^D_m}$, $y=[t]_{\equiv^D_m}$, and $s<^D t$.
        \item $\bfe([t]_{\equiv^D_m})$ is the natural $e$ number such that $\Sigma^D_t(m)$ is the $e$th least member over the set $\{\Sigma^D_t(0),\cdots,\Sigma^D_t(m)\}$.
    \end{enumerate}
    For an embedding $f\colon D\to E$, let us define $\Cell(f)$ by $\Cell(f)([t]_{\equiv^D_m}) = [f(t)]_{\equiv^E_m}$.
\end{definition}

Then we can see that $\Cell$ and $\Dec$ form category equivalences:
\begin{theorem}
    $\Cell$ and $\Dec$ are category equivalences between the category of predendrograms and the category of predilators.
\end{theorem}

Let us provide the characterization of preflowers in terms of predendrograms:
\begin{proposition} \label{Proposition: Dendrogram characterization of a flower}
    A predilator $D$ is a preflower if and only either every $x\in \Cell(D)$ has length 0, or if there is $x^*\in \Cell(D)$ of length 0 such that for every $x\in\Cell(D)$, we have either
    \begin{enumerate}
        \item $\lh(x)=0$ and $\Cell(D)\vDash x<x^*$, or
        \item $\lh(x)>0$, $x^*$ occurs in the sequence of predecessors of $x$, and $\bfe^{\Cell(D)}(x)<\lh(x)$ if $\bfe^{\Cell(D)}(x)$ is defined.
    \end{enumerate}
\end{proposition}
\begin{proof}[Sketch of proof]
    We only consider the case when $D$ has a non-nullary term.
    For one direction, suppose that $D$ is a preflower and $t\in \field(D)$.
    If $\arity t>0$, then $\Sigma^D_t(0) = \arity t -1$. Hence $\Sigma^D_t(i)$ is an $<i$th least element of $\{\Sigma^D_t(j)\mid j\le i\}$.
    Also, if $\arity s=0$ and $s\in \field(D)$, then $D\vDash s<_\cyrDe t$ for every arity diagram $\cyrDe$ between $s$ and $t$, which implies $[s]_{\equiv^D_0} < [t]_{\equiv^D_0}$ by \autoref{Corollary: First order characterization of a semiflower}. If $t'\in\field(D)$ is another term of arity $>0$, then again \autoref{Corollary: First order characterization of a semiflower} implies $\bfp^D(t,t')>0$, so we can take $x^*=[t]_{\equiv^D_0}$.
    For the other direction, from the given assumption, we have that $\Cell(\Dec(D))\cong D$ satisfies the following:
    \begin{enumerate}
        \item If $\arity^{\Cell(\Dec(D))}(x)=0<\arity^{\Cell(\Dec(D))}(y)$, then $x <^{\Cell(\Dec(D))} y$.
        \item If $\arity^{\Cell(\Dec(D))}(x)>0$, then $\Sigma^{\Cell(\Dec(D))}_s(0) = \arity(x)-1$ for every $x\in \Cell(\Dec(D))$.
    \end{enumerate}
    Hence $D$ is a preflower by \autoref{Corollary: First order characterization of a semiflower}.
\end{proof}

In the latter part of the paper, we iterate a measure over a finite dendrogram. Dendrograms are not linear, so we need to specify the order over the dendrogram before the iteration.
It turns out that the following type of dendrogram includes a correct iteration order:
\begin{definition} \label{Definition: Trekkable dendrogram}
    A predendrogram $D$ is \emph{trekkable} if
    \begin{enumerate}
        \item The field of $D$ is an ordinal. 
        \item For each $x,y\in D$, if $D\vDash x<y$ or $D\vDash x\multimap y$, then $x$ is less than $y$ as ordinals. That is, the predecessor relation or a comparison relation over $D$ respects an ordinal order.
    \end{enumerate}
    A function $f\colon D\to D'$ between two trekkable predendrogram is a \emph{trekkable predendrogram morphism} if $f$ is an ordinal order-preserving predendrogram morphism.
\end{definition}
In particular, the domain of a finite dendrogram is a natural number. Most trekkable predendrograms we care about are finite, although we will see a countable trekkable dendrogram in \autoref{Subsection: omega1 completeness of the measure family}.

We may think of a dendrogram as a tree structure also showing `hidden terms' (or non-terminal terms) of a dilator. We will later associate a measurable dilator term to each node in a finite dendrogram, even for non-terminal ones.
Hence, it is convenient to consider the `closure' of a dendrogram exhibiting every hidden term.
\begin{definition}
    Let $d$ be a predendrogram. We define a predilator $\Dec^\bullet(d)$ as follows: For each linear order $\alpha$, $\Dec^\bullet(d)(\alpha)$ is the set of all $\lag x_0,\xi_0,\cdots,\xi_{m-1},x_m\rag$ such that there is $x=x_m\in d$ such that $\pred(x) = \lag x_0,\cdots, x_m\rag$ and for each $i<m$, $\xi_i$ is the $\bfe(i)$th least element over $\{\xi_0,\cdots,\xi_i\}$. We impose $\Dec^\bullet(d)(\alpha)$ on the Kleene-Brouwer order, where we compare $x_i$ by the $d$-order, and $\xi_i$ by the $\alpha$-order. For $f\colon \alpha\to\beta$, we define $\Dec^\bullet(d)(f)$ by
    \begin{equation*}
        \Dec^\bullet(d)(f)(\lag x_0,\xi_0,\cdots,\xi_{m-1},x_m\rag) = \lag x_0,f(\xi_0),\cdots,f(\xi_{m-1}),x_m\rag.
    \end{equation*}
\end{definition}
The previous definition is the same as \autoref{Definition: Predendrogram to an alpha-predendroid}, except that in $\Dec^\bullet(d)$, we also allow non-terminal $x$.
The following definition will give a dendrogram representation for the closure of $d$:
\begin{definition}
    Let $d$ be a predendrogram. We define $d^\bullet$ by the disjoint union of $\{x^\bullet \mid x\in d\}$ and the set of all non-terminal points of $d$.
    We define $\multimap$, $<$, and $\bfe$ over $d^\bullet$ as follows:
    \begin{enumerate}
        \item For $x,y\in d$, we have $d^\bullet \vDash x\multimap y$ (if $y$ is not terminal in $d$) and $d^\bullet \vDash x\multimap y^\bullet$. 
        \item For $x,y\in d$ with $d\vDash x<y$, we have $d^\bullet \vDash x<x^\bullet<y<y^\bullet$. We ignore undefined elements from the defining inequality.
        \item For a non-terminal $x\in d$, $\bfe^{d^\bullet}(x)=\bfe^d(x)$.
    \end{enumerate}
\end{definition}
We can describe the construction of $d^\bullet$ as follows: Starting from $d$, we put all terminal nodes with a bullet. Then, for every intermediate node $x$, let us add a new node $x^\bullet$ with the same immediate predecessor just to the right side of $x$. The new nodes represent unraveled intermediate nodes in $d$. See \autoref{Figure: Bullet dendrogram construction} for an example of the construction.
$x^\bullet$ is always a terminal in $d$, so $\bfe(x^\bullet)$ is undefined. 

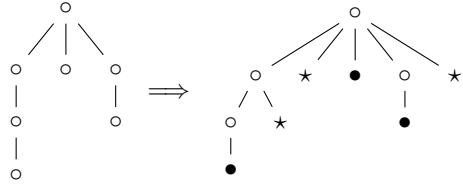
\begin{figure}
    \centering
    \begin{center}
    \adjustbox{valign=c}{
    \begin{forest} for tree={fit=tight, for children={l sep-=0.8em,l-=0.8em}}
    [$\circ$
        [$\circ$[$\circ$[$\circ$]]] 
        [$\circ$]
        [$\circ$[$\circ$]]
        ]
    \end{forest}}
    $\implies$
    \adjustbox{valign=c}{
    \begin{forest} for tree={fit=tight, for children={l sep-=1em,l-=1em}}
    [$\circ$
        [$\circ$[$\circ$[$\bullet$]][$\star$]] [$\star$]
        [$\bullet$]
        [$\circ$[$\bullet$]] [$\star$]
        ]
    \end{forest}}
    \end{center}
    \caption{The construction of $d^\bullet$ from $d$. Filled circle nodes represent terminal nodes in $d$, and filled starred nodes represent unraveled intermediate nodes in $d$.}
    \label{Figure: Bullet dendrogram construction}
\end{figure}

\begin{proposition}
    For a predendrogram $d$, $\Dec^\bullet(d)$ and $\Dec(d^\bullet)$ are isomorphic.
\end{proposition}
\begin{proof}
    For a linear order $\alpha$, we can see that $\iota_\alpha\colon \Dec^\bullet(d)(\alpha)\to \Dec(d^\bullet)(\alpha)$ defined by
    \begin{equation*}
        \iota_\alpha (\lag x_0,\xi_0,\cdots, \xi_{m-1},x_m\rag) = \lag x_0,\xi_0,\cdots, \xi_{m-1},x_m^\bullet\rag
    \end{equation*}
    {} is an isomorphism natural in $\alpha$.
\end{proof}

The following theorem says we can decompose comparison relations over a dilator into simpler ones.

\begin{theorem}[Elementary comparison decomposition theorem] \label{Theorem: Elementary comparison decomposition}
    Let $d$ be a predendrogram.
    Every comparison relation $s<_{\cyrDe} t$ over $\Dec^\bullet(d)$ is decomposed into the following \emph{elementary comparison rules}:
    If we fix $\pred(s) = \lag s_0,\cdots, s_{\lh s}\rag$ and $\pred(t) = \lag t_0,\cdots, t_{\lh t}\rag$,
    \begin{enumerate}
        \item[(A)] $t<_{\Diag_\bbN(b,a)}s$ for $s\multimap t$ and $a\subseteq b\subseteq \bbN$.
        
        \item[(B)] $s <_{\Diag_\bbN(a,a)} t$ when $s$ and $t$ have the same predecessor and $s < t$.
        
        \item[(C)] $s <_{\Diag_\bbN(a,b)} t$ when $s$ and $t$ have the same predecessor, $a_i = b_i$ for every $i<\lh s-1$, and $a_{\lh s-1} < b_{\lh t-1}$ (Note that $\lh s= \lh t$.)
        
        \item[(D)] $s <_{\Diag_\bbN(a,b)} t$ when $\lh s + 1 = \lh t$, $\pred(t) \restriction \lh s = \pred s$, $a\subseteq b\subseteq\bbN$, $a_i=b_i$ for $i<\lh s$, and $s < t_{\lh s}$. 
    \end{enumerate}
    Here we enumerate $a=\{a_0,\cdots,a_{\lh s-1}\}$ with respect to $\prec_s$, i.e., in the way that $a_i < a_j$ iff $i \prec_s j$ and similar to $b=\{ b_0,\cdots, b_{\lh t-1}\}$.
    In particular, if the map $f\colon \Dec^\bullet(d) \to \Omega$ preserves every elementary comparison relation, then $f$ preserves every comparison relation.
    
    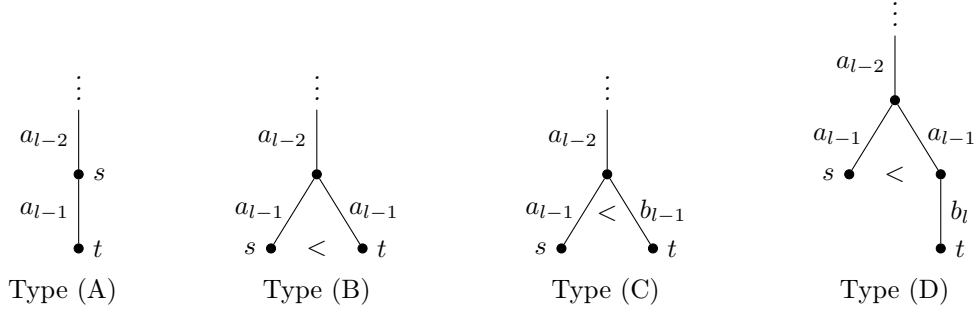
\begin{figure}
        \centering
        \tikzset{
            solid node/.style={circle,draw,inner sep=1.2,fill=black}
        }
        \setlength{\tabcolsep}{2em}
        \begin{tabular}{ c c c c }
        {\begin{forest}
            [$\vdots$
                [,for tree=solid node, edge label={node[left, midway] {$a_{l-2}$}}, label = {right:$s$}
                    [,for tree=solid node, edge label={node[left, midway] {$a_{l-1}$}}, label = {right:$t$}]
                ]
            ]
        \end{forest}}&
        {\begin{forest}
            [$\vdots$
                [,for tree={s sep=3em, solid node}, edge label={node[left, midway] {$a_{l -2}$}}, label = {[label distance=1.9em]below:$<$}
                    [,for tree=solid node, edge label={node[left, midway] {$a_{l -1}$}}, label = {left:$s$}]
                    [,for tree=solid node, edge label={node[right, midway] {$a_{l -1}$}}, label = {right:$t$}]
                ]
            ]
        \end{forest}} &
        {\begin{forest}
            [$\vdots$ 
                [,for tree={s sep=3em, solid node}, edge label={node[left, midway] {$a_{l -2}$}}, label = {[label distance=0.6em]below:$<$}
                    [,for tree=solid node, edge label={node[left, midway] {$a_{l -1}$}}, label = {left:$s$}]
                    [,for tree=solid node, edge label={node[right, midway] {$b_{l -1}$}}, label = {right:$t$}]
                ]
            ]
        \end{forest}} &
        {\begin{forest}
            [$\vdots$ 
                [,for tree={s sep=3em, solid node}, edge label={node[left, midway] {$a_{l -2}$}}, label = {[label distance=1.9em]below:$<$}
                    [,for tree=solid node, edge label={node[left, midway] {$a_{l -1}$}}, label = {left:$s$}]
                    [,for tree=solid node, edge label={node[right, midway] {$a_{l -1}$}}
                        [,for tree=solid node, edge label={node[right, midway] {$b_{l}$}}, label = {right:$t$}
                        ]
                    ]
                ]
            ]
        \end{forest}} \\
        \textnormal{Type (A)} & \textnormal{Type (B)} & \textnormal{Type (C)} & \textnormal{Type (D)}
        \end{tabular}
        \caption{Elementary comparison relations}
    \end{figure}
    \label{Figure: Elementary comparison relations}
\end{theorem}
\begin{proof}
    Suppose that $\pred(s) = \lag s_0,s_1,\cdots,s_{\lh s}\rag$ and $\pred(t) = \lag t_0,t_1,\cdots,t_{\lh t}\rag$, and $a\in [\bbN]^{\lh s}$, $b\in [\bbN]^{\lh t}$. Now suppose that $d^*(\bbN)\vDash s(a)<t(b)$: We have the following three possible cases:
    \begin{enumerate}
        \item $\pred(s) \supsetneq \pred(t)$ and $a\supsetneq b$.
        \item There is $m\le \min(\lh s,\lh t)$ such that for every $i<m$, $a_i=b_i$ and $s_i=t_i$, but $s_m<t_m$.
        \item There is $m\le \min(\lh s,\lh t)$ such that for every $i<m-1$, $a_i=b_i$ and $s_i=t_i$, $s_{m-1}=t_{m-1}$, but $a_{m-1}<b_{m-1}$.
    \end{enumerate}
    The first case is easily decomposed into a series of Type (A) comparison relations.
    In the latter two cases, let us observe that $s(a) < s_m(\{a_0,\cdots,a_{m-1}\}) < t(b)$, and the first comparison is Type (A). Hence, we may assume $m=\lh s$ in the latter two cases.
    
    Now let us consider the second case with $m=\lh s$. We have Type (B) if $m=\lh t$. We claim that if $m<\lh t$, then we can decompose the comparison $s(a)<t(b)$ into comparisons of Type (D) by induction on $\lh t\ge m$:
    The case $\lh t = m+1$ is Type (C). 
    Now consider the case $\lh t > m+1$.
    Since $a$ and $b$ describe the relative position of parameters, we may assume that every component of $a$ and $b$ is a non-zero even number.
    If we take $b'=\{b_0,b_1,\cdots, b_{\lh t - 2}-1)\}$, then $s(a) < t_{\lh t-1}(b') < t(b)$. The first comparison $s(a) < t_{\lh t-1}(b')$ is the second case with $\lh(t_{\lh t-1}) < \lh(t)$, which is decomposed into comparisons of type (D) by the inductive hypothesis. The second comparison $t_{\lh t-1}(b') < t(b)$ is of Type (D). 

    Similarly, let us consider the third case with $m=\lh s$. We have Type (C) if $m=\lh t$. For $m<\lh t$, we claim that we can decompose $s(a)<t(b)$ into Type (D) and (C) by induction on $\lh t\ge m$: 
    We may also assume that every component of $a$ and $b$ is a non-zero even number. Let us take $b' = \{b_0,\cdots, b_{m-1}-1\}$.
    Then $s(a) < t_{\lh s}(b') < t(b)$, and the first comparison $s(a) < t_{\lh s}(b')$ is of Type (C).
    If $\lh t = m+1$, the second comparison $t_{\lh s}(b') < t(b)$ is of type (D). If $\lh t>m+1$, we can further decompose $t_{\lh s}(b') < t(b)$ into a comparison of type (D) and (C) by the inductive hypothesis.
\end{proof}

We finish this section with the following lemma we will apply:
\begin{lemma} \label{Lemma: Lemma 3.39}
    Suppose that $d$ is a dendrogram and $x,y\in d$. If $\lh x=\lh y=m$and $\Dec^\bullet(d)\vDash x(m)<y(m)$, then either $(x'\multimap y \land x<y)$ or $\Dec^\bullet(d)(\omega)\vDash x'(m\setminus \{\bfe(x')\}) < y(m)$.
\end{lemma}
\begin{proof}
    Let $\pred(x)=\lag x_0,\cdots, x_m\rag$ and $\pred(y)=\lag y_0,\cdots, y_m\rag$. If $q<m$ is the least number such that $x_q\neq y_q$, then $q=\bfp^{\Dec^\bullet(d)}(x,y)$ and one of the following holds by \autoref{Theorem: Fundamental theorem of predilator comparsion}:
    \begin{enumerate}
        \item There is $p<q$ such that $\Sigma^{\Dec^\bullet(d)}_x(p) < \Sigma^{\Dec^\bullet(d)}_y(p)$.
        \item $\Sigma^{\Dec^\bullet(d)}_x(p) = \Sigma^{\Dec^\bullet(d)}_y(p)$ for every $p<q$ and $x_q<y_q$.
    \end{enumerate}
    If $q=m$, then $x$ and $y$ have the same immediate predecessor, and the second case holds, so $x<y$.
    Otherwise, if we let $a = m\setminus \{\bfe(x')\}$, then $a\big(\Sigma^{\Dec^\bullet(d)}_{x'}(p)\big) = \Sigma^{\Dec^\bullet(d)}_x(p)$ for every $p<m-1$.    
    Hence, each case implies the following 
    \begin{enumerate}
        \item There is $p<q$ such that $a\big(\Sigma^{\Dec^\bullet(d)}_{x'}(p)\big) < \Sigma^{\Dec^\bullet(d)}_y(p)$, or 
        \item $a\big(\Sigma^{\Dec^\bullet(d)}_{x'}(p)\big) = \Sigma^{\Dec^\bullet(d)}_y(p)$ for every $p<q$ and $x_q<y_q$.
    \end{enumerate}
    But in either case, we have $\Dec^\bullet(d)(\omega)\vDash x'(a)<y(m)$.
\end{proof}

\section{Measurable dilator} \label{Section: Measurable dilator}

A measurable dilator is a dilator analogue of a measurable cardinal defined by Kechris \cite{KechrisUnpublishedDilators}. Like the existence of a measurable cardinal proves $\bfPi^1_1$-determinacy, the existence of a measurable dilator proves $\bfPi^1_2$-determinacy.
In this section, we define and examine the properties of a measurable dilator. 

\subsection{Universal dilator and measurable dilator}
Let us start with the following question: Do we have a dilator embedding every countable dilator? It is like asking if there is a dilator analogue of $\omega_1$, a well-order that embeds every countable well-order. 

\begin{definition}
    A dilator $D$ is \emph{universal} if $D$ embeds every countable dilator.
\end{definition}

$\ZFC$ proves there is a universal dilator: Let us enumerate $\{D_\alpha\mid \alpha<\frakc\}$ of every dilator whose field is $\bbN$, and take the ordered sum $\sum_{\alpha<\frakc} D_\alpha$.
However, the resulting universal dilator is far from being definable. Indeed, Kechris proved that $\ZFC$ does not prove there is an ordinal definable universal dilator:

\begin{proposition}[Kechris {\cite{KechrisUnpublishedDilators}}]
    $\ZFC$ does not prove there is an ordinal definable universal dilator.
\end{proposition}
\begin{proof}
    Suppose that $\ZFC$ proves there is an ordinal definable universal dilator $\Omega$, and let us work over $\ZFC$. Suppose that $D_x$ is a recursive dilator with a real parameter $x$.
    Now let us consider the tree $T$ trying to construct a real and an embedding $D_x\to \Omega$ as follows: $T$ is a tree over $\omega\times \field(\Omega)$, and 
    \begin{equation*}
        \lag (s_0,t_0),\cdots,(s_{m-1},t_{m-1})\rag \in T
    \end{equation*}
    if the following holds: For $i,j<m$, suppose that $\vec{s}=\lag s_0,\cdots, s_{m-1}\rag$ is long enough to determine $i,j\in \field(D_{\underline{\vec{s}}})$, the arity of $i$ and $j$ as $D_{\underline{\vec{s}}}$-terms.
    Furthermore, assume that $\cyrDe$ is an arity diagram between $i$ and $j$ and $D_{\underline{\vec{s}}}$ can also see $i<_\cyrDe j$ holds. Then $t_i <_\cyrDe t_j$.
    (Note that we may turn $T$ into a ptyx, but it is unnecessary in our context.)

    It is clear that if $(x,\vec{t})$ is an infinite branch, then $D_x\le \Omega$. Conversely, we can turn an embedding $D_x\le \Omega$ into an infinite branch of $T$.
    Hence $D_x\le \Omega$ if and only if
    \begin{equation*}
        T_x = \{\lag t_0,\cdots, t_{m-1}\rag\mid \lag (x(0),t_0),\cdots, \lag (x(m-1),t_{m-1}) \rag\in T\}
    \end{equation*}
    has an infinite branch. Since $\Omega$ is universal and $D_x$ is always countable, we have $\Dil(D_x)$ if and only if $D_x$ embeds $\Omega$.
    Therefore $\Dil(D_x)$ if and only if $x\in p[T]$, and $T$ is ordinal definable since $\Omega$ is. Hence, every $\Pi^1_2$-set has an ordinal definable element, by taking the leftmost branch of $T$.
    
    However, the previous statement consistently fails over a generic extension of $L$ obtained by adding a Cohen real since the set
    \begin{equation*}
        X = \{r\in\bbR \mid \forall M [\text{$M$ is a transitive model of $\ZFC^- + (V=L)$ $\to$ $r$ is Cohen over $M$}]\}
    \end{equation*}
    is a $\Pi^1_2$ set of reals without an ordinal definable element.
\end{proof}

Meanwhile, we can find an ordinal definable universal dilator under `every real has a sharp.' We sketch its construction in the next subsection.
The next definition is what we promised at the beginning of the section:
\begin{definition}
    A universal dilator $\Phi$ is \emph{measurable} if for each finite dilator $d$ there is a countably complete measure $\mu_d$ over $\Phi^d$
    satisfying the following:
    \begin{enumerate}
        \item (Coherence) For each $f\colon d\to d'$, let $f^*\colon \Phi^{d'}\to \Phi^d$ be $f^*(p)=p\circ f$. ($\Phi^d$ is the set of embeddings from $d$ to $\Phi$.) Then we have
        \begin{equation*}
            X\in \mu_d \iff (f^*)^{-1}[X]\in \mu_{d'}.
        \end{equation*}
        
        \item ($\omega_1$-completeness) For a given countable dilator $D$ and a countable family $\{d_n\mid n<\omega\}$ of finite subdilators of $D$, if we have $X_n\subseteq \Phi^{d_n}$ and $X_n \in \mu_{d_n}$ for each $n<\omega$, then we can find an embedding $e\colon D\to \Phi$ such that $e\restriction d_n\in X_n$ for every $n<\omega$.
    \end{enumerate}
    We say $\Phi$ is \emph{half-measurable} if the measure $\mu_d$ is defined only for sets in $\bigcup_{x\in\bbR} \calP(\Phi^d)\cap A_x$, where $A_x$ is the least admissible set containing $\Phi$ and $x$.
\end{definition}

$\omega_1$-completeness has the following equivalent formulation:
\begin{proposition}
    Let $\Phi$ be a universal dilator with a measure family $\{\mu_d\}_d$ satisfying coherence. Then the following are equivalent:
    \begin{enumerate}
        \item $\Phi$ satisfies $\omega_1$-completeness.
        \item For a given family $\{d_n\mid n<\omega\}$ of finite dilators, embeddings $f_n\colon d_n\to D$, and $X_n\in \Phi^{d_n}$, we can find $e\colon D\to \Phi$ such that $e\circ f_n\in X_n$ for each $n<\omega$.
    \end{enumerate}
\end{proposition}
\begin{proof}
    For one direction, let $d_n'$ be the range of $f_n$. Then $\lag d_n'\mid n<\omega\rag$ is the sequence of subdilators of $D$. Then by the $\omega_1$-completeness, we can find $e\colon D\to \Phi$ such that $e\restriction d_n' \in (f_n^*)^{-1}[X_n]$ for every $n<\omega$. Hence $e\circ f_n\in X_n$ for each $n$, as desired. 
    The other direction follows by taking $f_n$ to be the inclusion map.
\end{proof}

\subsection{From a measurable flower to a measurable dilator}
It will turn out that the most natural way to define a universal dilator iby s iterating an ultrapower. The resulting `dilator' is a flower, so it cannot embed every countable dilator. However, the resulting flower still embeds every countable flower, so we can think of it as a \emph{universal flower}. We can extract a universal dilator from it by `differentiating' the universal flower.

\begin{proposition}
    Let $\Omega$ be a universal flower, i.e., $\Omega$ is a flower and embeds every countable flower. Then $\partial \Omega$ is a universal dilator.
\end{proposition}
\begin{proof}
    Let $D$ be a countable dilator, so $\int D$ is a countable flower.
    By universality, there is an embedding $f\colon \int D\to \Omega$, and we have $\partial f\colon \partial(\int D)\to \partial\Omega$.
    Since $D\cong \partial(\int D)$, we have the desired result.
\end{proof}

From the previous proposition, let us sketch how to construct an ordinal definable universal dilator from sharps of reals:
\begin{example}
    Let us work over $\ZFC$ with `every real has a sharp.'
    Let $F_x$ be a sharp flower for $x^\sharp$ defined in \cite{AguileraFreundRathjenWeiermann2022}.
    By \cite[Proposition 12]{AguileraFreundRathjenWeiermann2022}, every countable flower in $L[x]$ embeds into $F_x$. Now let us consider the system of flowers $\{F_x\mid x\in\bbR\}$ with a natural choice of embeddings provided in \cite[Lemma 11]{AguileraFreundRathjenWeiermann2022}, and consider its direct limit. 
    The resulting flower $F$ is ordinal definable and embeds every countable flower. Then consider $\partial F$, which is an ordinal definable universal dilator.
\end{example}

The next example does not precisely give an ordinal definable universal flower without an additional assumption (like $V=L[U]$), but let us include it to illustrate how the iterated ultrapower reveals the structure of a universal flower.
\begin{example}
    Let $\kappa$ be a measurable cardinal with a normal measure $U$.
    Then we can define the $\alpha$th iterate $\Ult^\alpha(V,U)$ with an embedding $j_\alpha\colon V\to \Ult^\alpha(V,U)$.
    Then let us define the flower $F$ by $F(\alpha) = \kappa_\alpha := j_\alpha(\kappa)$. To define $F(f)$ for an increasing $f\colon \alpha\to\beta$, let us observe the following fact \cite[Lemma 19.6]{Kanamori2008}, which can be thought of as that $\Ult^\alpha(V,U)$ can be decomposed into a `term part' and `indiscernibles':
    For every ordinal $\alpha$ and $x\in \Ult^\alpha(V,U)$, we have $m<\omega$, $h\colon [\kappa]^m\to V$, and $\gamma_0<\cdots<\gamma_{m-1}<\alpha$ such that $x=j_\alpha(h)(\kappa_{\gamma_0},\cdots,\kappa_{\gamma_{m-1}})$.
    Combining with \cite[Lemma 19.9]{Kanamori2008},
    we can see that if we define $\Ult^f(V,U)\colon \Ult^\alpha(V,U)\to\Ult^\beta(V,U)$ by
    \begin{equation*}
        \Ult^f(V,U)(j_\alpha(h)(\kappa_{\gamma_0},\cdots,\kappa_{\gamma_{m-1}})) =
        j_\beta(h)(\kappa_{f(\gamma_0)},\cdots,\kappa_{f(\gamma_{m-1})})
    \end{equation*}
    then $\Ult^f(V,U)$ is well-defined. Then set $F(f) = \Ult^f(V,U)\restriction F(\alpha)$.
    
    We need an appropriate support transformation to turn $F$ into a dilator. For $\xi<\kappa_\alpha$, let $m$ be the least natural number such that there are $h\colon [\kappa]^m\to V$ and $\gamma_0<\cdots<\gamma_{m-1}<\alpha$ such that $\xi=j_\alpha(h)(\kappa_{\gamma_0},\cdots,\kappa_{\gamma_{m-1}})$.
    From \cite[Lemma 19.9]{Kanamori2008} we can prove that $\gamma_0<\cdots<\gamma_{m-1}$ are uniquely determined from $\xi$, so we can define $\supp^F_\alpha(\xi) = \{\gamma_0,\cdots,\gamma_{m-1}\}$.
    The support condition is easy to verify.

    We have defined $F$ only for ordinals, but we can easily extend $F$ to other linear orders. To see $F$ is a flower, observe that $F(\alpha)$ is an initial segment of $F(\beta)$ if $\alpha<\beta$.
    To see $F$ is universal, let us prove the following stronger claim: If $D\in V_\kappa$ is a flower, then $D$ embeds into $F$.
    If $D\in V_\kappa$, then $D(\kappa)$ is isomorphic to $\kappa$. (See the proof of \autoref{Proposition: Martin Flower is Universal} for the reason.)
    Fix an isomorphism $g\colon D(\kappa)\to\kappa$, and consider $\iota_\alpha\colon D(\alpha)\to \kappa_\alpha$ given by
    \begin{equation*}
        \iota_\alpha(t(\xi_0,\cdots,\xi_{m-1})) = j_\alpha(g)(t(\kappa_{\xi_0},\cdots,\kappa_{\xi_{m-1}})).
    \end{equation*}
    Then we can see that $\iota\colon D\to F$ is a natural transformation.
\end{example}

So far we know that a universal flower induces a universal dilator. Is it the same for a `measurable flower?' The answer is affirmative:
\begin{proposition}
    Let us say a universal flower $\Omega$ is a \emph{measurable flower} if, for every finite flower $d$ with no nullary terms, there is a countably complete measure $\mu_d$ over $\Omega^d$ satisfying the coherence and $\sigma$-completeness for flowers.
    Then $\partial \Omega$ is a measurable dilator.
\end{proposition}
\begin{proof}
    Let $d$ be a finite dilator. Then $\int d$ is a finite flower with no nullary terms. Then let us define a measure $\nu_d$ over $(\partial\Omega)^d$ by
    \begin{equation*}
        X \in \nu_d \iff \{p\in \Omega^{\int d}\mid \partial p \circ \phi_d\in X\}\in \mu_{\int d},
    \end{equation*}
    where $\phi_d\colon d\to \partial(\int d)$ is the isomorphism natural in $d$.
    For coherence, suppose that $d,d'$ are finite dilators and $f\colon d\to d'$ is a dilator embedding.
    Then for $X\subseteq (\partial\Omega)^d$,
    \begin{align*} 
        X\in \nu_d & \textstyle \iff \{p\in \Omega^{\int d}\mid \partial p\circ \phi_d\in X\}\in \mu_d 
        \iff (\int f^*)^{-1}[\{p\in \Omega^{\int d}\mid \partial p\circ \phi_d\in X\}] \in \mu_{d'}\\
        &\textstyle \iff 
        \{q\in \Omega^{\int d'}\mid \partial(q\circ \int f)\circ \phi_d \in X\} = \{q\in \Omega^{\int d'}\mid \partial q\circ \phi_{d'}\circ f \in X\} \in \mu_{\int d'}\\
        &\textstyle \iff 
        \{q\in \Omega^{\int d'}\mid \partial q\circ \phi_{d'}\in \{r\in (\partial\Omega)^{d'}\mid r\circ f\in X\}\} \in \mu_{\int d'} \\
        &\iff \{r\in (\partial \Omega)^{d'} \mid r\circ f\in X\}\in \nu_{d'} \iff 
        (f^*)^{-1}[X]\in \nu_{d'}.
    \end{align*}
    Note that $\partial(q\circ \int f)\circ \phi_d = (\partial q)\circ \partial(\int f)\circ \phi_d = \partial q \circ \phi_{d'}\circ f$.
    For $\sigma$-completeness, suppose that $D$ is a countable dilator and $\{d_n\mid n<\omega\}$ is a countable family of finite subdilators of $D$.
    Then $\int d_n\subseteq \int D$ is a finite flower with no nullary terms for each $n<\omega$.
    Now suppose that we are given $X_n \in \nu_{d_n}$ for each $n<\omega$, so
    \begin{equation*}
        Y_n := \{p\in \Omega^{\int d_n}\mid \partial p\circ \phi_{d_n}\in X_n\} \in \mu_{\int d_n}.
    \end{equation*}
    Hence by the $\sigma$-completeness of $\Omega$, there is an embedding $e\colon \int D\to \Omega$ such that $e\restriction \int d_n \in Y_n$ for each $n<\omega$.
    Now observe that the following diagram commutes:
    \begin{equation*}
        \begin{tikzcd}
            d_n & \partial(\int d_n) & \\
            D & \partial(\int D) & \partial \Omega 
            \arrow[from=1-1, to=1-2, "\phi_{d_n}"]
            \arrow[from=2-1, to=2-2, "\phi_D"']
            \arrow[from=1-1, to=2-1, "\subseteq"']
            \arrow[from=1-2, to=2-2, "\subseteq"]
            \arrow[from=2-2, to=2-3, "\partial e"']
            \arrow[from=1-2, to=2-3, "\partial (e\restriction \int d_n)"]
        \end{tikzcd}
    \end{equation*}
    Hence $e\restriction \int d_n \in Y_n$ implies $(\partial e\circ \phi_D)\restriction d_n = \partial e \circ \phi_{d_n}\in X_n$.
\end{proof}
Hence, we have a measurable dilator if we construct a universal flower with a measure family $\mu_d$ for a finite flower with no nullary terms satisfying coherence and $\sigma$-completeness. Thus, we construct a measurable flower instead of constructing a measurable dilator directly.

\subsection{Measurable cardinal and $\bfPi^1_1$-determinacy} \label{Subsection: Measurable cardinal and Pi 1 1 Det}
In this subsection, we review a proof of $\bfPi^1_1$-determinacy from the existence of a measurable cardinal. We will see later that almost the same proof carries over the proof of $\bfPi^1_1$-determinacy from the existence of a measurable dilator.
Let us start with the following lemma:
\begin{lemma} \label{Lemma: A continuous family of finite linear orders}
    For every $\Pi^1_1[R]$-formula $\phi(x)$ we can find an $R$-recursive family of linear orders $\{\prec_s \mid s\in \omega^{<\omega}\}$ such that 
    \begin{enumerate}
        \item $\prec_s$ is a linear order over $|s|$,
        \item $s\subseteq t\implies \prec_s \subseteq \prec_t$, 
        \item $\phi(x)$ iff $\prec_x=\bigcup_{n<\omega}\prec_{x\restriction n}$ is well-ordered.
    \end{enumerate}
\end{lemma}
\begin{proof}
    Let $T$ be an $R$-recursive tree over $\omega\times\omega$ such that $\lnot \phi(x)$ iff there is a real $y$ such that $\lag x,y\rag$ form an infinite branch of $T$.
    Fix a recursive enumeration $\lag\bfs_i\mid i<\omega\rag$ of $\omega^{<\omega}$ such that if $\bfs_i\subsetneq \bfs_j$ then $i< j$. (It implies $|\bfs_i|\le i$ for every $i$.)
    Then let us define $\prec_s$ of field $|s|$ as follows: We have $i\prec_s j$ if and only if either
    \begin{enumerate}
        \item $\lag s\restriction|\bfs_i|,\bfs_i\rag\in T$, $\lag s\restriction|\bfs_j|,\bfs_j\rag\in T$, and $\bfs_i <_\KB \bfs_j$, or
        \item $\lag s\restriction|\bfs_i|,\bfs_i\rag\notin T$ and $\lag s\restriction|\bfs_j|,\bfs_j\rag\in T$, or 
        \item $\lag s\restriction|\bfs_i|,\bfs_i\rag\notin T$, $\lag s\restriction|\bfs_j|,\bfs_j\rag\notin T$, and $i<j$.
    \end{enumerate}
    Then we can see that the first two conditions hold. Furthermore, $\prec_x$ satisfies the definition of $\prec_s$ in which $s$ is replaced by $x$, so $\prec_x$ has ordertype $\alpha+(T_x,<_\KB)$ for some $\alpha\le \omega$, where $T_x = \{t\in \omega^{<\omega}\mid (x\restriction |t|,t)\in T\}$.
    Hence $\prec_x$ is well-founded iff $(T_x,<_\KB)$ is well-founded iff $T_x$ has no infinite branch iff $\phi(x)$.
\end{proof}

Let $\kappa$ be a measurable cardinal with a normal measure $\calU$. By iterating a measure, we have
\begin{lemma} \label{Lemma: Countably complete family of measures for linear orders}
    Suppose that $c$ is a countable well-order such that $c=\bigcup_{n<\omega} c_n$ for finite $c_n\subseteq c$. Then we can find an $\omega_1$-complete measure $\calU_{c_n}$ over the set $\kappa^{c_n}$ of embeddings $c_n\to\kappa$ such that if $\{X_n\mid n<\omega\}$ is a family of sets such that $X_n\in \calU_{c_n}$ for each $n<\omega$, then there is $f\colon c\to \kappa$ such that $f\restriction c_n\in X_n$ for each $n<\omega$.
\end{lemma}
\begin{proof}
    Let us consider 
    \begin{equation*}
        \calU_{c_n} = \{X \subseteq \kappa^{c_n} \mid \epsilon_n^*[X] := \{p\circ \epsilon_n \colon |c_n|\to\kappa\mid p\in X\} \in \calU^{|c_n|}\}
    \end{equation*}
    where $\epsilon_n\colon |c_n| \to c_n$ is the unique order isomorphism.
    We claim that $\calU_{c_n}$ is the desired ultrafilter. Suppose that $X_n\in \calU_{c_n}$ for each $n$. Then we can find $Y_n\in \calU$ such that 
    \begin{equation*}
        [Y_n]^{|c_n|} \subseteq \epsilon_n^*[X_n] = \{p\circ \epsilon_n \colon |c_n|\to\kappa\mid p\in X_n\}.
    \end{equation*}
    Then let $Y=\bigcap_{n<\omega} Y_n\in \calU$ and choose an embedding $f\colon c\to Y$. 
    Then clearly
    \begin{equation*}
        (f\restriction c_n)\circ \epsilon_n \in [Y]^{|c_n|} \subseteq \epsilon^*_n[X_n],
    \end{equation*}
    so $f\restriction c_n\in X_n$.
\end{proof}

For a coanalytic game $G$, let us associate the family $\{\prec_s \mid s\in \omega^{<\omega}\}$ such that 
\begin{equation*}
    \text{Player I wins $G$ in the play $x$} \iff \WO(\prec_x).
\end{equation*}
Then consider the subsidiary game $G'$
\begin{equation*}
    \begin{array}{cccccc}
     \mathrm{I} & x_0,\eta_0 & & x_2,\eta_1 & & \cdots  \\
     \mathrm{II} & & x_1 & & x_3 & \cdots
\end{array}
\end{equation*}
for $\eta_i<\kappa$. Player I wins $G'$ iff
\begin{equation*}
    \forall i,j<\omega (i \prec_x j) \iff \eta_i < \eta_j.
\end{equation*}

Clearly, if Player I has a winning strategy in $G'$, then taking the projection gives a winning strategy for Player I in $G$.
\begin{proposition}
    If Player II has a winning strategy in $G'$, then Player II has a winning strategy in $G$.
\end{proposition}
\begin{proof}
    Let $\sigma'$ be a winning strategy for Player II in $G'$.
    For each partial play $s=\lag x_0,x_1,\cdots, x_{2n-2}\rag$ and $p\in \kappa^{\prec_s}$, define
    \begin{equation*}
        f_s(p) = \sigma'
        \begin{pmatrix}
            x_0,p(0) & & \cdots & x_{2n-2},p(n-1) \\
            & x_1 & \cdots & 
        \end{pmatrix}
    \end{equation*}
    Then define
    \begin{equation*}
        \sigma(s)=a \iff \{p\in \kappa^{\prec_s} \mid f_s(p)=a\}\in \calU^n.
    \end{equation*}
    It is well-defined because $f_s\colon \kappa^{\prec_s}\to\omega$ and $\calU_{\prec_s}$ is $\omega_1$-complete.
    Then take $Z_s = \{p\in[\kappa]^n\mid f_s(p)=\sigma(s)\}$.

    We claim that $\sigma$ is a winning strategy for Player II in $G$. Suppose not, let $x$ be a play respecting $\sigma$ but Player I wins. By the lemma, we can find an embedding $e\colon (\omega,\prec_x)\to \kappa$ such that $e\restriction (2n-1,\prec_{x\restriction(2n-1)}) \in Z_{x\restriction (2n-1)}$ for each $n$.
    Hence we have
    \begin{equation*}
        f_{x\restriction(2n-1)}(e\restriction d_{x\restriction(2n-1)}) = \sigma(x\restriction(2n-1))=x_{2n-1},
    \end{equation*}
    so the game
    \begin{equation*}
        \begin{matrix}
            x_0,e(0) & & \cdots & x_{2n-2},e(n-1) & \\
            & x_1 & \cdots & &  x_{2n-1} & 
        \end{matrix}
    \end{equation*}
    becomes a valid play respecting $\sigma'$. Since the play respects $\sigma'$, Player II wins. However, $e\colon (\omega,\prec_x)\to \kappa$ is an embedding, so Player I wins, a contradiction.
\end{proof}

\subsection{Measurable dilator and $\bfPi^1_2$-determinacy} \label{Subsection: Measurable dilator and Pi 1 2 Det}
In this subsection, we prove $\bfPi^1_2$-determinacy from the existence of a measurable dilator. We can also see that the following proof is more or less similar to that of $\bfPi^1_1$-determinacy from a measurable dilator.
\begin{lemma} \label{Lemma: A continuous family of finite dilators}
    Let $\phi(x)$ be a $\Pi^1_2[R]$-formula. Then there is a recursive family of finite dilators  $\{d_s \mid s\in\omega^{<\omega}\}$ such that
    \begin{enumerate}
        \item $|s| = \field(d_s)$.
        \item $s\subseteq t \implies d_s\subseteq d_t$.
        \item $\phi(x)$ iff $d_x = \bigcup_{n<\omega} d_{x\restriction n}$ is a dilator.
    \end{enumerate}
\end{lemma}
\begin{proof}
    Let $\phi(x) \equiv \forall y\in\bbR \lnot\psi(x,y)$ for some $\Pi^1_1[R]$-formula $\psi(x,y)$.
    By modifying \autoref{Lemma: A continuous family of finite linear orders}, for a $\Pi^1_1[R]$-formula $\psi(x,y)$ we can find an $R$-recursive family of linear orders $\{\prec_{s,t} \mid s,t\in \omega^{<\omega},\ |s|=|t|\}$ such that 
    \begin{enumerate}
        \item $\prec_{s,t}$ is a linear order over $|s|$,
        \item $s\subseteq s',\ t\subseteq t'\implies \prec_{s,t}\subseteq \prec_{s',t'}$, 
        \item $\psi(x,y)$ iff $\prec_{x,y}=\bigcup_{n<\omega}\prec_{x\restriction n, y\restriction n}$ is well-ordered.
    \end{enumerate}
    Again, fix a recursive $\lag\bfs_i\mid i<\omega\rag$ of $\omega^{<\omega}$ provided in the proof of \autoref{Lemma: A continuous family of finite linear orders}. Note that $\bfs_0=\lag\rag$, which we will exclude in constructing a dilator family.

    Then let us define a dilator $d_s$ of the field $|s|$.
    Before starting the main construction, let us state the motivation:
    We want to define a family of $R$-recursive finite dilators approximating the tree
    \begin{align*}
        \hat{T}_x(\alpha) &= \{\lag r_0,\xi_0,\cdots, r_{m-1},\xi_{m-1}\rag\mid r_0,\cdots,r_{m-1}\in \omega,\ \xi_0,\cdots,\xi_{m-1}\in \alpha\\ 
        & \hspace{10em} \land i\mapsto \xi_i\text{ is an increasing map }\prec_{x\restriction m,\lag r_0,\cdots,r_{m-1}\rag}\to \alpha\}.
    \end{align*}
    equipped with the Kleene-Brouwer order.
    This tree occurs in the proof of Shoenfield absoluteness theorem, and we can see that $\lnot\phi(x)$ holds iff there is $\alpha<\omega_1$ such that $\hat{T}_x(\alpha)$ has an infinite branch.
    We can also see that $\alpha\mapsto (\hat{T}_x(\alpha),<_\KB)$ is a predilator, so $\hat{T}_x$ can be seen as a functorial Shoenfield tree. 

    Let us consider the following finite F-semidilator:
    \begin{align*}
        \hat{d}_s(\alpha) &= (\{\lag \bfs_{i+1}(0),\xi_0,\cdots, \bfs_{i+1}(|\bfs_{i+1}|-1),\xi_{|\bfs_{i+1}|-1}\rag\mid \xi_0,\cdots,\xi_{|\bfs_{i+1}|-1}\in \alpha,\ i<|s|,\\ 
        & \hspace{14em} \land k\mapsto \xi_k\text{ is an increasing map }\prec_{s\restriction |\bfs_{i+1}|,\bfs_{i+1}}\to \alpha\},<_\KB).
    \end{align*}
    We also define for $f\colon \alpha\to\beta$,
    \begin{itemize}
        \item $\hat{d}_s(f)(\lag s_0,\xi_0,\cdots, s_m,\xi_m\rag) = \lag s_0,f(\xi_0),\cdots, s_m,f(\xi_m)\rag$,
        \item $\supp^{\hat{d}_s}_\alpha (\lag s_0,\xi_0,\cdots, s_m,\xi_m\rag) = \{\xi_0,\cdots,\xi_{m-1}\}$.
    \end{itemize}
    We use $\bfs_{i+1}$ instead of $\bfs_i$ to exclude the empty sequence, which will be the topmost of the linear order. It also makes $\hat{d}_s$ a finite flower, but this observation is unnecessary in this proof.
    Then we can see that $\hat{d}_s$ is an F-predilator and
    \begin{align*}
        \Tr(\hat{d}_s) & = \Big\{\lag \bfs_{i+1}(0),\sigma(0),\cdots, \bfs_{i+1}(|\bfs_{i+1}|-1),\sigma(|\bfs_{i+1}|-1)\rag\mid \\ & \hspace{20em} i<|s| \land \sigma \colon \!\!\prec_{s\restriction|\bfs_{i+1}|,\bfs_{i+1}}\to |\bfs_{i+1}| \text{ increasing}\Big\}
    \end{align*}
    Note that for each $i<|s|$ there is a unique increasing map $\prec_{s\restriction|\bfs_{i+1}|,\bfs_{i+1}}\to |\bfs_{i+1}|$.
    For notational convenience, write
    \begin{equation*}
        \tau_i = \lag \bfs_{i+1}(0),\sigma(0),\cdots, \bfs_{i+1}(|\bfs_{i+1}|-1),\sigma(|\bfs_{i+1}|-1)\rag
    \end{equation*}
    for the unique increasing map $\sigma \colon \!\!\prec_{x\restriction|\bfs_{i+1}|,\bfs_{i+1}}\to |\bfs_{i+1}|$.
    Then $\frakf(\hat{d}_s)$ is a finite dilator satisfying
    \begin{enumerate}
        \item $\field(\frakf(\hat{d}_s)) = \{\tau_i\mid i<|s|\}$.
        \item For each $i<|s|$, $\arity^{\frakf(\hat{d}_s)}(\tau_i) = |\bfs_{i+1}|$.
        \item $\Sigma^{\frakf(\hat{d}_s)}_{\tau_i}(j) = k$ when $j$ is the $k$th least element over $(|\bfs_{i+1}|,\prec_{s\restriction\bfs_{i+1},\bfs_{i+1}})$.
        \item $\bfp^{\frakf(\hat{d}_s)}(i,j)$ is the least natural $m$ such that $\bfs_{i+1}(m)\neq\bfs_{j+1}(m)$.
    \end{enumerate}
    
    We want to turn $\frakf(\hat{d}_s)$ into a dilator of the field $|s|$ with all desired properties pertaining. Thus we define a new dilator $d_s$ from $\hat{d}_s$ by `replacing $\tau_i$ with $i$';
    More precisely, we define $d_s$ in a way that
    \begin{enumerate}
        \item $\field(d_s) = |s|$.
        \item For each $i<|s|$, $\arity^{d_s}(i) = |\bfs_{i+1}|$.
        \item $\Sigma^{d_s}_i(j) = k$ when $j$ is the $k$th least element over $(|\bfs_{i+1}|,\prec_{s\restriction\bfs_{i+1},\bfs_{i+1}})$.
        \item $\bfp^{d_s}(i,j)$ is the least natural $m$ such that $\bfs_{i+1}(m)\neq\bfs_{j+1}(m)$.
    \end{enumerate}
    Then clearly $d_s\cong \frakf(\hat{d}_s)$, and it is easy to see that $d_s$ satisfies the first two conditions.
    For the last condition, observe that $\hat{T}_x(\alpha) =  \bigcup_{n<\omega} \hat{d}_{x\restriction n}(\alpha)$.
\end{proof}

Let us fix a measurable dilator $\Omega$ with a family of measures $\{\calU_d\mid d\in\Dil_{<\omega}\}$.
Similarly, let $G$ be a $\bfPi^1_2$-game, and fix a family of finite dilators $\{d_s \mid s\in\omega^{<\omega}\}$ such that
\begin{enumerate}
    \item $\field(d_s)=|s|$.
    \item $s\subseteq t \implies d_s\subseteq d_t$.
    \item Player I wins in $G$ by the play $x$ iff $d_x = \bigcup_{n<\omega} d_{x\restriction n}$ is a dilator.
\end{enumerate}
Then consider the subsidiary game $G'$
\begin{equation*}
    \begin{array}{cccccc}
     \mathrm{I} & x_0,\tau_0 & & x_2,\tau_1 & & \cdots  \\
     \mathrm{II} & & x_1 & & x_3 & \cdots
    \end{array}
\end{equation*}
where $\tau_i \in \field(\Omega)$. Player I wins $G'$ iff $i\mapsto \tau_i$ forms an embedding, i.e.,
\begin{equation*}
    \forall \cyrDe (d_x \vDash i<_\cyrDe j \iff \Omega\vDash \tau_i<_\cyrDe \tau_j).
\end{equation*}
\begin{proposition}
    If Player II has a winning strategy in $G'$, then Player II has a winning strategy in $G$.
\end{proposition}
\begin{proof}
    Again, let $\sigma'$ be a winning strategy for Player II in $G'$.
    For each partial play $s=\lag x_0,x_1,\cdots, x_{2n-2}\rag$ and $p\in \Omega^{d_s}$, define
    \begin{equation*}
        f_s(p) = \sigma'
        \begin{pmatrix}
            x_0,p(0) & & \cdots & x_{2n-2},p(n-1) \\
            & x_1 & \cdots & 
        \end{pmatrix}
    \end{equation*}
    Then define
    \begin{equation*}
        \sigma(s)=a \iff \{p\in \Omega^{d_s} \mid f_s(p)=a\}\in \calU_{d_s}.
    \end{equation*}
    Again, take $Z_s = \{p\in[\kappa]^n\mid f_s(p)=\sigma(s)\}\in\calU_{d_s}$.
    Now suppose the contrary that $\sigma$ is not a winning strategy for Player II, and let $x$ be a play on $G$ respecting $\sigma$ but Player I wins.
    Since $\Omega$ is measurable, we can find an embedding $e\colon d_x\to \Omega$ such that $e\restriction d_{x\restriction(2n-1)}\in Z_{x\restriction(2n-1)}$ for each $n$.
    It means for each $n$,
    \begin{equation*}
        f_{x\restriction(2n-1)}(e\restriction d_{x\restriction(2n-1)}) = \sigma(x\restriction(2n-1))=x_{2n-1},
    \end{equation*}
    so the game
    \begin{equation*}
        \begin{matrix}
            x_0,e(0) & & \cdots & x_{2n-2},e(n-1) & \\
            & x_1 & \cdots & &  x_{2n-1} & 
        \end{matrix}
    \end{equation*}
    becomes a valid play respecting $\sigma'$. Since the play respects $\sigma'$, Player II wins. However, $e\colon d_x\to \Omega$ is an embedding, so Player I wins, a contradiction.
\end{proof}

\section{The Martin Flower} \label{Section: Martin Flower}
In this section, we define a flower from an iterable cardinal that will be a measurable flower, which will be called the \emph{Martin flower} $\Omega^1_\sfM$ ($\sfM$ denotes Martin.) We define the Martin flower only for ordinals, which is enough to establish its properties, and we can extend the Martin flower to every linear order with known machinery.
We define the Martin flower for natural numbers and increasing maps between them first, then we define the Martin flower for ordinals. We will see that the two definitions cohere. Lastly, we prove that the Martin flower embeds every countable flower.

\subsection{The Martin flower for natural numbers}
We will define the Martin flower $\Omega^1_\sfM$ from $j\colon V_\lambda\to V_\lambda$, and let us define $\Omega^1_\sfM$ for natural numbers first:
\begin{equation*}
    \Omega^1_\sfM(n) = \kappa_n := \crit j_n.
\end{equation*}
and for $f\colon m\to n$,
\begin{equation*}
    \Omega^1_\sfM(f)(x) = j_{f(m-2)+1,f(m-1)}\circ\cdots\circ j_{f(0)+1,f(1)}\circ j_{0,f(0)}(x).
\end{equation*}
For notational convenience, let us write
\begin{equation*}
    j_f(x) = j_{f(m-2)+1,f(m-1)}\circ\cdots\circ j_{f(0)+1,f(1)}\circ j_{0,f(0)}(x).
\end{equation*}

Note that we will show that $\Omega^1_\sfM$ is a flower. So far, we only defined $\Omega^1_\sfM$ over the category of natural numbers with strictly increasing maps. To see $\Omega^1_\sfM$ is a functor, we need it to preserve the function composition:
\begin{lemma} \label{Lemma: Martin flower composition preservation}
    Let $f\colon m\to n$, $g\colon n\to k$ for $m\le n\le k$. Then $\Omega^1_\sfM(g\circ f) =  \Omega^1_\sfM(g)\circ \Omega^1_\sfM(f)$.
\end{lemma}
\begin{proof}
    \allowdisplaybreaks
    We always assume $\xi\in \Omega^1_\sfM(m) = \kappa_m$ throughout this proof.
    To illustrate the idea of the proof, let us consider the case $m=0$ first: Then we can see that
    \begin{align*}
        j_g\circ j_{0,f(0)}(\xi) & = j_{g(n-2)+1,g(n-1)}\circ\cdots \circ j_{g(0)+1,g(1)} \circ j_{0,g(0)}\circ j_{0,f(0)}(\xi) \\
        & = j_{g(n-2)+1,g(n-1)}\circ\cdots \circ j_{g(0)+1,g(1)} \circ j_{g(0),g(0)+f(0)} \circ j_{0,g(0)}(\xi) \\
        & = j_{g(n-2)+1,g(n-1)}\circ\cdots \circ j_{g(0)+1,g(1)} \circ j_{g(0)+1,g(0)+f(0)} \circ j_{g(0)} \circ j_{0,g(0)}(\xi) \\
        & = j_{g(n-2)+1,g(n-1)}\circ\cdots \circ j_{g(1),g(1)-1+f(0)} \circ j_{g(0)+1,g(1)} \circ j_{g(0)} \circ j_{0,g(0)}(\xi) \\
        & = j_{g(n-2)+1,g(n-1)}\circ\cdots \circ j_{g(1),g(1)-1+f(0)} \circ  j_{0,g(1)}(\xi) \\
        & = j_{g(n-2)+1,g(n-1)}\circ\cdots \circ j_{g(2),g(2)-2+f(0)} \circ  j_{0,g(2)}(\xi) \\
        & \hspace{15em}\vdots \\
        & = j_{g(n-2)+1,g(n-1)}\circ\cdots \circ j_{g(f(0)),g(f(0))-f(0)+f(0)} \circ  j_{0,g(f(0))}(\xi) \\
        & = j_{g(n-2)+1,g(n-1)}\circ\cdots \circ j_{g(f(0))+1, g(f(0)+1)} \circ j_{0,g(f(0))}(\xi)
    \end{align*}
    and $j_{0,g(f(0))}(\xi)<\kappa_{g(f(0))+1}$ since $\xi<\kappa_1$. Hence all other elementary embeddings do not change $j_{0,g(f(0))}(\xi)$, so we have $j_g\circ j_{0,f(0)}(\xi) = j_{0,g(f(0))}(\xi)$.

    For a general case, let us start from
    \begin{equation*}
        j_f(\xi) = j_{f(m-2)+1,f(m-1)}\circ\cdots \circ j_{f(0)+1,f(1)} \circ j_{0,f(0)} (\xi).
    \end{equation*}
    Applying $j_{0,g(0)}$ gives
    \begin{equation*}
        j_{g(0)+f(m-2)+1,g(0)+f(m-1)}\circ\cdots \circ j_{g(0)+f(0)+1,g(0)+f(1)} \circ j_{g(0),g(0)+f(0)} (j_{0,g(0)}(\xi)),
    \end{equation*}
    which is equal to
    \begin{equation*}
        j_{g(0)+f(m-2)+1,g(0)+f(m-1)}\circ\cdots \circ j_{g(0)+f(0)+1,g(0)+f(1)} \circ j_{g(0)+1,g(0)+f(0)} (j_{0,g(0)+1}(\xi)).
    \end{equation*}
    Applying $j_{g(0)+1,g(1)}$ gives
    \begin{equation*}
        j_{g(1)+f(m-2),g(1)+f(m-1)-1}\circ\cdots \circ j_{g(1)+f(0),g(1)+f(1)-1} \circ j_{g(1),g(1)+f(0)-1} (j_{0,g(1)}(\xi)).
    \end{equation*}
    Applying $j_{g(1)+1,g(2)}$ then gives
    \begin{equation*}
        j_{g(2)+f(m-2)-1,g(2)+f(m-1)-2}\circ\cdots \circ j_{g(2)+f(0)-1,g(2)+f(1)-2} \circ j_{g(2),g(2)+f(0)-2} (j_{0,g(2)}(\xi)).
    \end{equation*}
    By repetition, we have
    \begin{multline*}
        j_{g(f(0))+f(m-2)-f(0)+1,g(f(0))+f(m-1)-f(0)}\circ \\ \cdots \circ j_{g(f(0))+f(0)-f(0)+1,g(f(0))+f(1)-f(0)} \circ j_{g(f(0)),g(f(0))+f(0)-f(0)} (j_{0,g(f(0))}(\xi)).
    \end{multline*}
    $j_{g(f(0)),g(f(0))+f(0)-f(0)}$ is the identity, so the above is equal to
    \begin{equation*}
        j_{g(f(0))+f(m-2)-f(0)+1,g(f(0))+f(m-1)-f(0)}\circ \cdots \circ j_{g(f(0))+1,g(f(0))+f(1)-f(0)} (j_{0,g(f(0))}(\xi)).
    \end{equation*}
    Then apply $j_{g(f(0))+1,g(f(0)+1)}$, so we get
    \begin{multline*}
        j_{g(f(0)+1)+f(m-2)-f(0),g(f(0)+1)+f(m-1)-f(0)-1}\circ \cdots \circ j_{g(f(0)+1),g(f(0)+1)+f(1)-f(0)-1} \\ (j_{g(f(0))+1,g(f(0)+1)}(j_{0,g(f(0))}(\xi))).
    \end{multline*}
    Applying $j_{g(f(0)+1)+1,g(f(0)+2)}$, $\cdots$, $j_{g(f(1)-1)+1,g(f(1))}$ consecutively gives
    \begin{equation*}
        j_{g(f(1))+f(m-2)-f(1)+1,g(f(1))+f(m-1)-f(1)}\circ \cdots \circ j_{g(f(1)),g(f(1))+f(1)-f(1)} (j_{g(f(0))+1,g(f(1))}(j_{0,g(f(0))}(\xi))),
    \end{equation*}
    which is equal to
    \begin{equation*}
        j_{g(f(1))+f(m-2)-f(1)+1,g(f(1))+f(m-1)-f(1)}\circ \cdots (j_{g(f(0))+1,g(f(1))}(j_{0,g(f(0))}(\xi))).
    \end{equation*}
    Then we can see that applying $j_{g(f(1))+1,g(f(1)+1)}$, $\cdots$, $j_{g(f(m-1)-1)+1, g(f(m-1))}$ gives
    \begin{equation} \label{Formula: Composition rule for embeddings 00}
         j_{g(f(m-2))+1,g(f(m-1))}(\cdots (j_{g(f(0))+1,g(f(1))}(j_{0,g(f(0))}(\xi)))\cdots) = j_{g\circ f}(\xi).
    \end{equation}
    Since $\xi<\kappa_m$, \eqref{Formula: Composition rule for embeddings 00} is $<\kappa_{g(f(m-1))+1}$. Hence, applying $j_{g(f(m-1))+1,g(f(m-1)+1)}$ or all other remaining embeddings of $j_g$ does not change the value of \eqref{Formula: Composition rule for embeddings 00}. In sum, we have $j_g(j_f(\xi)) = j_{g\circ f}(\xi)$.
\end{proof}

Now, let us define the support function to ensure a semidilator structure.
For a strictly increasing $f\colon m\to n$, define
\begin{equation*}
    I_f = \{j_f(\xi)\mid \xi<\kappa_m\}.
\end{equation*}
Throughout this paper, we identify an increasing map $f\colon m\to n$ to a finite subset $a\subseteq n$ of size $m$. 
Under this convention, we can identify $f$ with its range. We want to understand $I_a$ as members of $\Omega^1_\sfM(\omega)$ whose support is a subset of $a$.
To see this idea working, we need a lemma:
\begin{lemma} \label{Lemma: Support lemma}
    For two finite subsets $a,b\subseteq \omega$, we have
    \begin{enumerate}
        \item $a\subseteq b$ implies $I_a\subseteq I_b$.
        \item $I_a\cap I_b = I_{a\cap b}$.
    \end{enumerate}
\end{lemma}
\begin{proof}
    To prove the first statement, it suffices to show it when $|b|=|a|+1$. Let $\{a(0),\cdots,a(m-1)\}$ be an increasing enumeration of $a$, $b=a\cup\{l\}$. 
    \begin{enumerate}
        \item Consider the case $l<a(0)$. We have
        \begin{equation*}
            j_a(\xi) = j_{a(m-2)+1,a(m-1)}\circ\cdots\circ j_{a(0)+1,a(1)}\circ j_{l+1,a(0)} \circ j_{0,l+1} (\xi)
        \end{equation*}
        Also, $j_{0,l+1} = j_0^{l+1} = j_0^l\circ j_0 = j_{0,l}\circ j_0$.
        Hence we get
        \begin{equation*}
            j_a(\xi) = j_{a(m-2)+1,a(m-1)}\circ\cdots\circ j_{a(0)+1,a(1)}\circ j_{l+1,a(0)} \circ j_{0,l} (j_0(\xi)).
        \end{equation*}
        If $\xi<\kappa_m$, then $j_0(\xi)<\kappa_{m+1}$. Hence $j_a(\xi) = j_{a\cup\{l\}}(j_0(\xi)) \in I_{a\cup \{l\}}$.
        
        \item Now let $a(i) < i < a(i+1)$ for some $i < i+1 < m$. We have
        \begin{equation*}
            j_a(\xi) = j_{a(m-2)+1,a(m-1)}\circ\cdots j_{l+1,a(i+1)}\circ j_l \circ j_{a(i)+1,l}\circ\cdots\circ j_{0,a(0)} (\xi)
        \end{equation*}
        Observe that $j_l\circ j_{a(i)+1,l} = j_{a(i)+1,l} \circ j_{a(i)+1}$, so we have
        \begin{equation*}
            j_a(\xi) = j_{a(m-2)+1,a(m-1)}\circ\cdots j_{l+1,a(i+1)} \circ j_{a(i)+1,l}\circ j_{a(i)+1} \circ\cdots\circ j_{0,a(0)} (\xi)
        \end{equation*}
        By repeating a similar computation, we have
        \begin{align*}
            j_a(\xi) & = j_{a(m-2)+1,a(m-1)}\circ\cdots j_{l+1,a(i+1)}\circ \fbox{$j_l$} \circ j_{a(i)+1,l}\circ\cdots\circ j_{0,a(0)} (\xi) \\ 
            & = j_{a(m-2)+1,a(m-1)}\circ\cdots j_{l+1,a(i+1)} \circ j_{a(i)+1,l}\circ \fbox{$j_{a(i)+1}$} \circ j_{a(i-1)+1,a(i)}\circ \cdots\circ j_{0,a(0)} (\xi) \\
            & = j_{a(m-2)+1,a(m-1)}\circ \cdots\circ j_{a(i)+1,l}\circ j_{a(i-1)+1,a(i)} \circ \fbox{$j_{a(i-1)+2}$} \circ\cdots\circ j_{0,a(0)} (\xi) \\ 
            & \hspace{35ex}\vdots \\
            & = j_{a(m-2)+1,a(m-1)}\circ\cdots \circ \fbox{$j_{a(0)+(i+1)}$}\circ j_{0,a(0)} (\xi) \\
            & = j_{a(m-2)+1,a(m-1)}\circ\cdots \circ j_{0,a(0)} (\fbox{$j_{i+2}$}(\xi)) = j_{a\cup\{l\}}(j_{i+2}(\xi)).
        \end{align*}
        Since $\xi<\kappa_m$, we have $j_{i+2}(\xi)<\kappa_{m+1}$. Hence $j_a(\xi) = j_{a\cup\{l\}}(j_{i+2}(\xi)) \in I_{a\cup\{l\}}$.

        \item Consider the case $l>a(m-1)$. By definition, every member of $I_a$ has the form
        \begin{equation*}
            j_a(\xi) = j_{a(m-2)+1,a(m-1)}\circ\cdots\circ j_{a(0)+1,a(1)}\circ j_{0,a(0)} (\xi)
        \end{equation*}
        for some $\xi<\kappa_m$. Then $j_{a(0),0}(\xi) < \kappa_{m+a(0)}$, $j_{a(0)+1,a(1)}\circ j_{0,a(0)} < j_{a(0)+1,a(1)}(\kappa_{m+a(0)}) = \kappa_{(m-1)+a(1)}$, and so on. Hence we have $j_a(\xi) < \kappa_{1+a(m-1)}$.
        This implies 
        \begin{equation*}
            j_a(\xi) = j_{a(m-1)+1,l}(j_a(\xi)) = j_{a\cup\{l\}}(\xi) \in I_{a\cup\{l\}}. 
        \end{equation*}
    \end{enumerate}

    It implies the first clause of the lemma, and we immediately have $I_{a\cap b}\subseteq I_a\cap I_b$.
    For the remaining direction, it suffices to show the following: Suppose that $l$ is the largest element of $(a\setminus b)\cup (b\setminus a)$, and assume that $l\in a$. Then $I_a\cap I_b\subseteq I_{a\setminus \{l\}}\cap I_b$.
    Suppose that $a=a'\cup\{l\}\cup c$, $b=b'\cup c$ for some $a',b',c$ such that $\max a',\max b'<l<\min c$. Also, assume that we are given $\xi<\kappa_{|a|}$, $\eta<\kappa_{|b|}$ such that $j_a(\xi)=j_b(\eta)$. Now let us divide the cases:
    \begin{enumerate}
        \item $c=\varnothing$: Then we have $j_{a'\cup \{l\}}(\xi) = j_{b'}(\eta)=j_b(\eta)<\kappa_{\max b+1}$. Also, we have $\max b + 1 \le l$. Hence
        \begin{equation*}
            j_{l-1} \circ j_{\max a'+1,l-1}\circ j_{a'}(\xi) = j_{\max a'+1,l}\circ j_{a'}(\xi) < \kappa_l. 
        \end{equation*}
        However, $\ran j_{l-1} \restriction \Ord \subseteq [0,\kappa_{l-1})\cup [\kappa_l,\lambda)$, so we get
        \begin{equation*}
            j_{l-1} (j_{\max a'+1,l-1}\circ j_{a'}(\xi)) < \kappa_{l-1}. 
        \end{equation*}
        This is possible only when $j_{l-1} (j_{\max a'+1,l-1}\circ j_{a'}(\xi)) = j_{\max a'+1,l-1}\circ j_{a'}(\xi)$. By repeating the same argument sufficiently many times, we have
        \begin{equation*}
            j_{\max a'+1,l}\circ j_{a'}(\xi) = j_{\max a'+1,l-1}\circ j_{a'}(\xi) = \cdots = j_{a'}(\xi),
        \end{equation*}
        so $j_{a'\cup\{l\}}(\xi) = j_{a'}(\xi)\in I_{a\setminus\{l\}}\cap I_b$.

        \item $c\neq\varnothing$: By canceling the $c$-part of the elementary embeddings from $j_a(\xi) = j_b(\eta)$, we have
        \begin{equation*}
            j_{\max a'+1,l}\circ j_{a'}(\xi) = j_l\circ j_{\max b'+1, l} \circ j_{b'}(\eta).
        \end{equation*}
        Hence $j_{\max a'+1,l}\circ j_{a'}(\xi) \in \ran j_l = \ran (j_{l-1}\cdot j_{l-1})$. We may view this equality as
        \begin{equation*}
            j_{l-1}(j_{\max a'+1,l-1}\circ j_{a'}(\xi)) \in \ran j_{l-1}(j_{l-1}),
        \end{equation*}
        so $j_{\max a'+1,l-1}\circ j_{a'}(\xi) \in \ran j_{l-1}$. Repeating the same manipulation several times, we have $j_{a'}(\xi)\in \ran j_{\max a'+1}$. Now, let us prove that
        \begin{equation} \label{Formula: Support lemma - Elementary embedding reduction}
            j_{a'}(\xi)\in \ran j_{\max a'+1} \implies \xi \in \ran j_{|a'|}.
        \end{equation}
        Let $m=|a'|$. Then
        \begin{equation*}
            j_{a'}(\xi) = j_{a'(m-2)+1,a'(m-1)} (j_{a'\setminus \{a'(m-1)\}}(\xi)) \in \ran j_{a'(m-1)+1} = \ran (j_{a'(m-1)-1}\cdot j_{a'(m-1)}).
        \end{equation*}
        Thus
        \begin{equation*}
            j_{a'(m-2)+1,a'(m-1)-1} (j_{a'\setminus \{a'(m-1)\}}(\xi)) \in \ran j_{a'(m-1)}.
        \end{equation*}
        By repeating the previous manipulation, we have
        \begin{equation*}
            j_{a'\setminus \{a'(m-1)\}}(\xi) \in \ran j_{a'(m-2)+2}
        \end{equation*}
        Thus, the tedious repetition gives \eqref{Formula: Support lemma - Elementary embedding reduction}.
        Now let $m=|a'|$, $n=|c|$, $\xi=j_m(\xi')$, and
        \begin{equation*}
            k = j_{c(n-2)+1,c(n-1)}\circ\cdots\circ j_{l,c(0)}.
        \end{equation*}
        Then
        \begin{align*}
            j_a(\xi) & = k \circ j_{a'(m-1)+1,l}\circ j_{a'}(\xi) = k \circ j_{a'(m-1)+1,l}\circ j_{a'}(j_m(\xi'))\\
            &= k \circ j_{a'(m-1)+1,l} \circ j_{a'(m-2)+1,a'(m-1)}\circ\cdots j_{0,a'(0)} \circ j_m (\xi') \\
            &= k \circ j_{a'(m-1)+1,l} \circ j_{a'(m-2)+1,a'(m-1)}\circ\cdots \circ j_{a'(0)-1,a'(1)}\circ j_{a'(0)+m}\circ  j_{0,a'(0)}  (\xi') \\
            &= k \circ j_{a'(m-1)+1,l} \circ j_{a'(m-2)+1,a'(m-1)}\circ\cdots \circ j_{a'(1)+(m-1)} \circ j_{a'(0)-1,a'(1)}\circ  j_{0,a'(0)}  (\xi') \\
            & \hspace{20em} \vdots \\
            &= k \circ j_{a'(m-1)+1,l} \circ j_{a'(m-1)+1}  \circ j_{a'(m-2)+1,a'(m-1)}\circ\cdots \circ j_{a'(0)-1,a'(1)}\circ  j_{0,a'(0)}  (\xi') \\
            &=  k \circ j_l \circ j_{a'(m-1)+1,l} \circ j_{a'(m-2)+1,a'(m-1)}\circ\cdots \circ j_{a'(0)-1,a'(1)}\circ  j_{0,a'(0)}  (\xi') \\
            &= j_{c(n-2)+1,c(n-1)}\circ\cdots\circ j_{l,c(0)} \circ j_l \circ j_{a'(m-1)+1,l} \circ j_{a'}(\xi') \\
            &= j_{c(n-2)+1,c(n-1)}\circ\cdots\circ j_{a'(m-1)+1,c(0)} \circ j_{a'}(\xi') = j_{a'\cup c}(\xi').
        \end{align*}
        Hence $j_a(\xi) = j_{a\setminus \{l\}}(\xi)\in I_{a\setminus\{l\}}\cap I_b$. \qedhere 
    \end{enumerate}
\end{proof}

Then the following claim is immediate:
\begin{proposition}
    For every $\xi<\lambda$, a finite subset $a\subseteq \omega$ satisfying $\xi\in I_a$ with the least cardinality uniquely exists. Furthermore, if $\xi<\kappa_n$, then the corresponding $a$ is a subset of $n$.
\end{proposition}
\begin{proof}
    For the uniqueness, if $a\neq b$ satisfies $\xi \in I_a\cap I_b$, then $\xi\in I_{a\cap b}$, and $|a\cap b|<|a|, |b|$. Now let us prove the existence and the last claim.
    If $\xi<\kappa_n$, then $\xi\in I_n$ (Recall that $j_{\mathsf{Id}_n} = \mathsf{Id}_n$.) That is, we have
    \begin{equation*}
        n \in \{a\in[\omega]^{<\omega}\mid \xi\in I_a\}.
    \end{equation*}
    Clearly, we can find a member of $\{a\in[\omega]^{<\omega}\mid \xi\in I_a\}$ of the least cardinality.
\end{proof}
Now, let us define the support function as follows:
\begin{definition}
    For $\xi<\kappa_n$, $\supp_n(\xi)$ is the unique subset $a\subseteq n$ of the least cardinality such that $\xi\in I_a$. Alternatively, $\supp_n(\xi) = \bigcap \{a \subseteq n\mid \xi \in I_a\}$.
\end{definition}

\begin{lemma} \label{Lemma: Support condition for Martin Flower}
    $\supp_n$ is a natural transformation from $\Omega^1_\sfM$ to $[\cdot]^{<\omega}$: That is, for every $f\colon m\to n$ and $\xi\in \Omega^1_\sfM(m)=\kappa_m$,
    \begin{equation*}
        \supp_n(\Omega^1_\sfM(f)(\xi)) = f^"[\supp_m(\xi)].
    \end{equation*}
    
    Furthermore, $\supp$ satisfies the support condition: That is, for every $f\colon m\to n$, 
    \begin{equation*}
        \{\xi < \kappa_n \mid \supp_n(\xi)\subseteq \ran f\} \subseteq \ran \Omega^1_\sfM(f).
    \end{equation*}
\end{lemma}
\begin{proof}
    First, let us observe that for two increasing $f\colon m\to n$, $g\colon n\to k$,
    \begin{equation*}
        I_{g\circ f} = \{j_{g\circ f}(\xi)\mid \xi<\kappa_m\} = \{j_g(j_f(\xi))\mid \xi<\kappa_m\} = j_g[I_f].
    \end{equation*}
    Hence $I_{f[a]}=j_f[I_a]$.
    Now suppose that $\supp_m(\xi)=a$. Then $\xi\in I_a$ and so
    \begin{equation*}
        j_f(\xi) \in I_{f[a]} \implies \supp_n j_f(\xi) \subseteq f[a].
    \end{equation*}
    If $\supp_n j_f(a)\neq f[a]$, then there is $a'\subsetneq a$ such that $j_f(\xi)\in I_{f[a']}$. It implies $\xi\in I_{a'}$, contradicting with that $\supp_m(\xi)=a$.
    For the support condition, $\supp_n(\xi)\subseteq\ran f$ implies there is $a\subseteq m$ such that $\supp_n(\xi)\subseteq f[a]$.
    Hence there is $\eta<\kappa_m$ such that $\xi=j_{f[a]}(\eta)=j_f(j_a(\eta))$, so $\xi\in \ran j_f$.
\end{proof}

\subsection{The Martin flower for ordinals}

So far, we have defined $\Omega^1_\sfM$ only for natural numbers. We want to define it to other ordinals, and the natural choice should be
\begin{equation*}
    \Omega^1_\sfM(\alpha) = \kappa_\alpha.
\end{equation*}
However, many parts of the definition of $\Omega^1_\sfM$ for natural numbers do not work smoothly: For example, for $f\colon \alpha\to\beta$, the naive definition for $\Omega^1_\sfM(f)$ will introduce infinitely long embedding composition that is unclear to formulate. But we still need to define $\Omega^1_\sfM(f)$. To address this issue, we define the support of an ordinal first, then define $\Omega^1_\sfM(f)$.

\begin{proposition}
    Let $\alpha\ge\omega$. Then every element of $\kappa_\alpha$ has the form
    \begin{equation} \label{Formula: Normal form for elements of kappa alpha}
        j_{\alpha_{n-2}+1,\alpha_{n-1}}\circ\cdots\circ j_{\alpha_0+1,\alpha_1}\circ j_{0,\alpha_0}(\xi)
    \end{equation}
    for some $n<\omega$, $\xi<\kappa_n$, $\alpha_0<\cdots<\alpha_{n-1}<\alpha$.
\end{proposition}
\begin{proof}
    We prove the following by induction on $\alpha$: If $\alpha$ is limit and for every $m<\omega$, every element of $\kappa_{\alpha+m}$ has the form \eqref{Formula: Normal form for elements of kappa alpha} for some $n<\omega$, $\xi<\kappa_n$, $\alpha_0<\cdots<\alpha_{n-1}<\alpha+m$.

    The previous claim holds for $\alpha=0$ trivially. For a general limit $\alpha>0$, observe that every ordinal in $M_\alpha$ has the form $j_{\beta,\alpha}(\eta)$ for some limit $\beta<\alpha$ (including 0) and $\eta\in M_\beta$.
    Then
    \begin{equation*}
        j_{\beta,\alpha}(\eta) < \kappa_{\alpha+m} = j_{\beta,\alpha}(\kappa_{\beta+m})
        \implies \eta < \kappa_{\beta+m}.
    \end{equation*}
    Hence by the inductive hypothesis, there is $n<\omega$, $\xi<\kappa_n$ and $\alpha_0<\cdots<\alpha_{n-1}<\beta+m$ such that $\eta = j_{\alpha_{n-2}+1,\alpha_{n-1}}\circ\cdots\circ j_{\alpha_0+1,\alpha_1}\circ j_{0,\alpha_0}(\xi)$.
    Now let $l$ be the least natural number such that $\alpha_l\ge \beta$. Then for every $k\ge l$, $\alpha_k$ takes the form $\beta+e_k$ for some $e_k<\omega$. Hence
    \begin{align*}
        j_{\beta,\alpha}(\eta) &= 
        j_{\beta,\alpha}(j_{\alpha_{n-2}+1,\alpha_{n-1}}\circ\cdots\circ j_{\alpha_0+1,\alpha_1}\circ j_{0,\alpha_0}(\xi)) \\
        &= (j_{\beta,\alpha}\circ j_{\beta+e_{k-2}+1,\beta+e_{k-1}}\circ\cdots\circ j_{\alpha_0+1,\alpha_1}\circ j_{0,\alpha_0})(\xi) \\
        &= ( j_{\alpha+e_{k-2}+1,\alpha+e_{k-1}}\circ j_{\beta,\alpha}\circ\cdots\circ j_{\alpha_0+1,\alpha_1}\circ j_{0,\alpha_0})(\xi) \\
        & \hspace{15em} \vdots \\
        &= (j_{\alpha+e_{k-2}+1,\alpha+e_{k-1}}\circ \cdots \circ j_{\alpha+e_l+1,\alpha+e_{l+1}} \circ j_{\beta,\alpha} \circ j_{\alpha_{l-1},\alpha_l}\circ\cdots\circ j_{0,\alpha_0})(\xi) \\
        &= (j_{\alpha+e_{k-2}+1,\alpha+e_{k-1}}\circ \cdots \circ j_{\alpha+e_l+1,\alpha+e_{l+1}} \circ j_{\beta,\alpha} \circ j_{\beta,\beta+e_l}\circ j_{\alpha_{l-1},\beta}\circ\cdots\circ j_{0,\alpha_0})(\xi) \\
        &= (j_{\alpha+e_{k-2}+1,\alpha+e_{k-1}}\circ \cdots \circ j_{\alpha+e_l+1,\alpha+e_{l+1}} \circ j_{\alpha,\alpha+e_l}\circ j_{\beta,\alpha}  \circ j_{\alpha_{l-1},\beta} \circ\cdots\circ j_{0,\alpha_0})(\xi) \\
        &= (j_{\alpha+e_{k-2}+1,\alpha+e_{k-1}}\circ \cdots \circ j_{\alpha+e_l+1,\alpha+e_{l+1}} \circ j_{\alpha_{l-1},\alpha+e_l} \circ\cdots\circ j_{0,\alpha_0})(\xi) \\
    \end{align*}
    and $\alpha+e_{k-1} < \alpha+m$. This finishes the proof.
\end{proof}

Now, let us use the notation
\begin{equation*}
    j_{\{\alpha_0,\cdots,\alpha_{n-1}\}}(\xi) = j_{\alpha_{n-2}+1,\alpha_{n-1}}\circ\cdots\circ j_{\alpha_0+1,\alpha_1}\circ j_{0,\alpha_0}(\xi).
\end{equation*}
Similar to the finite case, define
\begin{equation*}
    I_{\{\alpha_0,\cdots,\alpha_{n-1}\}} = \{j_{\{\alpha_0,\cdots,\alpha_{n-1}\}}(\xi) \mid \xi<\kappa_n\}.
\end{equation*}
Then we can see that the proof for \autoref{Lemma: Support lemma} also works for general $I_a$, so we can define the support for ordinals in $\kappa_\alpha$. 
We can also prove that if $\xi\in\kappa_\alpha$ has support $a$, then there is a unique $t_\xi < \kappa_{|a|}$ such that $\xi = j_a(t_\xi)$.
Moreover, we have the following:
\begin{lemma} \label{Lemma: Support for general Martin Flower}
    Let $\xi<\kappa_m$ be such that $\supp(\xi) = m$. For $\eta_0<\cdots<\eta_{m-1}<\alpha$, we have $\supp (j_{\{\eta_0,\cdots,\eta_{m-1}\}}(\xi)) = \{\eta_0,\cdots,\eta_{m-1}\}$ and $t_{j_{\{\eta_0,\cdots,\eta_{m-1}\}}(\xi)}=\xi$.
\end{lemma}
\begin{proof}
    By the proof of \autoref{Lemma: Support condition for Martin Flower} with $f(i)=\eta_i$ and $a=m$, we have $\supp (j_{\{\eta_0,\cdots,\eta_{m-1}\}}(\xi)) = \{\eta_0,\cdots,\eta_{m-1}\}$.
    The remaining equality follows from the definition of $t_{j_{\{\eta_0,\cdots,\eta_{m-1}\}}(\xi)}$.
\end{proof}

Now let us define $\Omega^1_\sfM(f)$ and shows that $\Omega^1_\sfM$ preserves function composition and satisfies the support condition:
\begin{proposition}
    Let $f\colon\alpha\to\beta$. Define $\Omega^1_\sfM(f)$ by
    \begin{equation*}
        \Omega^1_\sfM(f)(\xi) = j_{f[\supp \xi]}(t_\xi).
    \end{equation*}
    Then we have the following:
    \begin{enumerate}
        \item $\Omega^1_\sfM(f) = j_f$ for $f\colon m\to n$, $m\le n<\omega$.
        \item $\Omega^1_\sfM(g\circ f) = \Omega^1_\sfM(g) \circ \Omega^1_\sfM(f)$.
        \item $\supp (\Omega^1_\sfM(f)(\xi)) = f[\supp \xi]$.
    \end{enumerate}
\end{proposition}
\begin{proof}
    \begin{enumerate}[wide, labelwidth=!, labelindent=0pt]
        \item Let $\xi<\kappa_m$ and $g\colon k \to \supp\xi$ be an increasing enumeration of the support of $\xi$. Then $\xi = j_g(t_\xi)$, and $\Omega^1_\sfM(f)(\xi) = j_{f\circ g}(t_\xi)=j_f(j_g(t_\xi)) = j_f(\xi)$ by \autoref{Lemma: Martin flower composition preservation}.
        
        \item Let $f\colon\alpha\to\beta$ and $g\colon \beta\to\gamma$. For $\xi<\kappa_\alpha$, we have $\Omega^1_\sfM(g\circ f)(\xi) = j_{(g\circ f)[\supp\xi]}(t_\xi)$ and
        \begin{equation*}
             (\Omega^1_\sfM(g) \circ \Omega^1_\sfM(f))(\xi) = j_{g[\supp j_{f[\supp \xi]}(t_\xi)]}(t_{j_{f[\supp \xi]}(t_\xi)}).
        \end{equation*}
        Then by \autoref{Lemma: Support for general Martin Flower}, we have that the right-hand side is equal to $j_{g[f[\supp \xi]]}(t_\xi)$.

        \item Follows from \autoref{Lemma: Support for general Martin Flower}.
        \qedhere 
    \end{enumerate}
\end{proof}

 Furthermore, we can see that it is a preflower:
\begin{proposition}
    $\Omega^1_\sfM$ is a preflower. In particular, if $j$ is iterable, then $\Omega^1_\sfM$ is a flower.
\end{proposition}
\begin{proof}
    Let $\alpha\le\beta$ and $\iota\colon\alpha\to\beta$ be the insertion map (i.e., $\iota(\xi)=\xi$.) 
    Then $\Omega^1_\sfM(\iota)(\xi) = j_{\supp\xi}(t_\xi)=\xi$, so $\Omega^1_\sfM(\iota)\colon\kappa_\alpha\to\kappa_\beta$ is also an insertion map.
\end{proof}

\subsection{The universality of the Martin flower}
In this section, we prove the Martin flower embeds every countable flower, whose proof is motivated by \cite[2.1(a)]{Kechris2012HomoTreeProjScales}. In fact, we have something stronger:
\begin{proposition} \label{Proposition: Martin Flower is Universal}
    $\Omega^1_\sfM$ is universal. In fact, if $F\in V_\kappa$ is a flower, then $F$ embeds to $\Omega^1_\sfM$.
\end{proposition}
\begin{proof}
    Suppose that $F$ is constant (i.e., every $F$-term is nullary.)
    Since $\Omega^1_\sfM(0)=\kappa_0$ and every constant dilator in $V_\kappa=V_{\kappa_0}$ has ordertype $<\kappa_0$, we have the desired embeddability result.
    
    Now suppose that $F$ is nonconstant (i.e., there is a non-nullary term in $F$).
    Since $F\in V_\kappa$, $F(\alpha)\in V_\kappa$ for each $\alpha<\kappa$. Moreover, $F(\kappa)=\bigcup_{\alpha<\kappa} F(\alpha)$ and each $F(\alpha)$ is an initial segment of $F(\kappa)$. Hence the ordertype of $F(\kappa)$ is $\le\kappa$.
    Since $F$ is not nullary, we have that $F(\xi)$ is a \emph{proper} initial segment of $F(\eta)$ for $\eta>\xi\ge\omega$. This shows $F(\kappa)\cong\kappa$.
    
    Now let us fix $c_0\colon F(\kappa_0)\cong \kappa_0$, and define $c_n = j_{0,n}(c_0)\colon F(\kappa_n)\to\kappa_n$.
    Then define $\iota_n\colon F(n)\to \Omega^1_\sfM(n)$ by
    \begin{equation*}
        \iota_n(t(e_0,\cdots, e_{l-1})) = c_n(t(\kappa_{e_0},\cdots,\kappa_{e_{l-1}})),
    \end{equation*}
    where $t\in \field(F)$ has arity $l$ and $e_0<\cdots<e_{l-1}<n$.
    We first claim that $\iota\colon F\restriction\bbN\to \Omega^1_\sfM\restriction \bbN$ gives a natural transformation:
    For an increasing map $f\colon m\to n$, we have
    \begin{align*}
        j_f(c_m) & = j_f(j_{0,m}(c_0)) = j_{f(m-2)+1,f(m-1)}\circ\cdots\circ j_{f(0)+1,f(1)}\circ j_{0,f(0)}\circ j_{0,m}(c_0) \\
        &= j_{f(m-2)+1,f(m-1)}\circ\cdots\circ j_{f(0)+1,f(1)} \circ j_{f(0),f(0)+m} \circ j_{0,f(0)}(c_0) \\
        &= j_{f(m-2)+1,f(m-1)}\circ\cdots\circ j_{f(0)+1,f(1)} \circ j_{f(0)+1,f(0)+m} \circ j_{0,f(0)+1}(c_0) \\
        &= j_{f(m-2)+1,f(m-1)}\circ\cdots\circ j_{f(1),f(1)+m-1} \circ j_{f(0)+1,f(1)} \circ j_{0,f(0)+1}(c_0) \\
        &= j_{f(m-2)+1,f(m-1)}\circ\cdots\circ j_{f(1)+1,f(1)+m-1} \circ j_{f(0)+1,f(1)+1} \circ j_{0,f(0)+1}(c_0) \\
        & \hspace{15em} \vdots \\
        & = j_{f(m-1),f(m-1)+1} \circ j_{f(m-2)+1,f(m-1)}\circ\cdots\circ j_{f(0)+1,f(1)}\circ j_{0,f(0)+1} (c_0) \\
        & = j_{0,f(m-1)+1}(c_0) = c_{f(m-1)+1}.
    \end{align*}
    Hence
    \begin{equation} \label{Formula: Omega 1 M universality 00}
        \Omega^1_\sfM(f)(\iota_m(t(e_0,\cdots,e_{l-1}))) = j_f(c_m(t(\kappa_{e_0},\cdots,\kappa_{e_{l-1}}))) = c_{f(m-1)+1}(t(\kappa_{f(e_0)},\cdots,\kappa_{f(e_{l-1})})).
    \end{equation}
    Here $t$ is fixed by $j_f$ since $F\in V_\kappa$. Also,
    \begin{equation} \label{Formula: Omega 1 M universality 01}
        \iota_n(\Omega^1_\sfM(f)(t(e_0,\cdots,e_{l-1}))) = \iota_n(t(f(e_0),\cdots,f(e_{l-1}))) = c_n(t(\kappa_{f(e_0)},\cdots,\kappa_{f(e_{l-1})})).
    \end{equation}
    But observe that for $m<n$ and $x\in F(\kappa_m)\subseteq V_{\kappa_m}$, $c_m(x) = j_{m,n}(c_m(x)) = j_{m,n}(c_m)(j_{m,n}(x)) = c_n(x)$. Hence, theright-hand sidee of \eqref{Formula: Omega 1 M universality 00} and that of \eqref{Formula: Omega 1 M universality 01} are the same.

    We finalize the proof by showing that the transformation $\iota$ also preserves the support transformation.
    By letting $l=m$ and $e_i=i$ in the previous proof, we have
    \begin{equation*}
        \iota_n(t(f(0),\cdots,f(m-1))) = c_{f(m-1)+1}(t(\kappa_{f(0)},\cdots,\kappa_{f(m-1)}))
        =j_f(c_m(t(\kappa_0,\cdots,\kappa_{m-1}))).
    \end{equation*}
    By \autoref{Lemma: Support condition for Martin Flower}, it suffices to show that 
    \begin{equation*}
        \supp_m(c_m(t(\kappa_0,\cdots,\kappa_{m-1}))) = m.
    \end{equation*}
    $c_m(t(\kappa_0,\cdots,\kappa_{m-1})) <\kappa_m$ implies $c_m(t(\kappa_0,\cdots,\kappa_{m-1}))\in I_{\{0,1,\cdots,m-1\}}$.
    Now suppose that $c_m(t(\kappa_0,\cdots,\kappa_{m-1}))\in I_{m\setminus \{e\}}$ for some $e<m$, so there is $x<\kappa_{m-1}$ such that
    \begin{equation*}
        c_m(t(\kappa_0,\cdots,\kappa_{m-1})) = j_e(x).
    \end{equation*}
    Then we have $t(\kappa_0,\cdots,\kappa_{m-1}) = c_m^{-1}(j_e(x))=j_e(c_{m-1}^{-1}(x))$, so $t(\kappa_0,\cdots,\kappa_{m-1})\in \ran j_e$. However, we defined $t(\kappa_0,\cdots,\kappa_{m-1})$ as a pair $(t,\{\kappa_0,\cdots,\kappa_{m-1}\})$, so we have $\kappa_e\in \ran j_e$, a contradiction.
\end{proof}

\begin{remark} \label{Remark: An embedding to a Martin Flower can only take limit value}
    In the proof of \autoref{Proposition: Martin Flower is Universal}, observe that for each $t\in\field(F)$ of arity $n$, $\iota_n(t(0,1,\cdots,n-1)) = c_n(t(\kappa_0,\cdots,\kappa_{n-1}))$. $t(\kappa_0,\cdots,\kappa_{n-1})$ over $F(\kappa_n)$ is a limit ordinal since all of $\kappa_0,\cdots,\kappa_{n-1}$ are limit.
    Hence $\iota$ maps every $F$-term to a limit ordinal, so we have not only an embedding $\iota\colon F\to \Omega^1_\sfM$, but also that $\ran \iota$ is a set of limit ordinals, which we will call a \emph{limit embedding.}
\end{remark}

\section{The measurability of the Martin Flower} \label{Section: Measurability of the Martin Flower}
The main goal of this section is to construct a measure family for the Martin flower to establish its measurability, which is the heart of the paper. Martin's proof of determinacy \cite{Martin1980InfiniteGames} hints at how to construct a measure family, but deciphering the precise construction step is not easy: Martin defined a dependent product of measures, and the product is done along a tree. However, dilators and flowers themselves look like they do not have a tree-like structure.
Here, we need dendrograms, and we define the measure by the dependent product of measures along a dendrogram tree. We also need to specify how to traverse a dendrogram to define a product, and here is where we cast trekkable dendrograms (cf. \autoref{Definition: Trekkable dendrogram}).

\emph{Throughout the remaining part of the paper, a `dendrogram' means a dendrogram for flowers with no nullary terms. In particular, every dendrogram we will see has a unique node of length 0, which is not terminal.
Furthermore, every dendrogram we will consider is finite except in \autoref{Subsection: omega1 completeness of the measure family}, where we also consider countable dendrograms.}

\subsection{Construction of a measure family}
We shall define $\nu^d$ for each trekkable dendrogram $d$, which is a measure over the set of embeddings from $\Dec(d)$ to $\Omega^1_\sfM$.
The construction of a measure family will take the following steps:
\begin{enumerate}
    \item For a trekkable dendrogram $d$ and an embedding $\beta\colon \Dec(d^\bullet)\to \Omega^1_\sfM$, we will define the subsidiary space $D^{d,\beta}$, which is a set of tuples of elementary embeddings, and a measure $\hnu^{d,\beta}$ over $D^{d,\beta}$.
    \item We define $\nu^{d,\beta}$ by projecting $\hnu^{d,\beta}$.
    \item We will show that $\nu^{d,\beta}$ does not depend on the choice of $\beta$. We also show that $\nu^{d,\beta}$ gives the same measure for an isomorphic $d$.
\end{enumerate}

$\hnu^{d,\beta}$ will be a dependent product of measures over $d$. Trees are not linear orders, and this is why we use trekkable dendrograms instead of arbitrary dendrograms: Trekkable dendrograms provide a way to traverse a given dendrogram when we take a product. $\beta\colon \Dec(d^\bullet)\to \Omega^1_\sfM$ associates each node in $d$ an $\Omega^1_\sfM$-term, so we will think of $\beta(x^\bullet)$ an $\Omega^1_\sfM$-term associated with $x\in d$. For that reason, we write $\beta(s^\bullet)$ as $\beta(s)$.

\begin{definition} \label{Definition: Measure for elementary embeddings}
    Let $d$ be a trekkable dendrogram and $\beta\colon \Dec(d^\bullet)\to \Omega^1_\sfM$.
    For $s\in d$, let us define $\hnu^{d,\beta}_s$ and the corresponding domain $D^{d,\beta}_s$ as follows:
    \begin{enumerate}
        \item $D^{d,\beta}_{0}=\{0\}$ and $\hnu^{d,\beta}_{0}$ is the trivial measure.
        \item If $\lh s = 1$, define
        \begin{itemize}
            \item $D^{d,\beta}_s = \{\vec{k}\cup \{(s, k')\} \mid \vec{k}\in D^{d,\beta}_{s-1} \land k'\in \Emb^{j_1\restriction V_{\kappa_1+\beta(s)}}_{\beta(s)}\}.$
            \item $X \in \hnu^{d,\beta}_\sigma \iff
            \Big\{\vec{k}\in D^{d,\beta}_{s-1} \mid \Big\{k'\in \Emb^{j_1\restriction V_{\kappa_1+\beta(s)}}_{\beta(s)}\mid  \vec{k}\cup \{ (s, k')\}\in X\Big\}\in \mu^{j_1\restriction V_{\kappa_1+\beta(s)}}_{\beta(s)}\Big\}\in \hnu^{d,\beta}_{s-1}$.
        \end{itemize}
        \item Suppose that $\lh(s)>1$. Then we have some $t\multimap s$ with $a=\bfe(t)$. Define
        \begin{itemize}
            \item $D^{d,\beta}_s = \{\vec{k}\cup \{ (s, k')\} \mid \vec{k}\in D^{d,\beta}_{s-1} \land k'\in \Emb^{j_a(\vec{k}_t)}_{\beta(s)}\}.$
            \item $X \in \hnu^{d,\beta}_s \iff \Big\{\vec{k}\in D^{d,\beta}_{s-1} \mid \Big\{k'\in \Emb^{j_a(\vec{k}_t)}_{\beta(s)}\mid  \vec{k}\cup \{ (s, k')\}\in X\Big\}\in \mu^{j_a(\vec{k}_t)}_{\beta(s)}\Big\}\in \hnu^{d,\beta}_{s-1}$.
        \end{itemize}
    \end{enumerate}
    Here $\vec{k}_s = \vec{k}(s)$, so $\vec{k}_s$ is the $s$th component of $\vec{k}$. Then define $D^{d,\beta} = D^{d,\beta}_{|d|-1}$ and $\hnu^{d,\beta} = \hnu^{d,\beta}_{|d|-1}$. 
\end{definition}

We can easily see that $\hnu^{d,\beta}$ is a $(\min\beta)$-complete measure over $D^{d,\beta}$. In particular, $\min\beta \ge\kappa_0$ since there is no nullary term in $d$, so $\hnu^{d,\beta}$ is $\kappa_0$-complete.

To improve the readability of the following proofs, we use the measure quantifier notation introduced in \autoref{Definition: Measure quantifier}; For example, we can express the definition of $\hnu^{d,\beta}_s$ for a successor $s$ of $t>0$ as follows:
\begin{equation*}
    X\in \hnu^{d,\beta}_s \iff \forall\big(\hnu^{d,\beta}_{s-1}\big) \vec{k} \in D^{d,\beta}_{s-1} \forall\big(\mu^{j_a(\vec{k}_t)}_{\beta(s)}\big) k'\in \Emb^{j_a(\vec{k}_t)}_{\beta(s)} [\vec{k}\cup \{( s, k')\}\in X].
\end{equation*}
Then we have
\begin{equation} \label{Formula: Decomposing a measure quantifier}
    \forall(\hnu^{d,\beta})\vec{k} \phi(\vec{k}) \iff \forall(\hmu^{d,\beta}_1) k^1 \forall(\hmu^{d,\beta}_2) k^2\cdots \forall(\hmu^{d,\beta}_{m-1}) k^{m-1} \phi(k^1,\cdots,k^{m-1}),
\end{equation}
where $\hmu^{d,\beta}_s$ is a unit measure appearing when we define $\hnu^{d,\beta}_s$, so
\begin{equation} \label{Formula: hmu d beta definition}
    \hmu^{d,\beta}_s = 
    \begin{cases}
        \mu^{j_1\restriction V_{\kappa_1+\beta(s)}}_{\beta(s)}, & \lh s=1, \\
        \mu^{j_a(\vec{k}_{t})}_{\beta(s)}, & \lh s>1,\ d\vDash t\multimap s,\text{ and }a=\bfe^d(t).
    \end{cases}
\end{equation}

Going back to the construction of the measure family of Martin dilator, elements of $\Omega^1_\sfM$ are ordinals and not elementary embeddings, and elements of $D^{d,\beta}$ are tuples indexed by non-zero members of $d$ that can be non-terminal nodes.  Hence, $\hnu^{d,\beta}$ cannot serve as a desired measure.
We can define a `correct' measure by projecting $\hnu^{d,\beta}$:
\begin{definition}
    Let us define a measure $\nu^{d,\beta}$ over $(\Omega^1_\sfM)^{\Dec(d)}$ as follows:
    \begin{equation*}
        X\in \nu^{d,\beta}\iff \big\{\vec{k}\in D^{d,\beta}\mid \{(s,\crit \vec{k}_s) \mid s\in\term(d)\}\in X\big\}\in \hnu^{d,\beta}.
    \end{equation*}
\end{definition}

$\nu^{d,\beta}$ is a projection of $\hnu^{d,\beta}$, so is a $\kappa_0$-complete ultrafilter.
We will see later that $\nu^{d,\beta}$ does not depend on $\beta$, and only depends on the isomorphic type of $d$. Hence, we can write $\nu^{d,\beta}$ as $\nu^d$ and regard it as a measure over the set of embeddings from a finite dilator $d$ to $\Omega^1_\sfM$.

\subsection{The correct concentration of the measure family}
In this subsection, we prove that $\nu^d$ is a measure over the set $(\Omega^1_\sfM)^d$. The following theorem is a major intermediate step to prove that $\nu^d$ concentrates on $(\Omega^1_\sfM)^d$:
\begin{theorem} \label{Theorem: hnu concentrates correctly}
    $\hnu^{d,\beta}$ concentrates to the set 
    $\{\vec{k}\in D^{d,\beta}\mid s^\bullet\mapsto \crit \vec{k}_s \text{ is an embedding from $\Dec(d^\bullet)$ to $\Omega^1_\sfM$}\}$.
\end{theorem}

Then let us prove \autoref{Theorem: hnu concentrates correctly}.
First, let us prove the following lemma, which will have a significant role in the latter arguments:
\begin{lemma} \label{Lemma: Fundamental lemma for dendroid embedding by criticals}
    For $\hnu^{d,\beta}$-almost every $\vec{k} \in D^{d,\beta}$, for $s\in d$ and $n=\lh s$, if $d\vDash r\multimap s$ and $a=\bfe^d(r)<\max(1,n-1)$, $x$, then 
    \begin{equation*}
        \sup_{\xi<\kappa_{n-1}} j_{a+1}(\xi) < \crit\vec{k}_s < \vec{k}_s(\crit \vec{k}_s
        ) \le \kappa_n.
    \end{equation*}
    If $n\ge 2$, then we additionally have
    \begin{equation*}
        \crit\vec{k}_s <\vec{k}_s(\crit \vec{k}_s) = j_a(\crit \vec{k}_r) < \kappa_n.
    \end{equation*}
    In particular, for $\hnu^{d,\beta}$-almost every $\vec{k}$, we have
    \begin{enumerate}
        \item $\kappa_{n-1} < \crit\vec{k}_s < \kappa_n$,
        \item $j_{a+1}(\crit\vec{k}_r) < \crit\vec{k}_s$.
    \end{enumerate}
\end{lemma}
\begin{proof}
    If $n=1$, then for $\mu^{j_1\restriction V_{\kappa_1+\beta(s)}}_{\beta(s)}$-almost all $\vec{k}$ we have $\vec{k}_s\in \Emb^{j_1\restriction V_{\kappa_1+\beta(s)}}_{\beta(s)}$. Clearly if $k' \in \Emb^{j_1\restriction V_{\kappa_1+\beta(s)}}_{\beta(s)}$ then $\crit k'< \kappa_1$.
    Moreover, $\kappa_0 = \sup_{\xi<\kappa_0} j_1(\xi)$ and
    \begin{equation*}
        \forall\big(\mu^{j_1\restriction V_{\kappa_1+\beta(s)}}_{\beta(s)}\big) k'\in \Emb^{j_1\restriction V_{\kappa_1+\beta(s)}}_{\beta(s)} [\crit k'>\kappa_0]
        \iff \kappa_1= \crit j_1 > j_1(\kappa_0)=\kappa_0.
    \end{equation*}
    Hence for almost all $k'\in \Emb^{j_1\restriction V_{\kappa_1+\beta(s)}}_{\beta(s)}$ we have $\sup_{\xi<\kappa_0} j_1(\xi)<\crit k'$.
    
    If $d\vDash r\multimap s$ and $a=\bfe^d(r)<n-1$, then $\vec{k}_s \in \Emb^{j_a(\vec{k}_r)}_{\beta(s)}$ for $\hnu^{d,\beta}_s$-almost all $\vec{k} \in D^{d,\beta}$.
    Hence for $\hnu^{d,\beta}_{s}$-almost all $\vec{k}$,
    \begin{equation*}
        \crit \vec{k}_s <\crit j_a(\vec{k}_r)  = j_a(\crit \vec{k}_r) < j_a(\kappa_{n-1}) = \kappa_n.
    \end{equation*}
    For the lower bound, let us inductively assume that the lower bound inequality holds for $r$. 
    Observe that 
    \begin{align*}
        \forall\big(\mu^{j_a(\vec{k}_r)}_{\beta(s)}\big) k'\in \Emb^{j_a(\vec{k}_r)}_{\beta(s)} \left[\sup_{\xi<\kappa_{n-1}} j_{a+1}(\xi)<\crit k'\right]
        & \iff 
        j_a(\vec{k}_r)\left(\sup_{\xi<\kappa_{n-1}} j_{a+1}(\xi)\right) < j_a(\crit\vec{k}_r) \\
        & \iff j_a(\vec{k}_r)\left(\sup_{\xi<j_a(\kappa_{n-2})} (j_a\cdot j_a)(\xi)\right) < j_a(\crit\vec{k}_r) \\
        & \iff \vec{k}_r\left(\sup_{\xi<\kappa_{n-2}} j_a(\xi)\right) < \crit \vec{k}_r.
        \addtocounter{equation}{1}\tag{\theequation} \label{Formula: Fund-Ineq-00}
    \end{align*}
    We prove that the last inequality \eqref{Formula: Fund-Ineq-00} holds for $\hnu^{d,\beta}_{s-1}$-almost all $\vec{k}$; Equivalently,
    \begin{equation*}
        \forall\big(\hnu^{d,\beta}_r\big) \vec{k}\in D^{d,\beta}_r \cdots  \Big[\vec{k}_r\Big(\sup_{\xi<\kappa_{n-2}} j_a(\xi)\Big) < \crit \vec{k}_r\Big],
    \end{equation*}
    where we have measure quantifiers for elementary embeddings indexed by $t$ such that $r<t<s$. 
    However, $t$-indexed elementary embeddings for $r<t<s$ do not appear in the bracketed formula, so we can remove them by \autoref{Lemma: Eliminating unused measure quantifiers}.
    Thus what we prove is equivalent to the inequality \eqref{Formula: Fund-Ineq-00} for $\hnu^{d,\beta}_r$-almost all $\vec{k}\in D^{d,\beta}_r$.
    
    The inductive assumption on the lower bound for $r$ implies $\kappa_{n-1}<\crit \vec{k}_r$ for $\hnu^{d,\beta}_r$-almost all  $\vec{k}\in D^{d,\beta}_r$, so
    \begin{equation*}
        \sup_{\xi<\kappa_{n-2}} j_a(\xi) < j_a(\kappa_{n-2}) = \kappa_{n-1} < \crit \vec{k}_r.
    \end{equation*}
    Hence we have for $\hnu^{d,\beta}_r$-almost all $\vec{k}\in D^{d,\beta}_r$,
    \begin{equation*}
        \vec{k}_r\left(\sup_{\xi<\kappa_{n-2}} j_a(\xi)\right) = \sup_{\xi<\kappa_{n-2}} j_a(\xi) < \crit \vec{k}_r. \qedhere 
    \end{equation*}
\end{proof}

\begin{lemma} \label{Lemma: Arity lemma for dendroid embedding by criticals}
    Let $s\in d$ and $\lh s=n\ge 1$. Then for $\hnu^{d,\beta}$-almost all $\vec{k}\in D^{d,\beta}$, 
    \begin{enumerate}
        \item $\crit\vec{k}_s \notin \ran j_m$ for $m<n$.
        \item $j_0(\crit \vec{k}_s) > j_1(\crit \vec{k}_s) > \cdots > j_{n-2}(\crit \vec{k}_s) > j_{n-1}(\crit \vec{k}_s) > j_n(\crit \vec{k}_s)=\crit \vec{k}_s$.
    \end{enumerate}
    In particular, $\supp^{\Omega^1_\sfM}(\crit\vec{k}_s) = n$ and $(n,\crit \vec{k}_s)\in \Tr(\Omega^1_\sfM)$.
\end{lemma}
\begin{proof}
    \begin{enumerate}
        \item The case $\lh s=1$ is easy since $\kappa_0 < \crit \vec{k}_s < \kappa_1$ for $\hnu^{d,\beta}$.a.e $\vec{k}$.
        
        Now suppose that $d\vDash r\multimap s$ for some $r$ and $a=\bfe^d(r)<n-1$.
        Note that $\crit\vec{k}_s < \kappa_n$ for $\hnu^{d,\beta}$-almost all $\vec{k}\in D^{d,\beta}$ by \autoref{Lemma: Fundamental lemma for dendroid embedding by criticals}, so for $\hnu^{d,\beta}$-almost all $\vec{k}\in D^{d,\beta}$,
        \begin{equation*}
            \crit\vec{k}_s \in \ran j_m \iff \crit\vec{k}_s \in \ran j_m\restriction \kappa_{n-1}.
        \end{equation*}
        Thus we claim $\crit\vec{k}_s \notin \ran j_m\restriction \kappa_{n-1}$ for $\hnu^{d,\beta}$-almost all $\vec{k}\in D^{d,\beta}$, which is equivalent to
        \begin{equation*}
            \forall(\hnu^{d,\beta}_r)\vec{k}\in D^{d,\beta}_r \forall \big(\mu^{j_a(\vec{k}_r)}_{\beta(s)} \big) k'\in \Emb^{j_a(\vec{k}_r)}_{\beta(s)} {\Big[}k'\notin \ran j_m\restriction \kappa_{n-1}{\Big]}.
        \end{equation*}
        Moreover,
        \begin{equation} \label{Equation: Term arity calculation 00}
            \forall \big(\mu^{j_a(\vec{k}_r)}_{\beta(s)} \big) k'\in \Emb^{j_a(\vec{k}_r)}_{\beta(s)} {\Big[}k'\notin \ran j_m\restriction \kappa_{n-1}{\Big]} 
            \iff \crit j_a(\vec{k}_r) \notin \ran \Big(\big(j_a(\vec{k}_r)\big) (j_m\restriction \kappa_{n-1})\Big).
        \end{equation}
        Also, let us observe that \autoref{Lemma: Fundamental lemma for dendroid embedding by criticals} implies the following for $\hnu^{d,\beta}$-almost all $\vec{k}\in D^{d,\beta}$ (so also for $\hnu^{d,\beta}_r$-almost all $\vec{k}\in D^{d,\beta}_r$):
        \begin{equation*}
            \crit j_a(\vec{k}_r) = j_a(\crit \vec{k}_r) > j_a(\kappa_{n-2})=\kappa_{n-1}\ge\kappa_m, 
        \end{equation*}
        so $j_a(\vec{k}_r)(\kappa_{n-1}) = \kappa_{n-1}$ and $j_a(\vec{k}_r)(\kappa_m) = \kappa_m$. Hence the right-hand-side of \eqref{Equation: Term arity calculation 00} is equivalent to
        \begin{equation*}
            \crit j_a(\vec{k}_r) \notin \ran \Big(\big(j_a(\vec{k}_r)\cdot j_m\big) \restriction \kappa_{n-1}\Big).
        \end{equation*}
        Here observe that if $\xi<\kappa_{n-1}$, then $(j_a(\vec{k}_r))(\xi) = \xi$, so
        \begin{equation*}
            \big(j_a(\vec{k}_r)\cdot j_m\big)(\xi) = \big(j_a(\vec{k}_r)\cdot j_m\big)\big( (j_a(\vec{k}_r)(\xi) \big)
            = (j_a(\vec{k}_r))\big( j_m(\xi) \big),
        \end{equation*}
        so $\ran \Big(\big(j_a(\vec{k}_r)\cdot j_m\big) \restriction \kappa_{n-1}\Big) \subseteq \ran j_a(\vec{k}_r)$.
        However, for every elementary embedding $k'$, $\crit k'\notin \ran k'$, so we have the right-hand side of \eqref{Equation: Term arity calculation 00}.

        \item Let $m<n-1$. The previous item and \autoref{Lemma: Schlutzenberg's lemma for value comparison} applied to $\xi=\crit \vec{k}_s$ and $k=j_m\restriction V_{\kappa_n+\omega}$, $\delta=\kappa_n$ gives $j_{m+1}(\crit \vec{k}_s) < j_m(\crit \vec{k}_s)$.
        The last inequality follows from that $\crit \vec{k}_s <\kappa_n$ for $\hnu^{d,\beta}$-almost all $\vec{k}$.
        \qedhere 
    \end{enumerate}
\end{proof}

\begin{lemma} \label{Lemma: Immediate predecessor property lemma}
    The following holds for $\hnu^{d,\beta}$-almost all $\vec{k}$:
    Let $s\in d$ and $d\vDash r\multimap s$, $\lh s>1$, and $a=\bfe^d(r)<\lh r$. Then 
    \begin{equation*}
        \forall \gamma<\crit\vec{k}_r (j_a(\gamma)<\crit\vec{k}_s).
    \end{equation*}
    Hence for $\hnu^{d,\beta}$-almost all $\vec{k}$, $\crit\vec{k}_r$ is the least ordinal $\gamma$ such that $\crit\vec{k}_s \le j_a(\gamma)$.
    Furthermore, we have
    \begin{equation*}
        \sup_{\gamma<\crit\vec{k}_r} j_a(\gamma) < \crit\vec{k}_s.
    \end{equation*}
\end{lemma}
\begin{proof}
    Observe that
    \begin{align*}
        &
        \forall(\hnu^{d,\beta} )\vec{k}\in D^{d,\beta}
        \big[\forall \gamma<\crit\vec{k}_r (j_a(\gamma)<\crit\vec{k}_s)\big] 
        \\ \iff & 
        \forall(\hnu^{d,\beta}_r) \vec{k}\in D^{d,\beta}_r \forall(\mu^{j_a(\vec{k}_r)}_{\beta(s)}) k'\in \Emb^{j_a(\vec{k}_r)}_{\beta(s)}
        {\Big[}\forall \gamma<\crit\vec{k}_r \big[j_a(\gamma)<\crit k'\big]{\Big]} 
        \\ \iff &
        \forall(\hnu^{d,\beta}_r) \vec{k}\in D^{d,\beta}_r 
        {\Big[} \forall \gamma<j_a(\vec{k}_r)(\crit\vec{k}_r) \big[j_a(\vec{k}_r)\big(j_a\restriction V_{\vec{k}_r(\crit\vec{k}_r)}\big)(\gamma)<\crit(j_a(\vec{k}_r))\big] {\Big]}.
    \end{align*}
    Also, for $\hnu^{d,\beta}$-almost all $\vec{k}\in D^{d,\beta}$,
    \begin{equation*}
        \crit j_a(\vec{k}_r) = j_a(\crit \vec{k}_r) > \crit\vec{k}_r
    \end{equation*}
    since $a\le \lh r-1$. Hence $\crit j_a \le \kappa_{\lh r-1}\le \crit\vec{k}_r$ for $\hnu^{d,\beta}$-almost all almost all $\vec{k}\in D^{d,\beta}$ by \autoref{Lemma: Fundamental lemma for dendroid embedding by criticals}.
    This implies
    \begin{equation*}
        j_a(\vec{k}_r)(\crit\vec{k}_r) = \crit\vec{k}_r.
    \end{equation*}
    Hence
    \begin{align*}
        & \forall \gamma<j_a(\vec{k}_r)(\crit\vec{k}_r) \Big[j_a(\vec{k}_r)\big(j_a\restriction V_{\vec{k}_r(\crit\vec{k}_r)}\big)(\gamma)<\crit(j_a(\vec{k}_r))\Big] \\
        \iff& 
        \forall \gamma<\crit\vec{k}_r \Big[j_a(\vec{k}_r)\big(j_a\restriction V_{\vec{k}_r(\crit\vec{k}_r)}\big)(\gamma)<\crit(j_a(\vec{k}_r))\Big] \\
        \iff &
        \forall \gamma<\crit\vec{k}_r \Big[j_a(\vec{k}_r)(j_a(\gamma))<\crit(j_a(\vec{k}_r))\Big] \\ 
        \iff &
        \forall \gamma<\crit\vec{k}_r \Big[\vec{k}_r(\gamma)<\crit(\vec{k}_r)\Big].
    \end{align*}
    The last condition clearly holds, which finishes the proof.
    The last inequality follows from that  $\crit\vec{k}_r<\crit\vec{k}_s$ are inaccessible cardinals.\footnote{Its proof does not require the axiom of choice since the critical point $\kappa$ of an elementary embedding is inaccessible in the sense that $V_\kappa$ is a model of second-order $\ZF$, so it is closed under a limit of increasing ordinals below $\kappa$ of length $<\kappa$.}
\end{proof}

Now we prove that $s^\bullet\mapsto \crit \vec{k}_s$ is a dilator embedding from $\Dec(d^\bullet)$ to $\Omega^1_\sfM$ for almost all $\vec{k}$:
\begin{theorem} \label{Theorem: Critical preserves dilator relation}
    Let $s,t\in d$, $s,t\neq 0$ and $\cyrDe$ be an arity diagram between $s^\bullet$ and $t^\bullet$ in $\Dec(d^\bullet)$.
    Then for $\hnu^{d,\beta}$-almost all $\vec{k}$,
    \begin{equation*}
        \Dec(d^\bullet) \vDash s^\bullet <_\cyrDe t^\bullet \implies \Omega^1_\sfM\vDash \crit \vec{k}_s <_\cyrDe \crit \vec{k}_t.
    \end{equation*}
\end{theorem}

By the Elementary comparison decomposition theorem \autoref{Theorem: Elementary comparison decomposition}, it suffices to show \autoref{Theorem: Critical preserves dilator relation} for elementary comparison relations $<_\cyrDe$.
The following implies for almost all $\vec{k}$, $s^\bullet\mapsto \crit \vec{k}_s$ preserves elementary comparison relations of type (A):

\begin{proposition} \label{Proposition: Dendroid embeddings by criticals Case A}
    For $\hnu^{d,\beta}$-almost all $\vec{k}$, $s^\bullet\mapsto \crit\vec{k}_s$ preserves elementary comparison relations of type (A).
\end{proposition}
\begin{proof}
    Suppose that $s,t\in d$, $d\vDash t\multimap s$, $\lh t = n>0$ and $e=\bfe^d(t)$. There is only one elementary comparison relation $\cyrDe$ between $s^\bullet$ and $t^\bullet$ of type (A), namely,
    \begin{equation*}
        \Dec(d^\bullet) \vDash s^\bullet<_\cyrDe t^\bullet \iff \Dec(d^\bullet)(n+1) \vDash s^\bullet((n+1)\setminus \{e\}) < t^\bullet(n+1).
    \end{equation*}
    Also, $\Omega^1_\sfM \vDash \crit \vec{k}_s <_\cyrDe \crit\vec{k}_t$ is equivalent to
    \begin{equation*}
        j_{\{0,1,\cdots,n\}\setminus \{e\}}(\crit\vec{k}_s) < j_{\{0,1,\cdots,n\}}(\crit\vec{k}_t).
    \end{equation*}
    Hence it suffices to show the above inequality for $\hnu^{d,\beta}$-almost all $\vec{k}$: Observe that $j_{\{0,1,\cdots,n\}}$ is the identity and $j_{\{0,1,\cdots,n\}\setminus\{e\}} = j_e$. However, \autoref{Lemma: Fundamental lemma for dendroid embedding by criticals} implies $\crit\vec{k}_s <j_e(\crit\vec{k}_t)$ for $\hnu^{d,\beta}$-almost all $\vec{k}$, as desired.
\end{proof}

Then let us turn to the elementary comparison relations of type (B).
\begin{proposition} \label{Proposition: Dendroid embeddings by criticals Case B}
    For $\hnu^{d,\beta}$-almost all $\vec{k}$, $s^\bullet\mapsto \crit\vec{k}_s$ preserves elementary comparison relations of type (B).
\end{proposition}
\begin{proof}
    Let $s,t,r\in d$ be such that $d\vDash r\multimap s,t$ and $d\vDash s<t$. The type (B) comparison relation is $<_\cyrDe$ for the trivial $\cyrDe$, so it suffices to show the following: For $\hnu^{d,\beta}$-almost all $\vec{k}$, $\crit \vec{k}_s<\crit\vec{k}_t$. 
    This follows from the following computation:
    \begin{align*}
        & \forall(\hnu^{d,\beta})\vec{k}\in D^{d,\beta} \big[\crit\vec{k}_s < \crit\vec{k}_t\big]\\
         \iff &
         \forall(\hnu^{d,\beta}_t)\vec{k}\in D^{d,\beta}_t \big[\crit\vec{k}_s < \crit\vec{k}_t\big]\\
        \iff &
        \forall(\hnu^{d,\beta}_r)\vec{k}\in D^{d,\beta}_r
        \forall(\mu^{k'}_{\beta(s)}) k^0 \in \Emb^{k'}_{\beta(s)}
        \forall(\mu^{k'}_{\beta(t)}) k^1 \in \Emb^{k'}_{\beta(t)}
        \big[ \crit k^0 < \crit k^1 \big].\\
        \iff &
        \forall(\hnu^{d,\beta}_r)\vec{k}\in D^{d,\beta}_r
        \forall(\mu^{k'}_{\beta(s)}) k^0 \in \Emb^{k'}_{\beta(s)}
        \big[ (k' \restriction V_{\crit k' + \beta(t)})(\crit (k^0)) < \crit k' \big].\\
        \iff & 
        \forall(\hnu^{d,\beta}_r)\vec{k}\in D^{d,\beta}_r
        \big[ \big(k'\big(k' \restriction V_{\crit k' + \beta(t)}\big)\big)(\crit k') < k'(\crit k')\big]. \addtocounter{equation}{1}\tag{\theequation} \label{Formula: Case B comparison last formula}
    \end{align*}
    Here
    \begin{equation*}
        k' = 
        \begin{cases}
            j_{\bfe^d(r)}(\vec{k}_r), & \text{if } \lh r\ge 1 \\
            j_1\restriction V_{\kappa_1+\beta(t)} & \text{if }\lh r=0
        \end{cases}
    \end{equation*}
    We can see that \eqref{Formula: Case B comparison last formula} follows from \autoref{Lemma: Schlutzenberg's lemma for value comparison} and $\crit k'\notin \ran k'$ for \emph{every} elementary embedding $k'$.
\end{proof}

\begin{proposition}
    For $\hnu^{d,\beta}$-almost all $\vec{k}$, $s^\bullet\mapsto \crit\vec{k}_s$ preserves elementary comparison relations of type (C).
\end{proposition}
\begin{proof}
    Let $s,t\in d$ be two members with a common immediate predecessor $r$, and $a=\bfe^d(r)$, $m=\lh s=\lh t$. The corresponding comparison relation is
    \begin{equation*}
        d^\bullet(\omega) \vDash s^\bullet((m+1)\setminus \{a+1\}) < t^\bullet((m+1)\setminus \{a\}).
    \end{equation*}
    Hence we will prove the following for $\hnu^{d,\beta}$-almost all $\vec{k}$:
    \begin{equation*}
        j_{(m+1)\setminus \{a+1\}} (\crit\vec{k}_s) < j_{(m+1)\setminus \{a\}}(\crit \vec{k}_t),
    \end{equation*}
    which is equivalent to
    \begin{equation*}
        j_{a+1}(\crit \vec{k}_s) < j_a(\crit \vec{k}_t).
    \end{equation*}
    
    If $\lh r=0$, then $a=0$. By applying $j_0$ to $\forall \xi<\kappa_0 (j_0(\xi)<\crit \vec{k}_t)$, we have
    \begin{equation*}
        \forall \xi<\kappa_1 [j_1(\xi)<j_0(\crit \vec{k}_t)],
    \end{equation*}
    and the desired result follows from $\crit\vec{k}_s < \kappa_1$.
    If $\lh r\ge 1$, then applying $j_a$ to the inequality in \autoref{Lemma: Immediate predecessor property lemma} gives
    \begin{equation*}
        \forall \gamma < j_a(\crit\vec{k}_r) [j_{a+1}(\gamma) < j_a(\crit\vec{k}_t)],
    \end{equation*}
    and the conclusion follows from $\crit\vec{k}_s < j_a(\crit\vec{k}_r)$.
\end{proof}

\begin{proposition}
    For $\hnu^{d,\beta}$-almost all $\vec{k}$, $s^\bullet\mapsto \crit\vec{k}_s$ preserves elementary comparison relations of type (D).
\end{proposition}
\begin{proof}
    Let $r,s,t',t\in d$ be such that $d\vDash r\multimap s$ and $d\vDash r\multimap t'\multimap t$.
    We also write $a = \bfe^d(t')$, and $m=\lh s$ (so $m+1=\lh t$). If $a_0\ge a_1$, then the comparison relation (D) is equivalent to
    \begin{equation*}
        d^\bullet(\omega)\vDash s^\bullet((m+1)\setminus \{a\}) < t^\bullet(m+1),
    \end{equation*}
    so we need to $j_a(\crit\vec{k}_s) < \crit\vec{k}_t.$ for $\hnu^{d,\beta}$-almost all $\vec{k}$. We can see that
    \begin{align*}
        & \forall(\hnu^{d,\beta})\vec{k}\in D^{d,\beta} \big[j_a(\crit\vec{k}_s) < \crit\vec{k}_t\big]\\
        \iff &
        \forall(\hnu^{d,\beta}_r)\vec{k}\in D^{d,\beta}_r \forall(\hmu^{d,\beta}_s) k^0 \forall(\hmu^{d,\beta}_{t'}) k^1 \forall\big(\mu^{j_a(k^1)}_{\beta(t')}\big) k^2  \big[j_a(\crit k^0) < \crit k^2\big]\\
        \iff &
        \forall(\hnu^{d,\beta}_r)\vec{k}\in D^{d,\beta}_r \forall(\hmu^{d,\beta}_s) k^0 \forall(\hmu^{d,\beta}_{t'}) k^1 \big[j_a(k^1)\big(j_a(\crit k^0)\big) < \crit j_a(k^1)\big]\\
        \iff &
        \forall(\hnu^{d,\beta}_r)\vec{k}\in D^{d,\beta}_r \forall(\hmu^{d,\beta}_s) k^0 \forall(\hmu^{d,\beta}_{t'}) k^1 \big[j_a(\crit k^0) < j_a(\crit k^1)\big]\\
        \iff &
        \forall(\hnu^{d,\beta}_r)\vec{k}\in D^{d,\beta}_r \forall(\hmu^{d,\beta}_s) k^0 \forall(\hmu^{d,\beta}_{t'}) k^1 \big[\crit k^0 < \crit k^1\big]
    \end{align*}
    and the latter holds by $d^\bullet \vDash s<t'$ and \autoref{Proposition: Dendroid embeddings by criticals Case B}.
\end{proof}

\subsection{Independence of $\hnu$ from a trekkable order}
\label{Subsection: Independence from a trekkable order}
We defined $\hnu^{d,\beta}$ for a specific dendrogram $d$, and we want to guarantee the final measure only depends on the isomorphism type of $d$.
In this subsection, we prove that $\hnu^{d,\beta}$ and $\hnu^{d',\beta'}$ are the same if $d$ and $d'$ are isomorphic, and if $h\colon d\to d'$ is the isomorphism then $\beta'\circ h=\beta$.

We first discuss how to transform a given dendrogram into another isomorphic dendrogram. 
The following lemma says we can turn a dendrogram into another isomorphic one by successively exchanging $s$ and $s+1$ in the dendrogram. We include its proof for completeness.
\begin{lemma} \label{Lemma: Turn a trekkable dendrogram to the other}
    Say two trekkable dendrograms $d$ and $d'$ are \emph{adjacent witnessed by $m$} if $m+1<|d|=|d'|$ and the map $h\colon d\to d'$ switching $m$ and $m+1$ and fixing the others is a dendrogam isomorphism.
    For two isomorphic trekkable dendrograms $d$ and $d'$, we have a sequence of trekkable dendrograms
    \begin{equation*}
        d = d_0 \cong d_1 \cong \cdots \cong d_m \cong d'
    \end{equation*}
    such that for each $i$, $d_i$ and $d_{i+1}$ are adjacent.
\end{lemma}
\begin{proof}
    We claim that we can re-enumerate every trekkable dendrogram under the \emph{level-then-value order}: For a dendrogram $d$ and $s,t\in d$, we say $s<_\LV t$ if one of the following holds:
    \begin{enumerate}
        \item $\lh s<\lh t$, or
        \item $\lh s= \lh t = m$ and $\Dec^\bullet(d)(\omega)\vDash s(m)<t(m)$.
    \end{enumerate}
    We can see that $<_\LV$ is a linear order over $d$. Furthermore, we can see that if $d\vDash s\multimap t$ or $d\vDash s<t$, then $d\vDash s<_\LV t$.
    We say a trekkable dendrogram $d$ is \emph{aligned under the level-then-value order} if for $s,t\in d$, $s<_\bbN t$ if and only if $s <_\LV t$.
    Now let us consider the following algorithm: For a given trekkable dendrogram $d$, let us find the least $m$ such that $m >_\LV m+1$.
    If there is such, let us obtain a new dendrogram $d'$ by swapping $m$ and $(m+1)$; That is, we make $d'$ from $d$ in a way that the map $f\colon d\to d'$ such that $f(m)=m+1$, $f(m+1)=m$, $f(t)=t$ for $t\neq m,m+1$ is a dendrogram isomorphism.
    We repeat this process until we get a trekkable dengrogram aligned under the level-then-value order.
    \autoref{Figure: Bubble sorting a trekkable dendrogram} illustrates how it works.
    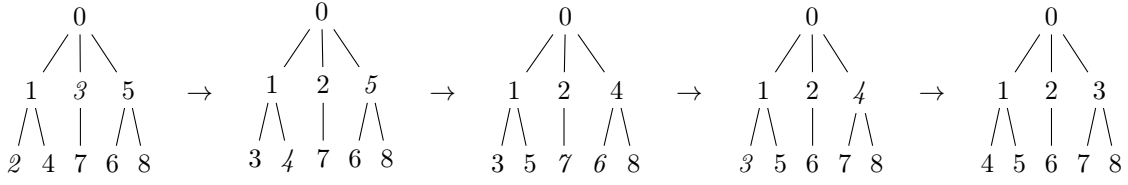
\begin{figure}
        \centering
        \begin{tabular}{ c c c c c c c c c}
            \begin{forest}
                [0, for tree={s sep=0em} 
                    [1 [\textit{2}] [4]]
                    [\textit{3} [7]]
                    [5 [6] [8]]
                ]
            \end{forest} & \parbox{1em}{\vspace{-6em}\centering$\to$} &
            \begin{forest}
                [0, for tree={s sep=0em} 
                    [1 [3] [\textit{4}]]
                    [2 [7]]
                    [\textit{5} [6] [8]]
                ]
            \end{forest} & \parbox{1em}{\vspace{-6em}\centering$\to$} &
            \begin{forest}
                [0, for tree={s sep=0em} 
                    [1 [3] [5]]
                    [2 [\textit{7}]]
                    [4 [\textit{6}] [8]]
                ]
            \end{forest} & \parbox{1em}{\vspace{-6em}\centering$\to$} &
            \begin{forest}
                [0, for tree={s sep=0em} 
                    [1 [\textit{3}] [5]]
                    [2 [6]]
                    [\textit{4} [7] [8]]
                ]
            \end{forest} & \parbox{1em}{\vspace{-6em}\centering$\to$} &
            \begin{forest}
                [0, for tree={s sep=0em} 
                    [1 [4] [5]]
                    [2 [6]]
                    [3 [7] [8]]
                ]
            \end{forest}
            \end{tabular}
        \caption{Sorting a trekkable dendrogram, with switched numbers italic. We assume $\bfe(s)=0$ for every $s$ in the example.}
        \label{Figure: Bubble sorting a trekkable dendrogram}
    \end{figure}

    We first claim that if $d$ is a trekkable dendrogram, so is $d'$.
    Suppose that we have $d'\vDash s\multimap t$. We have $d\vDash s\multimap t$ $s,t\notin \{m,m+1\}$, so $s<_\bbN t$.
    Now suppose that one of $s$ or $t$ is in $\{m,m+1\}$. Note that both $s$ and $t$ cannot be in $\{m,m+1\}$: Otherwise we have either $d\vDash m\multimap m+1$ or $d\vDash m+1 \multimap m$. The first possibility implies $m<_\LV m+1$, which contradicts the choice of $m$. The second possibility is impossible since $d$ is trekkable.

    Hence, we have four possible cases: $d'\vDash m\multimap t$, $d'\vDash (m+1)\multimap t$, $d'\vDash s\multimap m$, $d'\vDash s\multimap (m+1)$ with $s,t\notin\{m,m+1\}$.
    Each cases with the trekkability of $d$ imply $m+1<_\bbN t$, $m<_\bbN t$, $s<_\bbN m+1$, $s<_\bbN m$ respectively. The first and the fourth imply $m<_\bbN t$ and $s<_\bbN m+1$ respectively. Since $s,t\notin \{m,m+1\}$, the second and the third also imply $m+1<_\bbN t$, $s<_\bbN m$ respectively. This shows half of the trekkability of $d'$. By a similar argument, one can show that $d'\vDash s<t$ implies $s<_\bbN t$, so $d'$ is trekkable.

    We finish the proof by showing that our algorithm terminates: For a trekkable dendrogram $d$, let
    \begin{equation*}
        B(d) = \{(s,t)\in d^2\mid s<_\bbN t \land s>_\LV t\}.
    \end{equation*}
    We claim that $\lvert B(d)\rvert = 1+ \lvert B(d')\rvert$: 
    Observe that there are six types of elements in $B(d)$, namely, $(m,m+1)$, or $(s,m)$ or $(s,m+1)$ for $s<_\bbN m$, $(m,t)$ or $(m+1,t)$ for $t>_\bbN m+1$, or $(s,t)$ for $s<_\bbN t$ and $s,t\notin \{m,m+1\}$. In all cases, the first component is $<_\LV$-larger than the second component.
    The isomorphism from $d$ to $d'$ preserves $<_\LV$, and the isomorphism maps each tuples into $(m+1,m)$, $(s,m+1)$, $(s,m)$, $(m+1,t)$, $(m,t)$, or $(s,t)$ respectively. We can see that the second component is $<_\bbN$-larger for all types of tuples except the first. However, the first component is $<_\LV$-larger in all types of tuples.     
    It shows our algorithm terminates.
\end{proof}
Note that if $d$ and $d'$ are adjacent dendrograms witnessed by $l$, then we must have $l\ge 2$.
The previous algorithm turns a trekkable dendrogram $d$ into another isomorphic trekkable dendrogram $d'$ by turning $d$ into the trekkable dendrogram $d''$ that is aligned under the breadth-first search order, then turning $d''$ into $d'$.
The following proposition follows from examining the proof of \autoref{Lemma: Turn a trekkable dendrogram to the other}, which we record for a later purpose.
\begin{proposition} \pushQED{\qed}
    Let $d$ be a trekkable dendrogram. Then we can find a sequence of trekkable dendrograms
    \begin{equation*}
        d\cong d_0\cong d_1\cong\cdots\cong d_m
    \end{equation*}
    such that for each $k<m$, $d_k$ and $d_{k+1}$ are adjacent, and $d_m$ is aligned under the level-then-value order. Furthermore, if $d_k$ and $d_{k+1}$ are adjacent witnessed by $l$, then $l$ is the least number such that $d\vDash l>_\LV(l+1)$. \qedhere  
\end{proposition}

The next lemma says that the level-then-value order respects the size of the critical points:

\begin{lemma} \label{Lemma: Crit is a monotone map from LV to Omega 1 M}
    Let $d$ be a trekkable dendrogram and $s,t\in d$. If $s<_\LV t$, then $\crit\vec{k}_s < \crit\vec{k}_t$ for $\hnu^{d,\beta}$-almost all $\vec{k}$.
\end{lemma}
\begin{proof}
    If $\lh s<\lh t$, then $\crit\vec{k}_s < \kappa_{\lh s} \le \kappa_{\lh t-1} \le \crit\vec{k}_t$ for $\hnu^{d,\beta}$-almost all $\vec{k}$ by \autoref{Lemma: Fundamental lemma for dendroid embedding by criticals}.
    Now suppose that $\lh s=\lh t=m$. Then $\Dec(d^\bullet)\vDash s^\bullet(m) < t^\bullet (m)$, so $\crit \vec{k}_s <\crit\vec{k}_t$ for $\hnu^{d,\beta}$-almost all $\vec{k}$ by \autoref{Theorem: Critical preserves dilator relation}.
\end{proof}

The following theorem says for two adjacent dendrograms $d$ and $d'$, $\hnu^{d,\beta}$ and $\hnu^{d',\beta}$ are the same modulo permuting components. For a technical reason in the proof, \emph{we will assume in the rest of the paper that $\beta$ is a limit embedding}, that is, $\beta(\sigma)$ is always a limit ordinal for every $s\in d$. See \autoref{Remark: An embedding to a Martin Flower can only take limit value} for a limit embedding.
\begin{theorem} \label{Theorem: Measure independence from an order}
    Let $d$ and $d'$ be adjacent dendrograms witnessed by $l$, and suppose that $l$ is the least number such that $d\vDash l>_\LV (l+1)$. We also fix a limit embedding $\beta\colon \Dec(d^\bullet)\to \Omega^1_\sfM$ and the isomorphism $h_l\colon d\to d'$ switching $l$ and $(l+1)$.\footnote{In particular, we have that $h_l\circ h_l$ is the identity.}
    For $X\subseteq D^{d,\beta}$, we have
    \begin{equation*}
        X\in \hnu^{d,\beta} \iff h_l^*[X]\in \hnu^{d',\beta\circ h_l},
    \end{equation*}
    where $h^*_l[X] = \{p\circ h_l\mid p\in X\}$.
\end{theorem}
\begin{proof}
    Let $S = |d|$. Then
    \begin{equation} \label{Formula: Decomposing a measure quantifier 2}
        \forall(\hnu^{d,\beta})\vec{k} \phi(\vec{k}) \iff \forall(\hmu^{d,\beta}_0) k^0 \forall(\hmu^{d,\beta}_1) k^1\cdots \forall(\hmu^{d,\beta}_{S-1}) k^{S-1} \phi(k^0,\cdots,k^{S-1}),
    \end{equation}
    where $\hmu^{d,\beta}_s$ is as given in \eqref{Formula: hmu d beta definition}. For $m<S$, let us define $X\mid^d s\subseteq D^{d,\beta}_s$ by
    \begin{equation} \label{Formula: Measure decomposition}
        \vec{k}\in X\mid^d s \iff \forall (\hmu^{d,\beta}_{s+1}) \hat{k}^{s+1} \cdots \forall (\hmu^{d,\beta}_{S-1}) \hat{k}^{S-1} \Big[ \vec{k}\cup \big\{(t,\hat{k}^t)\mid s<t<S\big\} \in X\Big]
    \end{equation}
    Note that $\hmu^{d,\beta}_t$ may depend on some of $k^0,\cdots,k^s$. We have
    \begin{equation*}
        X\in \hnu^{d,\beta} \iff \forall (\hmu^{d,\beta}_0) k^0 \cdots \forall(\hmu^{d,\beta}_s) k^s \Big[ \big\{(t,k^t)\mid t\le s\big\}\in X\mid s\Big].
    \end{equation*}
    We assume that $l$ is the least number such that $d\vDash l>_\LV l+1$.
    Note that $\lh^d (l)>1$, otherwise, both $l$ and $(l+1)$ have $0$ as a common immediate predecessor in $d$, so $l$ cannot witness $d$ and $d'$ are adjacent. Moreover, for $s=1,2,\cdots, l-1,l+1$, $\lh^d(s)\le \lh^d(l)$.
    
    Now let $l' = l\restriction (\lh^d(l)-1)$ and $a=\bfe^d(l')$. By the assumption on $l$, we have $\lh^d(s)\le \lh^d(l)$ for every $s<_\bbN l$.
    Then we are tempted to argue
    \begin{align*}
        X\in \hnu^{d,\beta} \iff & 
        \forall (\hmu^{d,\beta}_0) k^0 \cdots \forall (\hmu^{d,\beta}_{l-1}) k^{l-1}
        \forall (\hmu^{d,\beta}_l) k^l \forall (\hmu^{d,\beta}_{l+1}) k^{l+1} \Big[ \{(s,k^s)\mid s\le l+1\big\}\in X\mid^d (l+1)\Big]
        \\ \iff & 
        \addtocounter{equation}{1}\tag{\theequation} \label{Formula: Swapping order 00}
        \forall (\hmu^{d,\beta}_0) k^0 \cdots \forall (\hmu^{d,\beta}_{l-1}) k^{l-1}
        \forall \big(j_a(k^{l'})(\hmu^{d,\beta}_{l+1})\big) k^{l+1}\in j_a(k^{l'})(\dom^{d,\beta}_{l+1}) \\
        & \hspace{0.5em}
        \Big[ \Big\{\big(s,j_a(k^{l'})(k^s)\big)\mid s<l \Big\}\cup 
        \Big\{\big(l,j_a(k^{l'})\big),\big(l+1,k^{l+1}\big)\Big\}
        \in j_a(k^{l'})\big(X\mid^d (l+1)\big)\Big],
    \end{align*}
    where
    \begin{equation*}
        \dom^{d,\beta}_s = \dom^{d,\beta}_s(k^0,\cdots,k^{l-1}) =
        \begin{cases}
            \Emb^{j_1\restriction V_{\kappa_1+\beta(s)}}_{\beta(s)} & \lh^d(s)=1, \\
            \Emb^{j_a(k^{s'})}_{\beta(s)} & \lh^d(s)>1,\ d\vDash s'\multimap s,\ \bfe^d(s')=a.
        \end{cases}
    \end{equation*}

    However, we need to check
    \begin{equation*}
        \forall (\hmu^{d,\beta}_0) k^0 \cdots \forall(\hmu^{d,\beta}_{l-1}) k^{l-1} \forall(\hmu^{d,\beta}_{l+1}) k^{l+1}\Big[\dom^{d,\beta}_{l+1}, \hmu^{d,\beta}_{l+1}, k^0,\cdots, k^{l-1}, X\mid^d (l+1) \in \dom j_a\big(k^{l'}\big)\Big]
    \end{equation*}
    to ensure the equivalence \eqref{Formula: Swapping order 00} works, otherwise, we do not know if we can apply $j_a(k^{l'})$ to the sets above.
    From now on, let us omit the expression `almost all,' which should be clear from context.
    Also, we fix the immediate predecessor $l'$ of $l$, and $a=\bfe^d(l')$.
    \begin{lemma}
        We have $\crit k^m < \crit k^l$ for $m=0,1,\cdots, l-1,l+1$.
    \end{lemma}
    \begin{proof}
        By the assumption on $l$, we have $d\vDash 0,1,\cdots, l-1,l+1<_\LV l$.
        Hence, we have a desired result by \autoref{Lemma: Crit is a monotone map from LV to Omega 1 M}.
    \end{proof}

    Let us recall that
    \begin{equation*}
        \dom j_a(k^{l'}) = V_{j_a(\crit k^{l'}) + j_a(\beta(l'))}.
    \end{equation*}
    In addition, for $s=1,\cdots,N-1,N+1$ and almost all $k^s$, $k^s \in V_{k^s(\crit k^s) + \beta(s) + 99}$ and
    \begin{equation*}
        k^s(\crit k^s) =
        \begin{cases}
            \kappa_1 & \text{If }\lh s=1, \\
            j_b(\crit k^{s'}) & \text{If }\lh^d(s)>1,\ d\vDash s'\multimap s,\ \bfe^d(s')=b.
        \end{cases}
    \end{equation*}
    
    Hence to see $k^s\in \dom j_a(k^{l'})$, it suffices to show:
    \begin{lemma} \label{Lemma: Rank bound for elementary embedding}
        For $s=1,\cdots, l-1,l+1$,
        \begin{equation} \label{Formula: Measure independence from an order 00}
            k^s(\crit k^s) + \beta(s) + 99 < j_a(\crit k^{l'}) + \beta(l).
        \end{equation}
        Furthermore, if $\lh^d(s) < \lh^d(l)$, then we have
        \begin{equation} \label{Formula: Measure independence from an order 01}
            k^s(\crit k^s) + \beta(s) + 99 < j_a(\crit k^{l'}).
        \end{equation}
    \end{lemma}
    \begin{proof}
        We write the immediate predecessor of $s$ in $d$ by $s'$, and $b=\bfe^d(s')$.
        If $\lh^d(s)=1$, then $b=0$ and $\beta(s) < \kappa_1 < \crit k^l < j_a(\crit k^{l'})$.
        Then we have \eqref{Formula: Measure independence from an order 01} since $j_a(\crit k^{l'})$ is inaccessible.

        Now suppose that $\lh^d(s)>1$. By the assumption on $l$, we have $d\vDash s<_\LV l$. If $\lh^d(s)<\lh^d(l)$, then we have
        \begin{equation*}
            j_b(\crit k^{s'}),\ \beta(s) < \kappa_{\lh^d(s)} \le \kappa_{\lh^d(l)-1} \le j_a(\crit k^{l'}).
        \end{equation*}
        Then the inaccessibility of $j_a(\crit k^{l'})$ and the equality $j_b(\crit k^{s'}) = k^s(\crit k^s)$ implies \eqref{Formula: Measure independence from an order 01}.
        Otherwise, we have $\lh^d(s)=\lh^d(l)$ since $\lh^d(s)\le \lh^d(l)$ holds. Following \autoref{Lemma: Lemma 3.39}, we have two possible cases:
        
        \begin{enumerate}
            \item Suppose that $s$ and $l$ have the same immediate predecessor in $d$, i.e., $s'=l'$.
            Since $\beta(s) < \beta(l)$ and $\beta(l)$ is a limit ordinal, we have \eqref{Formula: Measure independence from an order 00}. (Note that $s=l+1$ is impossible in this case; Otherwise, $l$ and $(l+1)$ have the common immediate predecessor in $d$.)
        
            \item Otherwise, \autoref{Lemma: Lemma 3.39} implies $\Dec^\bullet(d)\vDash s'(\lh^d(s)\setminus\{b\}) < l(\lh^d(l))$. Thus for almost all $\vec{k}$, we have $j_b(\crit \vec{k}_{s'})< \crit \vec{k}_l$. This shows the following holds for almost all $\vec{k}$:
            \begin{equation*}
                \vec{k}_s(\crit \vec{k}_s) = j_b(\crit \vec{k}_{s'})< \crit \vec{k}_l < j_a(\crit\vec{k}_{l'}).
            \end{equation*}
            Also, $d\vDash s<_\LV l$ implies $\beta(s) < \beta(l)$. Since $\beta(l)$ is limit, we have $\beta(s)+99<\beta(l)$.
            Combining all of this, we have \eqref{Formula: Measure independence from an order 00} for almost all $\vec{k}$.
            (Note that by \autoref{Lemma: Co-cardinal-small set is measure-large}, we can also derive \eqref{Formula: Measure independence from an order 01} in this case. However, we do not need this strengthened inequality in our purpose.)
            \qedhere 
        \end{enumerate}
    \end{proof}

    \begin{lemma}
        The rank of $\dom^{d,\beta}_{l+1}$ and $\hmu^{d,\mu}_{l+1}$ are less than $j_a(\crit k^{l'})$, so $\dom^{d,\beta}_{l+1}, \hmu^{d,\beta}_{l+1}\in \dom j_a(k^{l'})$ and they are fixed by $j_a(k^{l'})$.
    \end{lemma}
    \begin{proof}
    Let us divide the case:
    \begin{enumerate}
        \item Case $\lh^d(l+1)=1$: Then the rank of $\dom^{d,\beta}_{l+1}$ and $\hmu^{d,\beta}_{l+1}$ are no more than $\kappa_1+\beta(l+1)+99$.
        Observe that $\lh^d(l)>\lh^d(l+1)\ge 1$, so $\lh^d(l)\ge 2$.
        If $\lh^d(l)> 2$, then $j_a(\crit k^{l'}) > \crit k^l \ge \kappa_2$, so we have $\beta(l+1)<\kappa_2 < j_a(\crit k^{l'})$, which implies $\kappa_1+\beta(l+1)+99 < j_a(\crit k^{l'})$.
        If $\lh^d(l)=2$, then $a=0$ and $j_a(\crit k^{l'}) > j_0(\kappa_0) = \kappa_1\ge \beta(l+1)$, so again $\kappa_1+\beta(l+1)+99 < j_a(\crit k^{l'})$.
        
        \item Case $\lh^d(l+1)> 1$: Suppose that $d\vDash (l+1)'\multimap (l+1)$ and $b=\bfe^d((l+1)')$.
        The rank of $\dom^{d,\beta}_{l+1}$ and $\hmu^{d,\beta}_{l+1}$ are no more than $j_b(\crit k^{(l+1)'}) + \beta(l+1)+99$.
        Observe that $\lh^d((l+1)') < \lh^d(l')$, and \autoref{Lemma: Fundamental lemma for dendroid embedding by criticals} implies
        \begin{equation*}
            \crit k^{(l+1)'}<\kappa_{\lh^d (l+1)'} \quad \text{ and } \quad \kappa_{\lh^d(l')} < \crit k^{l'}. 
        \end{equation*}
        Also, note that $b<\lh^d (l+1)'$ and $a<\lh^d(l')$, so we have
        \begin{equation*}
            j_b(\crit k^{(l+1)'})<\kappa_{\lh^d(l+1)'} \le \kappa_{\lh^d(l')} = j_a(\kappa_{\lh^d(l') -1 }) < j_a(\crit k^{l'}).
        \end{equation*}
        By \autoref{Lemma: Co-cardinal-small set is measure-large}, we have 
        \begin{equation*}
            \beta(l+1)< \crit k^{l+1} < \kappa_{\lh^d(l+1)'} \le \kappa_{\lh^d(l')}<\crit k^l<j_a(\crit k^{l'}).
        \end{equation*}
        Then by the inaccessibility of $j_a(\crit k^{l'})$, we have $j_b(\crit k^{(l+1)'}) + \beta(l+1)+99<j_a(\crit k^{l'})$. \qedhere 
    \end{enumerate}
    \end{proof}

    \begin{lemma}
        $X\mid^d (l+1) \in \dom j_a(k^{l'})$.
    \end{lemma}
    \begin{proof}
        It suffices to show that $D^{d,\beta}_{l+1}$ has rank less than $j_a(\crit k^{l'}) + \beta(l)$.
        $D^{d,\beta}_{l+1}$ is a set of tuples of elementary embeddings, and we proved in \autoref{Lemma: Rank bound for elementary embedding} that each component of a tuple has rank less than $j_a(\crit k^{l'})+\beta(l)$.
        Hence the tuple also has rank less than $j_a(\crit k^{l'})+\beta(l)$, so the rank of $D^{d,\beta}_{l+1}$ is also less than $j_a(\crit k^{l'})+\beta(l)$.
    \end{proof}

    Hence \eqref{Formula: Swapping order 00} works, and is equivalent to
    \begin{align*}
        & \forall (\hmu^{d,\beta}_0) k^0 \cdots \forall (\hmu^{d,\beta}_{l-1}) k^{l-1}
        \forall \big(j_a(k^{l'})(\hmu^{d,\beta}_{l+1})\big) k^{l+1}\in j_a(k^{l'})(\dom^{d,\beta}_{l+1}) \\
        & \hspace{0.5em}
        \Big[ \Big\{\big(s,j_a(k^{l'})(k^s)\big)\mid s<l \Big\}\cup 
        \Big\{\big(l,j_a(k^{l'})\big),\big(l+1,k^{l+1}\big)\Big\}
        \in j_a(k^{l'})\big(X\mid^d (l+1)\big)\Big]\\
        \iff&
        \forall (\hmu^{d,\beta}_0) k^0 \cdots \forall (\hmu^{d,\beta}_{l-1}) k^{l-1} \forall (\hmu^{d,\beta}_{l+1}) k^{l+1} \forall (\hmu^{d,\beta}_{l}) k^{l}
        \Big[ \big\{(s,k^s)\mid s\le l+1\big\}\in X\mid^d (l+1)\Big]\\
        \iff&
        \forall (\hmu^{d,\beta}_0) k^0 \cdots \forall (\hmu^{d,\beta}_{l-1}) k^{l-1} \forall (\hmu^{d',\beta}_{l}) k^{l+1} \forall (\hmu^{d',\beta}_{l+1}) k^{l}\\& \hspace{0.5em}
        \Big[ \big\{(s,k^s)\mid s\le l-1\big\}\cup\{(l,k^{l+1}),(l+1,k^l)\}\in h^*_l[X\mid^d (l+1)]\Big].
    \end{align*}
    That is, we can switch the order between $k^{l+1}$ and $k^l$.
\end{proof}

\subsection{Independence of $\nu^{d,\beta}$ from $\beta$}
In this subsection, we prove that $\nu^{d,\beta}$ does not depend on the choice of $\beta$.

\begin{lemma} \label{Lemma: Product measure projection lemma}
    For a finite flower $d$ with no nullary terms and embeddings $\beta,\gamma\colon \Dec(d^\bullet)\to \Omega^1_\sfM$, suppose that $\beta(s)\le \gamma(s)$ for every $s\in d$. 
    If we define $\pi^d_{\beta,\gamma}\colon D^{d,\gamma}\to D^{d,\beta}$ by
    \begin{equation*}
        \pi^d_{\beta,\gamma}(\vec{k})(s) = \vec{k}_s\restriction V_{\crit\vec{k}_s + \beta(s)},
    \end{equation*}
    for $s\in \Dec(d^\bullet)$, then
    \begin{equation*}
        \forall Y\subseteq D^{d,\gamma}\Big[
        Y\in \hnu^{d,\beta} \iff 
        (\pi^d_{\beta,\gamma})^{-1}[Y]\in \hnu^{d,\gamma}\Big].
    \end{equation*}
\end{lemma}
\begin{proof}
    We prove it by induction on $s\in d$ as follows: For each $s\in d$ define $\pi^d_{\beta,\gamma,s}\colon D^{d,\gamma}_s\to D^{d,\beta}_s$ by
    \begin{equation*}
        \pi^d_{\beta,\gamma,s}(\vec{k}) = \big\{(t, \vec{k}_t\restriction V_{\crit\vec{k}_t + \beta(t)})\mid  t\in d,\ t\le_\bbN s\}.
    \end{equation*}
    Then we prove the following:
    \begin{equation*}
        \forall Y\subseteq D^{d,\gamma}_s\big[
        Y\in \hnu^{d,\beta}_s \iff 
        (\pi^d_{\beta,\gamma,s})^{-1}[Y]\in \hnu^{d,\gamma}_s\big],
    \end{equation*}
    which is equivalent to
    \begin{equation*}
        \forall Y\subseteq D^{d,\gamma}_s
        \left[
        \forall(\hnu^{d,\beta}_s)\vec{k}\in D^{d,\beta}_s[\vec{k}\in Y] \iff 
        \forall(\hnu^{d,\gamma}_s)\vec{k}\in D^{d,\gamma}_\sigma [\pi^d_{\beta,\gamma,s}(\vec{k})\in Y]\right].
    \end{equation*}
    If $\lh s =1$, then the inductive hypothesis and \autoref{Lemma: Measure projection lemma} implies
    \begin{align*}
        Y\in \hnu^{d,\beta}_s & \iff 
        \forall\Big(\mu^{j_1\restriction V_{\kappa_1+\beta(s)}}_{\beta(s)}\Big) k'\in \Emb^{j_1\restriction V_{\kappa_1+\beta(s)}}_{\beta(s)}
        \left[ \forall(\hnu^{d,\beta}_{s-1}) \vec{k}\in D^{d,\beta}_{s-1}
        [\vec{k}\cup \big\{(s,k')\big\}\in Y]\right]
        \\ (\text{by }\ref{Lemma: Measure projection lemma}) &\iff 
        \forall\Big(\mu^{j_1\restriction V_{\kappa_1+\beta(s)}}_{\beta(s)}\Big) k'\in \Emb^{j_1\restriction V_{\kappa_1+\gamma(\sigma)}}_{\gamma(\sigma)}
        \left[\forall(\hnu^{d,\beta}_{s-1})\vec{k}\in D^{d,\beta}_{s-1} [\vec{k}\cup\big\{ (s,k'\restriction V_{\kappa_1+\beta(s)})\big\}\in Y]\right]
        \\ (\text{Ind.}) &\iff 
        \forall\Big(\mu^{j_1\restriction V_{\kappa_1+\beta(s)}}_{\beta(s)}\Big)k'\in \Emb^{j_1\restriction V_{\kappa_1+\gamma(\sigma)}}_{\gamma(\sigma)}
        \left[\forall(\hnu^{d,\gamma}_{s-1})\vec{k}\in D^{d,\gamma}_{s-1} [\pi^d_{\beta,\gamma,s-1}(\vec{k})\cup\big\{ (s,k'\restriction V_{\kappa_1+\beta(s)})\big\}\in Y]\right]
        \\&\iff
        \forall\Big(\mu^{j_1\restriction V_{\kappa_1+\beta(s)}}_{\beta(s)}\Big) k'\in \Emb^{j_1\restriction V_{\kappa_1+\gamma(s)}}_{\gamma(s)}
        \left[\forall(\hnu^{d,\gamma}_{s-1})\vec{k}\in D^{d,\gamma}_{s-1} \Big[\pi^d_{\beta,\gamma,s}\big(\vec{k}\cup\big\{(s,k')\big\}\big)\in Y\Big]\right]
        \\&\iff \forall(\hnu^{d,\gamma}_s) \vec{k}\in D^{d,\gamma}_s [\pi^d_{\beta,\gamma,s}(\vec{k})\in Y]
        \iff (\pi^d_{\beta,\gamma,s})^{-1}[Y]\in \hnu^{d,\gamma}_s.
    \end{align*}
    The remaining case is similar, so we omit it.
\end{proof}

The following lemma is necessary to prove the next proposition, which roughly says the measure $\hnu^{d,\beta}$ sees an `initial segment of certain conditions' as small.
\begin{lemma} \label{Lemma: Co-cardinal-small set is measure-large}
    Suppose that $d$ is a finite flower without nullary terms and $\beta\colon \Dec(d^\bullet)\to \Omega^1_\sfM$. If we are given a sequence of ordinals $\alpha_s< \kappa_{\lh s}$ for each $s\in d$ satisfying the followng: For every $s\in d$, if $d\vDash s'\multimap s$ and $a=\bfe^d(s')$, then $\alpha_s \le j_a(\alpha_{s'})$.
	Then
	\begin{equation*} \textstyle
		\prod_{s\in d}\{k^s\mid \crit k^s>\alpha_s\}\in \hnu^{d,\beta}.
	\end{equation*}
\end{lemma}
\begin{proof}
    We prove it by induction on $(d,<_\bbN)$: That is, we prove for each $s\in d$,
	\begin{equation} \textstyle \label{Formula: Product of tails is large - Induction hyp}
		\prod_{t\le_\bbN s}\{k^t\mid \crit k^t> \alpha_{t}\}\in \hnu^{d,\beta}_s.
	\end{equation}
    Suppose that \eqref{Formula: Product of tails is large - Induction hyp} holds for $t<_\bbN s$. Then 
    \begin{align*} 
	& \textstyle 
        \prod_{t\le_\bbN s}\{k^t\mid \crit k^t> \alpha_{t}\}\in \hnu^{d,\beta}_s \\
        \iff & \textstyle \{\vec{k}\in D^{d,\beta}_{s-1}\mid 
        \{k^s\in \Emb^{k'}_{\beta(s)}\mid 
        \forall t\le_\bbN s [\crit\vec{k}_t> \alpha_t]\land \crit k^s> \alpha_s\} \in \mu^{k'}_{\beta(s)}
        \}\in \hnu^{d,\beta}_{s-1} \\
        \iff & \textstyle \{\vec{k}\in D^{d,\beta}_{s-1}\mid 
        \forall t\le_\bbN s [\crit\vec{k}_t> \alpha_t] \land 
        \{k^s\in \Emb^{k'}_{\beta(s)}\mid 
        \crit k^s> \alpha_s \} \in \mu^{k'}_{\beta(s)}
        \}\in \hnu^{d,\beta}_{s-1}
        \addtocounter{equation}{1}\tag{\theequation} \label{Formula: Co-small-00}
    \end{align*}
        Where $d\vDash s'\multimap s$, $a=\bfe^d(s')$, and
    \begin{equation*}
        k' = 
        \begin{cases}
        j_a(\vec{k}_{s'}), & \text{if } \lh^d s'\ge 2, \\
        j_1\restriction V_{\kappa_1+\beta(s)} & \text{if }\lh^d s'=1.
        \end{cases}
    \end{equation*}
        If $\lh^d(s')=1$, then $\{k^s\in \Emb^{j_1\restriction V_{\kappa_1+\beta(s)}}_{\beta(s)}\mid  \crit k^s>\alpha_s \} \in \mu^{j_1\restriction V_{\kappa_1+\beta(s)}}_{\beta(s)}$ holds since $\alpha_s<\kappa_1$.
    If $\lh^d(s')\ge 2$, then
    \begin{equation*}
        \{k^s\in \Emb^{j_a(\vec{k}_{s'})}_{\beta(s)}\mid 
        \crit k^s> \alpha_s \} \in \mu^{j_a(\vec{k}_{s'})}_{\beta(s)} \iff
        j_a(\crit\vec{k}_{s'})> j_a(\vec{k}_{s'})(\alpha_s),
    \end{equation*}
    and the latter inequality holds for $\hnu^{d,\beta}_{s-1}$-almost all $\vec{k}$ since we inductively assumed that $\forall t\le_\bbN s-1[\crit\vec{k}_t> \alpha_\tau]$ holds for $\hnu^{d,\beta}_{s-1}$-almost all $\vec{k}$, and 
    \begin{equation*}
        j_a(\crit\vec{k}_{s'})> j_a(\alpha_{s'}) = j_a(\vec{k}_{s'}(\alpha_{s'}))
        = j_a(\vec{k}_{s'})(j_a(\alpha_{s'})) \ge 
         j_a(\vec{k}_{s'})(\alpha_s),
    \end{equation*}
    where the first equality holds since $\alpha_{s'}<\crit\vec{k}_{s'}$.
\end{proof}

\begin{proposition}
    $\nu^{d,\beta}$ does not depend on the choice of $\beta$.
\end{proposition}
\begin{proof}
    First, we claim that if $\beta(s)\le\gamma(s)$ for every $s\in d$, then $X\in \nu^{d,\beta}\iff X\in \nu^{d,\gamma}$. By \autoref{Lemma: Product measure projection lemma},
    \begin{align*}
        X\in \nu^{d,\beta} &\iff
        \big\{\vec{k}\in D^{d,\beta}\mid \{(s,\crit\vec{k}_s)\mid s\in\term(d)\} \in X\big\}\in \hnu^{d,\beta} \\ &\iff
        (\pi^d_{\beta,\gamma})^{-1} \big\{\vec{k}\in D^{d,\beta}\mid \{(s,\crit\vec{k}_s)\mid s\in\term(d)\} \in X\big\}\in \hnu^{d,\gamma} \\ &\iff
        \big\{\vec{k}\in D^{d,\gamma}\mid \big\{\big(s,\crit \big(\pi^d_{\beta,\gamma}(\vec{k})\big)_s\big)\mid s\in\term(d)\big\} \in X\big\}\in \hnu^{d,\gamma}.
    \end{align*}
    Since $\crit\big(\pi^d_{\beta,\gamma}(\vec{k})\big)_s = \crit\vec{k}_s$, we have that the last formula is equivalent to $X\in \hnu^{d,\gamma}$.
    For a general case, let $p\colon \Ord\times\Ord\to\Ord$ be the order isomorphism, where $\Ord\times \Ord$ follows the lexicographic order.
    If we let $\alpha_s = p(\beta(s),\gamma(s))$, then it satisfies the condition of \autoref{Lemma: Co-cardinal-small set is measure-large}. Hence, the combination of \autoref{Lemma: Co-cardinal-small set is measure-large} and \autoref{Theorem: hnu concentrates correctly} implies there is an embedding $\delta\colon \Dec(d^\bullet)\to\Omega^1_\sfM$ such that $\delta(s)\ge \alpha_s$ for every $s\in d$.
    It is clear that $\alpha_s\ge \beta(s),\gamma(s)$, so we have an embedding $\delta$ pointwise dominating $\beta$, $\gamma$.
\end{proof}

Hence the choice of $\beta$ is irrelevant of $\nu^{d,\beta}$, so we drop $\beta$ and write $\nu^d$ instead of $\nu^{d,\beta}$.
\autoref{Theorem: hnu concentrates correctly} immediately implies
\begin{corollary}
    $\nu^d$ concentrates on $(\Omega^1_\sfM)^{\Dec(d)}$.
\end{corollary}
\begin{proof}
    Let $\beta\colon \Dec(d^\bullet)\to \Omega^1_\sfM$ be an embedding.
    By \autoref{Theorem: hnu concentrates correctly}, we have
    \begin{equation*}
        \{\vec{k}\in D^{d,\beta}\mid s\mapsto \crit \vec{k}_s \text{ is an embedding from $\Dec(d^\bullet)$ to $\Omega^1_\sfM$}\} \in \hnu^{d,\beta}.
    \end{equation*}
    Hence by the definition of $\nu^{d,\beta}$, we also have
    \begin{equation*}
        \{\gamma\in (\Omega^1_\sfM)^{\Dec(d)} \mid \text{$\gamma$ is a dilator embedding}\} \in \nu^{d,\beta} = \nu^d. \qedhere 
    \end{equation*}
\end{proof}

For a finite flower $F$ with no nullary terms, we can define $\nu^F$ with the help of $\nu^d$:
\begin{definition}
    Let $F$ be a finite flower with no nullary terms. We define $\nu^F$ over the set of embeddings from $F$ to $\Omega^1_\sfM$ as follows:
    \begin{equation*}
        X\in \nu^F \iff \{\gamma\in (\Omega^1_\sfM)^{\Dec(d)} \mid \gamma\circ h\in X\}\in\nu^d,
    \end{equation*}
    where $d$ is a trekkable dendrogram with the isomorphism $h\colon \Dec(d)\to F$.
\end{definition}

\subsection{The coherence of the measure family}
From the remaining part of the paper, we show that $\nu^d$ witnesses the measurability of $\Omega^1_\sfM$. We first verify the coherence of the measure family.

\begin{proposition} \label{Proposition: Coherence for simplest case}
    Let $d,d'$ be a finite trekkable dendrogram with no nullary terms such that $|d'|-|d|=1$, and there is $s<|d'|$ and a trekkable dendrogram morphism $h\colon d\to d'$ satisfying
    \begin{equation*}
        h(t) = 
        \begin{cases}
            t & \text{if }t <_\bbN s,\\
            t+1 & \text{if }t \ge_\bbN s.
        \end{cases}
    \end{equation*}
    For $\beta\colon \Dec((d')^\bullet)\to \Omega^1_\sfM$ and $X\subseteq D^{d,\beta}$, we have
    \begin{equation*}
        X\in \hnu^{d,\beta\circ h} \iff (h^*)^{-1}[X] \in \hnu^{d',\beta},
    \end{equation*}
    where $h^*$ is a map defined over the set of embeddings $\Dec((d')^\bullet)\to\Omega^1_\sfM$ by $h^*(\beta) = \beta\circ h$.
\end{proposition}
\begin{proof}
    Let $|d|=m$.
    Following the notation in the proof of \autoref{Theorem: Measure independence from an order}, \eqref{Formula: Decomposing a measure quantifier 2}, we have
    \begin{align*}
    \forall(\hnu^{d,\beta\circ h})\vec{k} [\vec{k}\in X] \iff \forall(\hmu^{d,\beta\circ h}_1) k^1 \forall(\hmu^{d,\beta\circ h}_2) k^2\cdots \forall(\hmu^{d,\beta\circ h}_{m-1}) k^{m-1} \big[ \{(t,k^t)\mid t<m\}\in X\big].
    \end{align*}
    Observe that in $d'$, $s$ is a terminal node. This means no other measure components $\hmu^{d,\beta}_t$ in defining $\hnu^{d',\beta}$ depends on the $s$th component. 
    Hence for $t\ge_\bbN s$, $\hmu^{d,\beta\circ h}_t = \hmu^{d',\beta}_{t+1}$. We also have $\hmu^{d,\beta\circ h}_t = \hmu^{d',\beta}_{t}$ for $t<_\bbN s$, so \autoref{Lemma: Eliminating unused measure quantifiers} implies
    \begin{align*}
    \forall(\hnu^{d,\beta\circ h})\vec{k} [\vec{k}\in X]\ & \iff \forall(\hmu^{d',\beta}_1) k^1 \forall(\hmu^{d',\beta}_2) k^2\cdots \forall(\hmu^{d',\beta}_{m}) k^m \big[ \{(h^{-1}(t),k^t)\mid t\le m, t\neq s\}\in X\big]. \\
    & \iff \forall(\hnu^{d',\beta}) \vec{k}\big[\vec{k}\in (h^*)^{-1}[X]\big]. \qedhere 
    \end{align*}
\end{proof}

We can derive the coherence by applying \autoref{Proposition: Coherence for simplest case} several times.
\begin{lemma} \label{Lemma: Trekkable morphism decomposition}
    Let $d$ and $d'$ be finite dendrograms, $|d|<|d'|$, and $f\colon d\to d'$ a trekkable dendrogram morphism. Then we can find a sequence of trekkable dendrograms $d_0,d_1,\cdots,d_m$ and $f_l\colon d_l\to d_{l+1}$ ($l<m$) such that each $f_l$ is trekkable, $d_0=d$, $d_m=d'$, $f=f_{m-1}\circ\cdots\circ f_0$, and $|d_{l+1}|-|d_l|=1$ for every $l<m$.
\end{lemma}
\begin{proof}
    Let $d'_0\subseteq d'$ be the range of $f$ and let $\{s_0,\cdots,s_{m-1}\}$ be the $<_\bbN$-increasing enumeration of $d'\setminus d_0'$.
    By the trekkability of $d'$, for each $l<m$, the set $d'_l = d'_0\cup \{s_0,\cdots,s_{l-1}\}$ is a subdendrogram of $d'$.
    Then let us find a trekkable dendrogram $d_l$, a trekkable dendrogram morphism $g_l\colon d'_l\to d_l$, and $f_l$ making the following diagram commutes:
    \begin{equation*}
        \begin{tikzcd}
            d'_0 & d'_1 & \cdots & d'_m \\
            d_0 & d_1 & \cdots & d_m 
            \arrow[from=1-1, to=1-2, "\subseteq"]
            \arrow[from=1-2, to=1-3, "\subseteq"]
            \arrow[from=1-3, to=1-4, "\subseteq"]
            \arrow[from=1-1, to=2-1, "h_0"]
            \arrow[from=1-2, to=2-2, "h_1"]
            \arrow[from=1-4, to=2-4, "h_m"]
            \arrow[from=2-1, to=2-2, "f_0", swap]
            \arrow[from=2-2, to=2-3, "f_1", swap]
            \arrow[from=2-3, to=2-4, "f_{m-1}", swap]
        \end{tikzcd}
    \end{equation*}
    We can see that $d_l$ and $f_l$ satisfy the desired properties.
\end{proof}

\begin{proposition} \label{Proposition: Coherence Dendrogram version}
    Let $d$ and $d'$ be finite dendrograms and $f\colon d\to d'$ a trekkable dendrogram morphism.
    For $\beta\colon \Dec((d')^\bullet)\to \Omega^1_\sfM$ and $X\subseteq D^{d,\beta}$, we have
    \begin{equation*}
        X\in \hnu^{d,\beta\circ f} \iff (f^*)^{-1}[X] \in \hnu^{d',\beta}.
    \end{equation*}
\end{proposition}
\begin{proof}
    We only consider the case $|d|<|d'|$.
    Let $d_0,\cdots, d_m$ and $f_l\colon d_l\to d_{l+1}$ ($l<m$) be the sequence of trekkable dendrograms and morphisms given by \autoref{Lemma: Trekkable morphism decomposition}.
    Then we have
    \begin{equation*}
        X\in \hnu^{d,\beta\circ f} \iff (f_0^*)^{-1}[X]\in \hnu^{d_1,\beta\circ f_{l-1}\circ\cdots \circ f_1} \iff \cdots  \iff (f^*)^{-1}[X]\in \hnu^{d',\beta}. \qedhere 
    \end{equation*}
\end{proof}

Hence we have
\begin{theorem}[Coherence]
    Let $F$, $F'$ be finite flowers with no nullary terms and $f\colon F\to F'$ be an embedding. For $X\subseteq (\Omega^1_\sfM)^F$, we have
    \begin{equation*}
        X\in \nu^F \iff (f^*)^{-1}[X] \in \nu^{F'}.
    \end{equation*}
\end{theorem}
\begin{proof}
    By replacing $F$ and $f$ if necessary, we may assume that $F\subseteq F'$ and $f$ is the inclusion map.
    Let $d'$ be a trekkable dendrogram with an isomorphism $h'\colon F'\cong \Dec(d')$, and consider a subdendrogam $\hat{d}\subseteq d'$ such that $h'[F] = \Dec(\hat{d})$.
    Then we can find a trekkable dendrogram $d$ with an isomorphism $g\colon d\to\hat{d}$ that is also $<_\bbN$-increasing. It is easy to see that $g\colon d\to d'$ is a trekkable dendrogram morphism, and we can find an isomorphism $h\colon F\to \Dec(d)$ making the following diagram commute:
    \begin{equation*}
        \begin{tikzcd}
            F' & \Dec(d') \\
            F & \Dec(d) 
            \arrow[from=1-1, to=1-2, "h'"]
            \arrow[from=2-1, to=1-1, "f"]
            \arrow[from=2-1, to=2-2, "h"']
            \arrow[from=2-2, to=1-2, "\Dec(g)"']
        \end{tikzcd}
    \end{equation*}
    Then by \autoref{Proposition: Coherence Dendrogram version},
    \begin{align*}
        X\in \nu^F\ &\iff \{\gamma\in (\Omega^1_\sfM)^{\Dec(d)} \mid \gamma\circ h \in X \} \in \nu^d \\ 
        &\iff \{\gamma\in (\Omega^1_\sfM)^{\Dec(d')} \mid (\gamma\circ \Dec(g))\circ h\in X \} \in \nu^{d'} \\ 
        &\iff \{\gamma\in (\Omega^1_\sfM)^{\Dec(d')} \mid (\gamma\circ h')\circ f\in X \} \in \nu^{d'}  \\ 
        &\iff \{\gamma\in (\Omega^1_\sfM)^{\Dec(d')} \mid \gamma\circ h'\in (f^*)^{-1}[X] \} \in \nu^{d'} \\
        &\iff  (f^*)^{-1}[X] \in \nu^{F'}. \qedhere 
    \end{align*}
\end{proof}

\subsection{The \texorpdfstring{$\omega_1$}{omega 1}-completeness of the measure family}
\label{Subsection: omega1 completeness of the measure family}

We finish this section by proving that the measure family we have constructed is $\omega_1$-complete.
The main idea of the proof is somewhat similar to that of \autoref{Lemma: A continuous family of finite linear orders}, but the argument is more complicated since we iterate measures along a tree and use a dependent product. As we did in the previous subsection, we handle the trekkable dendrogram version of the $\omega_1$-completeness first and transfer it into the flower version. \emph{We use the axiom of dependent choice in this subsection.}

Let $F$ be a countable flower with no nullary terms. Then its cell decomposition $\Cell(F)$ is a dendrogram. In particular, $\Cell(F)$ is locally well-founded and $\Dec^\bullet(\Cell(F))(n)$ is well-ordered for each $n$. This means the level-then-value order $<_\LV$ over $\Cell(F)$ is a well-order. Hence, we can re-label elements of $\Cell(F)$ into ordinals and form a trekkable dendrogram $C$ isomorphic to $\Cell(F)$.
By \autoref{Proposition: Dendrogram characterization of a flower}, the cell decomposition of $F$ (and also $C$) is a tree with the top node $0$.

Let $D$ be a countable trekkable dendrogram with no nullary terms, and $f_i\colon d_i\to D$ be an increasing dendrogram morphism for each $i<\omega$.
Let us also fix $\beta\colon \Dec(D^\bullet)\to \Omega^1_\sfM$, which will be a uniform bound for $\hnu^{d_i,\beta\circ f_i}$.
We want to find a sequence of elementary embeddings $\lag\tilde{k}_s\mid s\in D\rag$ such that for each $i<\omega$, $\{(s,\tilde{k}_{f_i(s)})\mid s\in d_i\}\in X_i$. We will find the desired sequence `cell-by-cell':
More precisely, from a given $\tilde{k}_s$, we will find $\tilde{k}_t$ for every immediate successor $t$ of $s$.
We need subsidiary notions for the proof: First, we need a `section' of a measure for a given sequence of elementary embeddings $\vec{k}$:

\begin{definition}
    Let $d$ be a trekkable dendrogram and $\beta\colon \Dec(d^\bullet)\to \Omega^1_\sfM$.
    A sequence $\vec{k}$ of elementary embeddings is \emph{$(d,\beta)$-coherent} if 
    \begin{enumerate}
        \item $\dom \vec{k}$ is a subdendrogram of $d$. That is, $\dom\vec{k}\subseteq d$ and is closed under immediate predecessors.
        \item $\dom \vec{k}_s = V_{\crit \vec{k}_s + \beta(s)}$ for every $s\in \dom\vec{k}$.
    \end{enumerate}
\end{definition}

\begin{definition}
    Let $\beta\colon \Dec(d^\bullet)\to \Omega^1_\sfM$ and $\vec{k}$ be a $(d,\beta)$-coherent sequence of elementary embeddings.
    We define $D^{d,\beta}_s[\vec{k}]$ and $\hnu^{d,\beta}_s[\vec{k}]$ similar to \autoref{Definition: Measure for elementary embeddings}, but with `skipping' the embeddings occurring in $\vec{k}$.
    More precisely, we define them as follows:
    \begin{enumerate}
        \item $D^{d,\beta}_{0}[\vec{k}]=\varnothing$ and $\hnu^{d,\beta}_{0}[\vec{k}]$ is the trivial measure.
        \item If $s \in \dom \vec{k}$, $D^{d,\beta}_s[\vec{k}] = D^{d,\beta}_{s-1}[\vec{k}]$ and $\hnu^{d,\beta}_s[\vec{k}] = \hnu^{d,\beta}_{s-1}[\vec{k}]$.
        
        \item If $\lh s = 1$ and $s\notin \dom\vec{k}$, define
        \begin{itemize}
            \item $D^{d,\beta}_s[\vec{k}] = \{\vec{k}'\cup \{(s, k'')\} \mid \vec{k}'\in D^{d,\beta}_{s-1}[\vec{k}] \land k''\in \Emb^{j_1\restriction V_{\kappa_1+\beta(s)}}_{\beta(s)}\}.$
            
            \item $X \in \hnu^{d,\beta}_s[\vec{k}] \iff \{\vec{k}'\in D^{d,\beta}_{s-1}[\vec{k}] \mid \{k''\in \Emb^{j_1\restriction V_{\kappa_1+\beta(s)}}_{\beta(s)}\mid  \vec{k}'\cup \{(s, k'')\}\in X\}\in \mu^{j_1\restriction V_{\kappa_1+\beta(s)}}_{\beta(s)}\}\in \hnu^{d,\beta}_{s-1}[\vec{k}]$.
        \end{itemize}
        \item If $s\notin \dom\vec{k}$, $d\vDash t\multimap s$, $\bfe(t)=a$. Define
        \begin{itemize}
            \item $D^{d,\beta}_s[\vec{k}] = \{\vec{k}'\cup \{(s, k'')\} \mid \vec{k}'\in D^{d,\beta}_{s-1}[\vec{k}] \land k''\in \Emb^{j_a((\vec{k}\cup\vec{k}')_t)}_{\beta(s)}\}.$
            
            \item $X \in \hnu^{d,\beta}_s[\vec{k}] \iff \{\vec{k}'\in D^{d,\beta}_{s-1}[\vec{k}] \mid \{k''\in \Emb^{j_a((\vec{k}            \cup\vec{k}')_t)}_{\beta(s)}\mid  \vec{k}'\cup \{ (s, k'')\}\in X\}\in \mu^{j_a((\vec{k}\cup\vec{k}')_t)}_{\beta(s)}\}\in \hnu^{d,\beta}_{s-1}[\vec{k}]$.
        \end{itemize}
    \end{enumerate}
    $D^{d,\beta}[\vec{k}]$, $\hnu^{d,\beta}[\vec{k}]$ are $D^{d,\beta}_s[\vec{k}]$ and $\hnu^{d,\beta}_s[\vec{k}]$ for the final element $s$ of $d$ respectively.
    
    For $X\subseteq D^{d,\beta}$ and a sequence $\vec{k}$, we define $X[\vec{k}] = \{\vec{k}'\in D^{d,\beta}[\vec{k}]\mid \vec{k}\cup\vec{k}'\in X\}$.
\end{definition}

Every measure we used in the definition of $\hnu^{d,\beta}_s[\vec{k}]$ is countably complete, so $\hnu^{d,\beta}_s[\vec{k}]$ is also countably complete.
We also define a generalization of \eqref{Formula: Measure decomposition}:
\begin{definition}
    Let $d$ be a finite trekkable dendrogram and $\beta\colon \Dec(d^\bullet)\to \Omega^1_\sfM$ an embedding, $\vec{k}$ a $(d,\beta)$-coherent sequence, and $d'\subseteq d$ a subdendrogram of $d$.
    Suppose that $\dom\vec{k}\subseteq d'$ and both $\dom \vec{k}$ and $d'$ are closed under nodes with the same immediate predecessor; That is, if $t\in d'$ and $d\vDash s\multimap t,t'$, then $t'\in d'$.
    
    For $X\in \hnu^{d,\beta}[\vec{k}]$, let us define $X\restriction^{d,\beta,\vec{k}} d'$ by
    \begin{equation*}
        \vec{k}' \in X\restriction^{d,\beta,\vec{k}} d' \iff \forall (\mu^{d,\beta}_{s_0}) k^{s_0} \cdots \forall (\mu^{d,\beta}_{s_{m-1}}) k^{s_{m-1}} \big(\vec{k}'\cup \{(s_i,k^i) \mid i<m\} \in X\big), 
    \end{equation*}
    where $\{s_i\mid i<m\}$ is the increasing enumeration of $d\setminus d'$ and $\mu^{d,\beta}_s$ is a unit measure occurring in the definition of $\hnu^{d,\beta}[\vec{k}]$:
    \begin{equation*}
        \mu^{d,\beta}_s = \mu^{d,\beta}_s(\vec{k}\cup\vec{k}', \lag k^{s_i}\mid s_i<s\rag) =
        \begin{cases}
            \mu^{j_1\restriction V_{\kappa_1+\beta(s)}}_{\beta(s)} & \lh s = 1, \\
            \mu^{j_a(k^{s'})}_{\beta(s)} & d\vDash s'\multimap s,\ \bfe^d(s')=a,\ s'\notin \dom \vec{k}\cup \vec{k}', \\
            \mu^{j_a((\vec{k}\cup\vec{k}')_{s'})}_{\beta(s)} & d\vDash s'\multimap s,\ \bfe^d(s')=a,\ s'\in \dom \vec{k}\cup \vec{k}'.
        \end{cases}
    \end{equation*}
    We also define a measure $\hnu^{d,\beta}[\vec{k}]\restriction d'$ over $D^{d,\beta}[\vec{k}]\restriction^{d,\beta,\vec{k}} d'$ by
    \begin{equation*}
        X \in \hnu^{d,\beta}[\vec{k}]\restriction d' \iff \forall (\mu^{d,\beta}_{t_0}) k^{t_0} \cdots \forall (\mu^{d,\beta}_{t_{p-1}}) k^{t_{p-1}} \big(\{(t_i,k^{t_i}) \mid i<p\} \in X\big), 
    \end{equation*}
    where $\{t_i\mid i<p\}$ is the increasing enumeration of $d'\setminus \dom\vec{k}$.
\end{definition}

\begin{lemma}
    Let $d$ be a finite trekkable dendrogram, $d'\subseteq d$ a subdendrogram, $\beta\colon \Dec(d^\bullet)\to \Omega^1_\sfM$ an embedding, 
    and $\vec{k}$ a $(d,\beta)$-coherent sequence of elementary embeddings such that $\dom\vec{k}\subseteq d'$ and both $d'$ and $\dom\vec{k}$ are closed under nodes with the same immediate predecessor. Then
    \begin{enumerate}
        \item For $X\subseteq D^{d,\beta}[\vec{k}]$, $X \in \hnu^{d,\beta}[\vec{k}] \iff X\restriction^{d,\beta,\vec{k}} d' \in \hnu^{d,\beta}[\vec{k}]\restriction d'$.
        
        \item For $X\subseteq D^{d,\beta}[\vec{k}]$, if $\vec{k}'$ is a $(d,\beta)$-coherent sequence such that $\vec{k}\cup\vec{k}'\in D^{d,\beta}$, then $\vec{k}'\in X\restriction^{d,\beta,\vec{k}} \dom \vec{k}' \iff X[\vec{k}'] \in\hnu^{d,\beta}[\vec{k}\cup\vec{k}']$.
    \end{enumerate}
\end{lemma}
\begin{proof}
    The main idea of the proof is that we can switch the order of measure quantifiers in the definition of $\hnu^{d,\beta}[\vec{k}]$ and $\hnu^{d,\beta}[\vec{k}]\restriction d'$ as long as the measure order is trekkable.
    More precisely, suppose that $d$ and $\hat{d}$ are trekkable dendrograms and $h\colon d\to\hat{d}$ is an isomorphism. Then we have
    \begin{itemize}
        \item For each $X\subseteq D^{d,\beta}[\vec{k}]$,
        \begin{equation*}
            X\in \hnu^{d,\beta}[\vec{k}] \iff \forall \big(\mu^{\hat{d},\beta\circ h}_{h(s_0)}\big) k^{s_0} \cdots \forall \big(\mu^{\hat{d},\beta\circ h}_{h(s_{m-1})}\big) k^{s_{m-1}} \big[ \big\{ \big(s_i,k^{s_i}\big)\mid i<m\big\} \in X\big],
        \end{equation*}
        where $\lag s_i\mid i<m\rag$ is the enumeration of $d\setminus \dom \vec{k}$ such that $\lag h(s_i)\mid i<m\rag$ is increasing.

        \item For $X\subseteq D^{d,\beta}[\vec{k}]$, if $\lag t_i\mid i<p\rag$ is an enumeration of $d\setminus d'$ such that $\lag h(t_i)\mid i<p\rag$ is increasing, we have
        \begin{equation*}
            \vec{k}' \in X\restriction^{d,\beta,\vec{k}} d'\iff  \forall \big(\mu^{\hat{d},\beta\circ h}_{h(t_0)}\big) k^{t_0} \cdots \forall \big(\mu^{\hat{d},\beta\circ h}_{h(t_{p-1})}\big) k^{t_{p-1}} \big[\vec{k}'\cup\{(t_i,k^{t_i}) \mid i<p\} \in X\big], 
        \end{equation*}

        \item For $X\subseteq D^{d,\beta}[\vec{k}]\restriction^{d,\beta,\vec{k}} d'$, if $\lag t_i\mid i<p\rag$ is an enumeration of $d'\setminus \dom \vec{k}$ such that $\lag h(t_i)\mid i<p\rag$ is increasing, we have
        \begin{equation*}
            X\in \hnu^{d,\beta}[\vec{k}]\restriction d' \iff \forall \big(\mu^{\hat{d},\beta\circ h}_{h(t_0)}\big) k^{t_0} \cdots \forall \big(\mu^{\hat{d},\beta\circ h}_{h(t_{p-1})}\big) k^{t_{p-1}} \big[\{(t_i,k^{t_i}) \mid i<p\} \in X\big], 
        \end{equation*}
    \end{itemize}
    Its proof follows from the proof of \autoref{Theorem: Measure independence from an order}, so we omit its details. Let us apply the previous observation to prove the lemma:
    \begin{enumerate}
        \item Let $\hat{d}$ be a dendrogram isomorphic to $d$, whose field is a natural number, but enumerates elements of $d'$ first, then enumerates those of $d\setminus d'$; That is, if $h\colon d\to \hat{d}$ is an isomorphism, $s\in d'$, $t\in d\setminus d'$, then $h(s)<_\bbN h(t)$.
        Such $\hat{d}$ exists and is trekkable by the assumption that $d'$ is closed under nodes with the same immediate predecessor.
        Suppose that $\lag h(s_i)\mid i<m\rag$ increasingly enumerates $d\setminus\dom\vec{k}$, and $\lag h(s_i)\mid i<p\rag$ increasingly enumerates $d'\setminus \dom\vec{k}$.
        Then for $X\subseteq D^{d,\beta}[\vec{k}]$,
        \begin{align*}
            X\in \hnu^{d,\beta}[\vec{k}]\ & \iff \forall \big(\mu^{\hat{d},\beta\circ h}_{h(s_0)}\big) k^{s_0} \cdots \forall \big(\mu^{\hat{d},\beta\circ h}_{h(s_{m-1})}\big) k^{s_{m-1}} \big[ \big\{ \big(s_i,k^{s_i}\big)\mid i<m\big\} \in X\big] \\
            & \iff \forall \big(\mu^{\hat{d},\beta\circ h}_{h(s_0)}\big) k^{s_0} \cdots \forall \big(\mu^{\hat{d},\beta\circ h}_{h(s_{p-1})}\big) k^{s_{p-1}} \big[ \big\{ \big(s_i,k^{s_i}\big)\mid i<p\big\} \in X\restriction^{d,\beta,\vec{k}} d'\big] \\
            &\iff X\restriction^{d,\beta,\vec{k}} d' \in \hnu^{d,\beta}[\vec{k}]\restriction d'.
        \end{align*}
        
        \item Let $d'=\dom \vec{k}'$, and $\hat{d}$ be a dendrogram isomorphic to $d$ whose field is a natural number but enumerates elements of $d'$ first, then enumerates elements of $d\setminus d'$. If $\lag h(s_i)\mid i<m\rag$ increasingly enumerates $d\setminus\dom\vec{k}$, and $\lag h(s_i)\mid i<p\rag$ increasingly enumerates $d'\setminus \dom\vec{k}$, then
        \begin{align*}
            \vec{k}'\in X\restriction^{d,\beta,\vec{k}}d'\ 
            & \iff \forall \big(\mu^{\hat{d},\beta\circ h}_{h(s_p)}\big) k^{s_p} \cdots \forall \big(\mu^{\hat{d},\beta\circ h}_{h(s_{m-1})}\big) k^{s_{m-1}} \big[ \vec{k}'\cup\big\{ \big(s_i,k^{s_i}\big)\mid p\le i<m\big\} \in X\big]
            \\& \iff \forall \big(\mu^{\hat{d},\beta\circ h}_{h(s_p)}\big) k^{s_p} \cdots \forall \big(\mu^{\hat{d},\beta\circ h}_{h(s_{m-1})}\big) k^{s_{m-1}} \big[ \big\{ \big(s_i,k^{s_i}\big)\mid p\le i<m\big\} \in X[\vec{k}']\big] \\
            &\iff X[\vec{k}'] \in \hnu^{d,\beta}[\vec{k}\cup\vec{k}'].\qedhere 
        \end{align*}
    \end{enumerate}
\end{proof}

The following theorem will immediately imply the $\omega_1$-completeness of the measure family.

\begin{theorem} \label{Theorem: Main step for omega 1 completeness}
    Let $D$ be a countable trekkable dendrogram with no nullary terms, $\beta\colon \Dec(D^\bullet)\to \Omega^1_\sfM$ a limit embedding, $\{d_i\mid i<\omega\}$ a sequence of finite trekkable dendrograms, and $f_i\colon d_i\to D$ an increasing dendrogram morphism such that $D=\bigcup_{i<\omega}\ran f_i$.
    If $X_i\in \hnu^{d_i,\beta\circ f_i}$ for each $i<\omega$, then we can find a family of elementary embeddings $\lag\tilde{k}_s\mid s\in D\rag$ such that $\{(s,\tilde{k}_{f_i(s)})\mid s \in d_i\}\in X_i$ for each $i<\omega$ and $s^\bullet\mapsto \crit \tilde{k}_s$ for $s\in D$ is an embedding from $\Dec(D^\bullet)$ to $\Omega^1_\sfM$.
\end{theorem}
\begin{proof}
    By \autoref{Theorem: hnu concentrates correctly}, we may assume that for every $\vec{k}\in X_i$, the map $s^\bullet\mapsto \crit\vec{k}_s$ is an embedding from $\Dec(d_i^\bullet)$ to $\Omega^1_\sfM$.
    Now define $D_0 = \{0\}$, $D_{\alpha+1} = D_\alpha\cup\{\xi\in D\mid \alpha\multimap \xi\}$, and $D_\alpha = \bigcup_{\xi<\alpha} D_\xi$ for a limit $\alpha$. Then each $D_\alpha$ is a subdendrogram of $D$ and $\alpha\in D_\alpha$ for every $\alpha\in D$. Clearly, each $D_\alpha$ is closed under nodes with the same immediate predecessor.
    We first find $\tilde{k}_s$ for $s\in D_1$ satisfying the following: For each $i<\omega$, $\{(s,\tilde{k}_{f_i(s)})\mid f_i(s)\in D_1 \}\in X_i\restriction^{d_i,\beta\circ f_i,\{\}} f_i^{-1}[D_1]$.
    
    Observe that $X_i\restriction^{d_i,\beta\circ f_i,\{\}} f_i^{-1}[D_1] \in \hnu^{d_i,\beta\circ f_i}\restriction f_i^{-1}[D_1]$ and $\hnu^{d_i,\beta\circ f_i}\restriction f_i^{-1}[D_i]$ is the product of measures of the form $\mu^{j_1\restriction V_{\kappa_1+\beta\circ f_i(s)}}_{\beta\circ f_i(s)}$ for $s\in d_i$ such that $f_i(s)\in D_1$.
    If we take $\gamma(0) = \sup \{\beta(s)\mid D\vDash 0\multimap s\}<\kappa_1$, then \autoref{Proposition: Product measure is generated by cubes} and Countable Choice imply for each $i<\omega$ we can find $Y^1_i\in \mu^{j_1\restriction V_{\kappa_1+\gamma(0)}}_{\gamma(0)}$ such that
    \begin{equation} \label{Formula: omega1 completeness 00} \textstyle 
        \Delta^{j_1\restriction V_{\kappa_1+\gamma(0)}}_{N^1_i}\cap \prod_{D\vDash 0\multimap f_i(s)}\pi^{j_1\restriction V_{\kappa_1+\gamma(0)}}_{\beta\circ f_i(s),\gamma(0)}[Y^1_i] \subseteq X_i\restriction^{d_i,\beta\circ f_i,\{\}} f_i^{-1}[D_1],
    \end{equation}
    where $N^0_i=|\{s\in d_i\mid D\vDash 0\multimap f_i(s)\}|$. We have $Y^1:=\bigcap_i Y^1_i \in \mu^{j_1\restriction V_{\kappa_1+\gamma(1)}}_{\gamma(1)}$,
    so by \autoref{Proposition: Arbitrary long Mitchell chain}, we can choose $\overline{k}_s\in Y^1$ for each $s\in D_1\setminus D_0$ such that $D\vDash s<s'$ implies
    \begin{equation*}
        \overline{k}_s\restriction V_{\crit\overline{k}_s+\beta(s)} \in \ran \left(\overline{k}_s\restriction V_{\crit\overline{k}_{s'}+\beta(s')}\right).
    \end{equation*}
    
    Then take $\tilde{k}_s = \overline{k}_s\restriction V_{\crit \overline{k}_s + \beta(s)}$. Combining with \eqref{Formula: omega1 completeness 00}, we have
    \begin{equation} \label{Formula: omega1 completeness 01}
        \{(s,\tilde{k}_{f_i(s)})\mid f_i(s)\in D_1\} \in X_i\restriction^{d_i,\beta\circ f_i,\{\}} f_i^{-1}[D_1]
    \end{equation}
    for each $i<\omega$. Note that if we write $\vec{\tilde{k}}^{\alpha,i} = \{(s,\tilde{k}_{f_i(s)})\mid f_i(s)\in D_\alpha\}$, then \eqref{Formula: omega1 completeness 01} becomes $\vec{\tilde{k}}^{1,i}\in X_i\restriction^{d_i,\beta\circ f_i,\{\}} \dom \vec{\tilde{k}}^{1,i}$, which is equivalent to $X_i[\vec{\tilde{k}}^{1,i}] \in \hnu^{d_i,\beta\circ f_i}[\vec{\tilde{k}}^{1,i}]$.

    Now, let us inductively assume that we have found $\tilde{k}_s$ for $s\in D_\alpha$ such that for every $i<\omega$, $X_i[\vec{\tilde{k}}^{\alpha,i}] \in \hnu^{d_i,\beta\circ f_i}[\vec{\tilde{k}}^{\alpha,i}]$.
    We will find $\tilde{k}_s$ for $s\in D$ such that $D\vDash\alpha\multimap s$ such that
    \begin{equation*}
        \big\{\big( s,\tilde{k}_{f_i(s)}\big) \mid s\in d_i\land  D\vDash\alpha\multimap  f_i(s)\big\} \in X_i[\vec{\tilde{k}}^{\alpha,i}]\restriction^{d_i,\beta\circ f_i, \vec{\tilde{k}}^{\alpha,i}} \big(f^{-1}_i[D_{\alpha+1}]\big)[\vec{\tilde{k}}^{\alpha,i}].
    \end{equation*}
    Note that $\alpha\in D_\alpha$, so $\tilde{k}_\alpha$ is defined. From the inductive assumption, we have
    \begin{equation}\label{Formula: omega1 completeness 02} 
        X_i[\vec{\tilde{k}}^{\alpha,i}] \restriction^{d_i,\beta\circ f_i,\vec{\tilde{k}}^{\alpha,i}} \big(f_i^{-1}[D_\alpha]\big)[\vec{\tilde{k}}^{\alpha,i}] \in \hnu^{d_i,\beta\circ f_i}[\vec{\tilde{k}}^{\alpha,i}]\restriction^{d_i,\beta\circ f_i,\vec{\tilde{k}}^{\alpha,i}} \big( f_i^{-1}[D_\alpha]\big)[\vec{\tilde{k}}^{\alpha,i}],
    \end{equation}
    The measure in \eqref{Formula: omega1 completeness 02} is a product of measures of the form $\mu^{j_a(\tilde{k}_\alpha)}_{\beta\circ f_i(s)}$ for $s\in d_i$ with $f_i(s)\in D_\alpha$, where $a=\bfe^D(\alpha)$. Now take $\gamma(\alpha) = \sup\{\beta(s)\mid D\vDash \alpha\multimap s\}$, then \autoref{Proposition: Product measure is generated by cubes} and Countable Choice imply for each $i<\omega$, we can find $Y^\alpha_i\in \mu^{j_a(\tilde{k}_\alpha)}_{\gamma(\alpha)}$ such that
    \begin{equation}\label{Formula: omega1 completeness 03} \textstyle
    \Delta^{j_a(\tilde{k}_\alpha)}_{N^{\alpha+1}_i}\cap \prod_{D\vDash \alpha\multimap f_i(s)}\pi^{j_a(\tilde{k}_\alpha)}_{\beta\circ f_i(s),\gamma(\alpha)}[Y^{\alpha+1}_i] \subseteq X_i[\vec{\tilde{k}}^{\alpha,i}]\restriction^{d_i,\beta\circ f_i,\vec{\tilde{k}}^{\alpha,i}} \big(f_i^{-1}[D_\alpha]\big)[\vec{\tilde{k}}^{\alpha,i}],
    \end{equation}
    where $N^{\alpha+1}_i=|\{s\in d_i\mid D\vDash \alpha\multimap f_i(s)\}|$. We have $Y^{\alpha+1}:=\bigcap_i Y^{\alpha+1}_i \in \mu^{j_a(\tilde{k}_\alpha)}_{\gamma(\alpha)}$,
    so by \autoref{Proposition: Arbitrary long Mitchell chain}, we can choose $\overline{k}_s\in Y^{\alpha+1}$ for each $s\in D_{\alpha+1}\setminus D_\alpha$ such that $D\vDash s<s'$ and $\alpha\multimap s,s'$ imply
    \begin{equation*}
        \overline{k}_s\restriction V_{\crit\overline{k}_s+\beta(s)} \in \ran \left(\overline{k}_s\restriction V_{\crit\overline{k}_{s'}+\beta(s')}\right).
    \end{equation*}
    Now take $\tilde{k}_s = \overline{k}_s\restriction V_{\crit \overline{k}_s + \beta(s)}$ for $s\in D_{\alpha+1}\setminus D_\alpha$ as before. 
    Combining with \eqref{Formula: omega1 completeness 03}, we have
    \begin{equation} \label{Formula: omega1 completeness 04}
        \{(s,\tilde{k}_{f_i(s)})\mid f_i(s)\in D_{\alpha+1}\setminus D_\alpha\} \in X_i[\vec{\tilde{k}}^{\alpha,i}]\restriction^{d_i,\beta\circ f_i,\vec{\tilde{k}}^{\alpha,i}} \big(f_i^{-1}[D_\alpha]\big)[\vec{\tilde{k}}^{\alpha,i}]
    \end{equation}
    Note that if we let $\vec{k}'$ be the left-hand-side of \eqref{Formula: omega1 completeness 04}, then \eqref{Formula: omega1 completeness 04} is equivalent to $X_i[\vec{\tilde{k}}^{\alpha,i}][\vec{k}']\in \hnu^{d_i,\beta\circ f_i}[\vec{\tilde{k}}^{\alpha,i}\cup \vec{k}']$, and $X_i[\vec{\tilde{k}}^{\alpha,i}][\vec{k}'] = X_i[\vec{\tilde{k}}^{\alpha,i}\cup\vec{k}']$.
    It shows the inductive hypothesis for $\alpha+1$.
    For limit $\alpha$, observe that for each $i<\omega$ we can find $\xi<\alpha$ such that $X_i[\vec{\tilde{k}}^{\alpha,i}]=X_i[\vec{\tilde{k}}^{\xi,i}]$ and $\hnu^{d_i,\beta\circ f_i}[\vec{\tilde{k}}^{\alpha,i}]=\hnu^{d_i,\beta\circ f_i}[\vec{\tilde{k}}^{\xi,i}]$ since each $d_i$ is finite.
    
    We finalize the proof by showing that $s^\bullet \mapsto \crit \tilde{k}_s$ is an embedding from $\Dec(D^\bullet)$ to $\Omega^1_\sfM$. Let $\alpha$ be the least ordinal such that $D=D_\alpha$. For $s,t\in D$, we can find $i<\omega$ such that $s,t\in \ran f_i$. For an arity diagram $\cyrDe$, we have
    \begin{equation*}
        D\vDash s<_\cyrDe t \implies d_i \vDash f_i^{-1}(s)<_\cyrDe f_i^{-1}(t) \implies \Omega^1_\sfM \vDash \crit\tilde{k}_s <_\cyrDe \crit \tilde{k}_t. 
        \qedhere 
    \end{equation*}
\end{proof}

\begin{theorem}[$\omega_1$-completeness]
    Let $D$ be a countable flower with no nullary terms and $\{d_i\mid i<\omega\}$ be a countable family of finite subflowers of $D$.
    If $X_i \in \nu^d$ for each $i<\omega$, then we can find an embedding $f\colon D\to \Omega^1_\sfM$ such that for each $i<\omega$, $f\restriction d_i\in X_i$.
\end{theorem}
\begin{proof}
    By replacing $d_i$ with a larger dilator and $X_i$ with its pullback if necessary, we may assume that $D=\bigcup_{i<\omega} d_i$. 
    Let $\hat{d}_i$ and $\hat{D}$ be trekkable dendrograms such that $\Dec(\hat{d}_i)\cong d_i$ and $\Dec(\hat{D})\cong D$.
    Let $h\colon \Dec(\hat{D})\to D$ be an isomorphism, and let $f_i\colon \hat{d}_i\to \hat{D}$ be an embedding such that $h\circ\Dec(f_i)$ is an isomorphism from $\Dec(\hat{d}_i)$ to $d_i$. 
    Let us also fix a limit embedding $\beta\colon \Dec(\hat{D}^\bullet)\to \Omega^1_\sfM$, then 
    \begin{equation*}
        \hat{X}_i:=\big\{\vec{k}\in D^{\hat{d}_i,\beta}\mid \{(h\circ\Dec(f_i)(s),\crit \vec{k}_s)\mid s\in \term(\hat{d}_i)\} \in X_i\big\} \in \hnu^{d_i,\beta}.
    \end{equation*}
    Hence by \autoref{Theorem: Main step for omega 1 completeness}, we can find $\lag\tilde{k}_s\mid s\in \hat{D}\rag$ such that 
    \begin{enumerate}
        \item $\{(s,\tilde{k}_{f_i(s)})\mid s\in \hat{d}_i\}\in \hat{X}_i$ for each $i<\omega$, and
        \item The map $s^\bullet\mapsto \crit\tilde{k}_s$ is an embedding from $\Dec(D^\bullet)$ to $\Omega^1_\sfM$, and
    \end{enumerate}
    Hence, the function $s^\bullet\mapsto \crit\tilde{k}_s$ restricted to (the isomorphic copy of) $\term(\hat{D})$ is a function witnessing the $\omega_1$-completeness.
\end{proof}

\section{Final remarks}

We finish this paper with the author's viewpoint and future research directions about the connection between large ptykes (i.e., ptykes with large cardinal properties), determinacy, and homogeneous Suslin representation of projective sets.

We work with an iterable cardinal in this paper due to its simplicity over a Woodin cardinal and a measurable cardinal above. However, the author expects every proof of $\bfPi^1_n$-determinacy from a large cardinal can be decomposed into a construction of a measurable $(n-1)$-ptyx and a proof of $\bfPi^1_n$-determinacy from a measurable $(n-1)$-ptyx.
Hence, the author conjectures we can also construct a measurable dilator from a Woodin cardinal and a measurable cardinal above, and a measurable $n$-ptyx from $n$ many Woodin cardinals and a measurable above.
However, Martin's measurable dilator from an iterable cardinal is expected to have a stronger property than an expected measurable dilator from a Woodin cardinal and a measurable above --- On the one hand, every measure associated with Martin's measurable dilator is $\kappa$-complete for an associated iterable cardinal $\kappa$. On the other hand, the author conjectures that if $\kappa$ is a measurable cardinal with a Woodin cardinal $\delta<\kappa$, then for each $\alpha<\delta$ the associated measurable dilator has a family of $\alpha$-complete measures.

It is well-known that the existence of a measurable cardinal is strictly stronger than $\bfPi^1_1$-determinacy. $\bfPi^1_1$-determinacy is equivalent to the assertion that every real has a sharp.
Similarly, it is reasonable to guess that the existence of a measurable dilator is strictly stronger than $\bfPi^1_2$-determinacy.
It is known by \cite[Corollary 2.2]{MullerSchindlerWoodin2020} that $\bfPi^1_{n+1}$-determinacy is equivalent to ``For every real $r$, $M_n^\sharp(r)$ exists and $\omega_1$-iterable.'' The author guesses that if we know how to construct a measurable $n$-ptykes from $n$ Woodin cardinals and a measurable above, then we should also be able to construct a half-measurable $n$-ptyx from the assertion ``For every real $r$, $M_n^\sharp(r)$ exists and $\omega_1$-iterable.''
Kechris \cite{KechrisUnpublishedDilators} stated without proof that Projective Determinacy is equivalent to `For every $n$, there is a half-measurable $n$-ptyx,' and he noted that the level-by-level equivalence should hold, but `it has not been proved yet.'
The author also conjectures that the existence of half-measurable $n$-ptyx is equivalent to $\bfPi^1_{n+1}$-determinacy.

We finish this paper with a possible connection with the homogeneous Suslin representation of projective sets:
As stated before, a homogeneous Suslin representation of a $\bfPi^1_1$-set into an `effective part' corresponding to a predilator $D$ and a measurable cardinal $\kappa$.
More precisely, we can decompose a homogeneous Suslin representation of a $\Pi^1_1[r]$-set into an `effective part' corresponding to an $r$-recursive predilator $D$ and a measurable cardinal $\kappa$.
The author expects that a homogeneous Suslin representation of a $\Pi^1_n[r]$ can be decomposed into an `effective part' given by an $(n+1)$-preptyx $P$ and a measurable $n$-ptyx $\Omega$, and so $P(\Omega)$ forms a homogeneous Suslin representation. It is interesting to ask if every homogeneous tree representation of a given definable set is decomposed into an `effective object,' which should be a generalization of ptykes and a large-cardinal-like object.

\printbibliography

\end{document}